\documentclass[11pt,reqno]{amsart}
\usepackage[T1]{fontenc}
\usepackage{graphicx}
\usepackage{cite}

\usepackage{color}
\definecolor{MyLinkColor}{rgb}{0,0,0.4}
\numberwithin{equation}{section}

\newcommand{\dv}{\mathop{\rm div}\nolimits}
\newcommand{\im}{\mathop{\rm Im}\nolimits}

\newcommand{\tr}{\mathop{\rm tr}\nolimits}

\newcommand{\re}{\mathop{\rm Re}\nolimits}

\newcommand{\0}{\Omega}
\newcommand{\e}{\varepsilon}

\newcommand{\p}{\partial}
\newcommand{\wt}{\widetilde}
\newcommand{\ov}{\overline}

\newcommand{\G}{\Gamma}

\newcommand{\A}{\mathcal{A}}
\newcommand{\bA}{\mathbb{A}}
\newcommand{\B}{\mathcal{B}}
\newcommand{\bB}{\mathbb{B}}

\newcommand{\kH}{\mathcal{H}}

\newcommand{\kL}{\mathcal{L}}

\newcommand{\T}{\mathcal{T}}

\newcommand{\V}{\mathcal{V}}

\newcommand{\C}{\mathbb{C}}

\newcommand{\R}{\mathbb{R}}
\newcommand{\s}{\mathbb S}
\newcommand{\N}{\mathbb{N}}
\newcommand{\Z}{\mathbb{Z}}
\newcommand{\X}{\mathbb{X}}
\newcommand{\Y}{\mathbb{Y}}

\DeclareMathOperator{\supp}{supp}

\newtheorem{thm}{Theorem}[section]
\newtheorem{prop}[thm]{Proposition}
\newtheorem{lemma}[thm]{Lemma}
\newtheorem{cor}[thm]{Corollary}
\newtheorem{rem}[thm]{Remark}
\newtheorem{rems}[thm]{Remarks}
\title[The domain of parabolicity for the Muskat problem]{The domain of parabolicity for the Muskat problem}

\author[J. Escher]{Joachim Escher}
\address{Institut f\" ur Angewandte Mathematik, Leibniz Universit\" at Hannover, 3016 Hannover, Deutschland}
\email{escher@ifam.uni-hannover.de}
\email{matioc@ifam.uni-hannover.de}
\email{walker@ifam.uni-hannover.de}

\author[B.--V. Matioc]{Bogdan--Vasile Matioc}

\author[Ch. Walker]{Christoph Walker}


\usepackage[colorlinks=true,linkcolor=MyLinkColor,citecolor=MyLinkColor]{hyperref} 

\begin{document}

\begin{abstract}
We address the well-posedness of the Muskat problem in a   periodic geometry and in a setting which allows us to consider general initial and boundary data,  gravity effects, as well as surface tension effects.
In the absence of surface tension we prove that the  Rayleigh-Taylor   condition identifies a domain of parabolicity for the Muskat problem. This property is used to establish the well-posedness of the problem.
In the presence of surface tension effects the Muskat problem is of parabolic type for general initial and boundary data. 
As a bi-product of our analysis we obtain that Dirichlet-Neumann type operators associated with certain diffraction problems are negative generators of strongly continuous and analytic semigroups in the scale of small H\"older spaces.
\end{abstract}

\maketitle

\tableofcontents


\section{Introduction}\label{Sec:1}
We study the evolution of two  vertically  superposed (or horizontally adjacent) immiscible  layers of Newtonian fluids with (possibly) different densities and viscosities in a two-dimensional periodic porous medium or Hele-Shaw cell when allowing for both gravity and surface tension effects.  
The pressure on  the fixed flat boundary of the lower layer is prescribed and the upper fluid layer is assumed to be bounded from above by air at uniform pressure. 
This leads to a moving boundary problem for the interface between the two layers, the  interface between the upper layer and the air, and the velocity potentials in the two fluid layers.
The associated mathematical model is the Muskat problem, which was originally proposed in \cite{Mu34} as a model for the encroachment of water into an oil sand.      
 It is given in \eqref{PB} below in the absence of surface tension effects and accordingly in \eqref{eq:SK} when allowing for surface tension effects.

Muskat problems have been studied extensively in the last two decades. 
Surface tension effects are included  in the  papers  \cite{FT03, HTY97, EM11a, EMM12a, SP16x} where questions related to well-posedness (for small initial data) and stability properties of trivial, that is, circular or flat, and finger-shaped equilibria are addressed
(see  \cite{EEM09c} for a classification of equilibria).  In this context arbitrary initial data are considered only in the very recent monograph \cite{SP15x} and in \cite{SP16x}.
 There is a much larger list of references dealing with the Muskat problem without surface tension effects, the methods used in the studies being numerous and quite different. 
  The well-posedness property  is established in \cite{Y96, Y03} by using Newton's iteration method; the references \cite{A04, CCG11}  use energy estimates 
  (see also \cite{BCG14,  CG07, CG10, CGO14} for the case of fluids with equal viscosities); \cite{EM11a, EMM12a} rely on abstract parabolic theory and  the continuous maximal regularity due to Da Prato and Grisvard \cite{DG79};
  \cite{SCH04} employs -- in the absence of gravity effects -- methods from complex analysis and a version of the Cauchy-Kowalewski theorem. 
  Existence of solutions for nonregular initial data is shown in \cite{BV14} by means of a fixed point argument.
  There are various interesting phenomena established  for fluids with equal viscosities: global existence of strong and weak solutions for initial data which are bounded by explicit constants  \cite{CCGS13, GB14},
  existence of initial data for which solutions turn over \cite{CCFGL11, CCFGL12, CCFG13}, or the absence of squirt or splash singularities \cite{CG10, FF14, GS14}.
  
 An important role in the study of the Muskat problem is played by the Rayleigh-Taylor condition, which is a sign restriction on the jump of the gradient of the pressure in normal direction along
 an interface that separates two phases (see \eqref{CT1} below for more details) and was originally found within the linear theory \cite{ST58}.
  In the absence of surface tension effects, the paper \cite{EM11a} was (one of) the first in which it was proved that the Muskat problem has, at least for small initial and boundary data, a parabolic character provided the Rayleigh-Taylor condition holds.
For general initial data the well-posedness of the problem in different geometries is also implied 
by the Rayleigh-Taylor condition  \cite{A04, CCG11, SCH04, Y96, Y03}. However, it is worth mentioning that the true character of the problem  
was not revealed in any of the just cited papers.

 In this paper we now prove for arbitrary (sufficiently smooth) initial data that the Muskat problem with and without surface tension effects has a parabolic character. More precisely, 
 when neglecting surface tension we establish the parabolicity of the problem 
provided the Rayleigh-Taylor condition  holds. 
  This enables us to use Da Prato and Grisvard's abstract parabolic theory, in particular continuous maximal regularity, in order to prove the well-posedness of this problem, cf. Theorem~\ref{MT1}.
  Having two moving interfaces, we actually need to impose the Rayleigh-Taylor at each of them.
 We also show that the Muskat problem with surface tension is parabolic for  arbitrary (sufficiently smooth) initial and boundary data, the corresponding well-posedness result being stated in  Theorem \ref{MT2}.
  As a bi-product of our analysis we show in Proposition \ref{L:DN1} and Remark~\ref{R:1K} that Dirichlet-Neumann type operators associated with certain diffraction problems are negative generators 
  of strongly continuous and analytic semigroups in the scale of small H\"older spaces.

 It is worth to  emphasize that the abstract parabolic setting mentioned above appears to be one of the few 
  where the Muskat problem can be handled both with or without surface tension effects and with general boundary data. In addition, our analysis allows us to handle two fluids with possibly different viscosities or densities or even to neglect the effects of gravity (the latter being reasonable, for example, when the Hele-Shaw cell is not vertical, but horizontal or also in microfluidic models). In many studies these aspects were not taken into account.\\

The paper is organized as follows. In Section \ref{Sec:2} we present the Muskat problem without surface tension and the first main result Theorem \ref{MT1}, whose proof requires some preparation.
In Section~\ref{Sec:3} we first  discuss the solvability of a general diffraction problem and recast the Muskat problem as a fully nonlinear and nonlocal evolution equation.
We then show in Sections \ref{Sec:4}-\ref{Sec:6} that the Fr\'echet derivative of the operator associated with this evolution problem is an analytic generator.
For this we use localization techniques in the spirit of \cite{ES95}, but such that we keep the setting of periodic functions. The proof of Theorem \ref{MT1} is then a consequence of this generator result.
The well-posedness of the Muskat problem with surface tension effects is addressed in Section \ref{Sec:7}.

\section{The   Muskat problem without surface tension effects}\label{Sec:2}

 To set the stage we need some notation. In what follows $\s$ denotes the unit circle $\R/(2\pi\Z)$ meaning that functions depending on  $x\in\s$ are $2\pi$-periodic with respect to the real variable $x$.
Given $m\in\N$ and $\beta\in(0,1)$, the small H\"older space $h^{m+\beta}(\s)$ stands for the closure of the smooth functions $C^{\infty}(\s)$ in $C^{m+\beta}(\s).$
It is well-known that $h^{m+\beta}(\s)$ is a true subspace of the classical H\"older space $C^{m+\beta}(\s),$ cf. e.g. \cite{EMM12a}, and that $C^r(\s)$ is densely embedded in $h^s(\s)$ if $r>s>0$. 
Recall  that $h^{m}(\s)=C^m(\s)$ for $m\in\N$.
Similarly, given two functions $\phi,\psi\in C(\s)$ with $\phi(x)<\psi(x)$ for all $x\in\s$ and setting 
\begin{equation}\label{Def1}
\Omega:=\Omega(\phi,\psi):=\{(x,y)\,:\, \text{$x\in\s$ and $\phi(x)<y<\psi(x)$}\},
\end{equation}
 we denote by
$h^{m+\beta}(\0)$ the closure of the smooth functions 
$C^{\infty}(\ov \0)$ in $C^{m+\beta}(\ov \0)$.
As before,   $C^r(\ov \0)$ is densely embedded in  $h^s(\0)$ if $r>s$ provided that $\phi,\psi\in h^s(\s).$
 
Let $\alpha\in(0,1)$ and $d<0$ be fixed constants and set
\[\V:=\{(f,h)\in (h^{2+\alpha}(\s))^2\,:\,d<f<h\}.\]
For each pair  $(f,h)\in \V$  we define
$\Omega(f):=\Omega(d,f)$ and $\Omega(f,h)$ according to \eqref{Def1}.
We look for a pair of functions $(f,h):[0,T_0)\to \V$ with $T_0>0$ describing the evolution of the interfaces $\Gamma(f):=[y=f]$ and $\Gamma(h):=[y=h]$
that bound two  incompressible and immiscible Newtonian fluid layers  in a porous medium and at constant temperature.
At each time instant $t\in[0,T_0),$ the domain $\0(f(t))$ is assumed to be occupied by a fluid with density $\rho_-$ and viscosity $\mu_-$, respectively $\Omega(f(t),h(t))$  is the domain occupied by a second  
fluid with density $\rho_+$ and viscosity $\mu_+.$ Note that neither the densities nor the viscosities need to be equal in what follows.
We define the velocity potentials 
\begin{equation}\label{VP}
 u_\pm:=p_\pm +g\rho_\pm y,
\end{equation}
with $g$ being  the Earth's gravity  and $p_\pm$ the fluids' pressures  with densities $\rho_\pm$.  Our results 
hold also true when neglecting gravity, that is, when $g=0$, so we assume  $g\geq0$ in the following. 
The velocity  fields $\vec v_\pm$ then obey Darcy's law, see \cite{Mu34},
\[
\vec v_-=-\frac{k}{\mu_-}\nabla u_-\quad\text{in $\0(f)$}\qquad\ \text{and} \qquad \ \vec v_+=-\frac{k}{\mu_+}\nabla u_+ \quad\text{in $\0(f,h)$,}
\]
and the incompressibility condition reads 
\[
\dv v_+=0\quad\text{in $\0(f)$}\qquad\ \text{and} \qquad \ \dv v_-=0 \quad\text{in $\0(f,h)$.}
\]
Here,  $k>0$ is a constant which stands for  the permeability of the porous medium.
Assuming the pressure on the boundary component  $\Gamma_d:=[y=d]$ to be known, in the absence of surface tension effects
the Muskat problem is the 
 system of partial differential equations
\begin{subequations}\label{PB}
\begin{equation}\label{eq:S}
\left\{\begin{array}{rllllll}
\Delta u_+\!\!&=&\!\!0&\text{in}& \Omega(f,h), \\
\Delta u_-\!\!&=&\!\!0&\text{in}& \Omega(f), \\
{\p_th}\!\!&=&\!\!-k\mu_+^{-1}\sqrt{1+h'^2}\p_\nu u_+&\text{on}& \Gamma(h),\\
u_+\!\!&=&\!\!g\rho_+h&\text{on}&\Gamma(h),\\
u_-\!\!&=&\!\!b&\text{on}&\G_d,\\
u_+-u_-\!\!&=&\!\!g(\rho_+-\rho_-)f&\text{on}& \Gamma(f),\\
{\p_tf}\!\!&=&\!\!-k{\mu_\pm^{-1}}\sqrt{1+f'^2}\p_\nu u_\pm &\text{on}& \Gamma(f),
\end{array}
\right.
\end{equation} 
governing the evolution of the fluids supplemented with the   initial conditions
\begin{equation}\label{eq:S1}
f(0)=f_0, \qquad h(0)=h_0.
\end{equation} 
\end{subequations}
We have additionally taken the pressure of the air to be constant  zero,  we   assumed that the interfaces between the fluids move along with the fluids (in particular, that ${\mu_-^{-1}}\p_\nu u_-={\mu_+^{-1}}\p_\nu u_+$ on $\Gamma(f)$), and that the pressure is continuous along the interfaces.
Given $\phi\in C^1(\s),$ we have chosen $\nu:=(-\phi',1)/\sqrt{1+\phi'^2}$ to be the unit outward normal vector at  the curve $[y=\phi]$.

The function  $b=b(t,x)$ corresponds to the given pressure at the interface $[y=d]$ and is assumed to belong to the class
\begin{equation}\label{CL}
 b\in C([0,T), h^{2+\alpha}(\s)),
\end{equation}
for some $T\in (0,\infty]$.
Our main goal is to study the existence and uniqueness of {\em classical H\"older solutions} to the Muskat problem \eqref{PB}, that is, of tuples  $(f,h,u_+,u_-)$
with 
\begin{equation}\label{CS}
 \begin{aligned}
 & (f,h)\in C([0,T_0),\V)\cap C^1\big([0,T_0), (h^{1+\alpha}(\s))^2\big)\,,\\
  &u_+(t)\in h^{2+\alpha}(\0(f(t),h(t))), \quad u_-(t)\in h^{2+\alpha}(\0(f(t)))
 \end{aligned}
\end{equation}
for all $t\in[0,T_0)$ with $T_0\in(0,T]$, and which satisfy the equations of \eqref{PB} pointwise. \bigskip

 Given $(f_0,h_0)\in\V$ and  $b_0:=b(0)\in h^{2+\alpha}(\s),$ Theorem \ref{DP} and  Remark \ref{Obs1} below ensure  that  the diffraction problem
\begin{equation}\label{eq:SSS}
\left\{\begin{array}{rllllll}
\Delta u_+^0\!\!&=&\!\!0&\text{in}& \Omega(f_0,h_0), \\
\Delta u_-^0\!\!&=&\!\!0&\text{in}& \Omega(f_0), \\
u_+^0\!\!&=&\!\!g\rho_+h_0&\text{on}&\Gamma(h_0),\\
u_-^0\!\!&=&\!\!b_{ 0}&\text{on}&\G_d,\\
u_+^0-u_-^0\!\!&=&\!\!g(\rho_+-\rho_-)f_0&\text{on}& \Gamma(f_0),\\
 {\mu_-}\p_\nu u_+^0&=&{\mu_+}\p_\nu u_-^0 &\text{on}& \Gamma(f_0),
\end{array}
\right.
\end{equation} 
possesses a unique solution  $(u_+^0,u_-^0)\in  h^ {2+\alpha}(\0(f_0,h_0)) \times h^ {2+\alpha}(\0(f_0))$. Letting $p_+^0$ and $p_-^0$ be the initial pressures determined, respectively, by $u_+^0$ and $u_-^0$ according to \eqref{VP}, 
we shall show that the condition
  \begin{equation}\label{CT1}
  \begin{aligned}
  & \p_\nu p_-^0-\p_\nu p_+^0 <0\qquad\text{on $\Gamma(f_0)$},\\
   &\p_\nu p_+^0 <0\qquad\text{on $\Gamma(h_0)$}
  \end{aligned} 
  \end{equation}
defines a regime (for $(f_0,h_0,b_0)\in\V\times  h^{2+\alpha}(\s)$)  where the Muskat problem \eqref{PB} is parabolic.
Since the air pressure is constant, the condition \eqref{CT1} expresses the Rayleigh-Taylor condition imposed at each interface as mentioned in the Introduction. 
To be more precise, we shall prove in Section \ref{Sec:3} that  the Muskat problem can be recast as a fully nonlinear abstract  evolution equation for the interfaces $f$ and $h$ only, that is, 
\[
\p_t(f,h)=\Phi(t, (f,h)),
\]
 which is of parabolic type when  \eqref{CT1} holds. 
By parabolicity we mean that the Fr\'echet derivative  $\p_{(f,h)}\Phi(0, (f_0,h_0))$ is the generator of a  strongly continuous and analytic semigroup.
This property is the  corner stone in our analysis and, together with the abstract parabolic theory due to Da Prato and Grisvard \cite{DG79, L95}, it enables us to establish
the following well-posedness result for the Muskat problem without  surface tension effects.

 \begin{thm}\label{MT1}
  Let $g\geq0,$ $(f_0,h_0)\in\V$, and  $b $ be given such that \eqref{CL} holds. Assume that {the Rayleigh-Taylor condition} \eqref{CT1} is satisfied.
  
  Then, there exist a maximal existence time $T_0:=T_0(f_0,h_0)\in(0,T]$ and a unique classical  H\"older  solution $(f,h,u_+,u_-)$ to \eqref{PB} on~$[0,T_0).$ 
Additionally, the solutions depend continuously on the initial data.
 \end{thm}

Let us remark that the Rayleigh-Taylor condition \eqref{CT1} plays a crucial role for the proof of Theorem~\ref{MT1}, in particular in the analysis presented in Sections~\ref{Sec:5} and \ref{Sec:6}. We now discuss the Rayleigh-Taylor condition in some special physical settings.

\begin{rems}\label{Rem:a} 
  {(a)\, } Our analysis discloses that the Muskat problem is backward parabolic when the Rayleigh-Taylor condition holds with reversed inequalities.

{(b)\, }	When the fluids have the same viscosities, the last two equations of \eqref{eq:SSS} show that the Rayleigh-Taylor condition on $\Gamma(f_0)$, i.e., the first condition in \eqref{CT1}, is equivalent to $\rho_->\rho_+$ (this is the
  case in \cite{BCG14,  CG07, CG10, CGO14}).
  Hence, in this case the Muskat problem \eqref{PB} is well-posed provided that
  \begin{equation*}
  \begin{aligned}
  &\rho_->\rho_+\qquad \text{and}\qquad \p_\nu p_+^0<0\quad\text{on $\Gamma(h_0)$.}
  \end{aligned}
  \end{equation*}
	
{(c)\, } In the particular case $d=-1, $ $ b(0)\equiv c\in\R,$ and $  (f_0,h_0)\equiv  (0,1),$ the condition~\eqref{CT1} coincides with that found in \cite[eq. (2.2)-(2.3)]{EMM12a}, that is,
\begin{equation*}
g\rho_+ >-\frac{c\mu_+}{\mu_-}\qquad \text{and} \qquad\frac{\mu_+-\mu_-}{\mu_++\mu_-}(c-g\rho_+)+g(\rho_+-\rho_-)<0.
\end{equation*}
This shows in particular that the set of data $(f_0,h_0,b(0))$ for which the Rayleigh-Taylor condition~\eqref{CT1} is satisfied is not empty.\medskip

  {(d)\, } If gravity is neglected, that is, if $g=0$, the Rayleigh-Taylor condition \eqref{CT1} is equivalent to
   \begin{equation}\label{CT1b}
  \begin{aligned}
  & (\mu_--\mu_+)\p_\nu p_+^0 <0\qquad\text{on $\Gamma(f_0)$},\\
   &\p_\nu p_+^0 <0\qquad\text{on $\Gamma(h_0).$}
  \end{aligned}
  \end{equation}
 \end{rems}

As for Remarks~\ref{Rem:a} (d) we point out that if $b_0$  is zero or a negative function, then \eqref{CT1b} cannot be satisfied. Indeed, if $b_0$ is the zero function, then both $p_+^0$ and $p_-^0$ are  identically zero. If $b_0$ is a negative function, then $p_-^0$ is also negative since otherwise there exists $x_0\in\s$ such that $p_-^0(x_0,f(x_0))=\max_{\ov\0_-}p_-^0\geq0,$ and so $\p_\nu p_-^0>0$ at $(x_0,f(x_0))$ by Hopf's lemma.
  But $p_+^0$ is harmonic as well and not constant, therefore $p_+^0(x_0,f(x_0))=\max_{\ov\0_+}p_+^0\geq0,$ so that $\p_\nu p_+^0<0$ at $(x_0,f(x_0))$
  in contradiction to the last equation of \eqref{eq:SSS}.
  Hence, $p_-^0$ is negative implying $p_+^0<0$ on $\G(f_0).$ By Hopf's lemma we find $\p_\nu p_+^0 >0$ on $\Gamma(h_0),$
  and \eqref{CT1b} is again not satisfied.\medskip

  Lastly, if $b_0$ is a positive function,  the previous arguments show that \eqref{CT1b} is equivalent to
   \begin{equation*}
  \begin{aligned}
  & (\mu_--\mu_+)\p_\nu p_+^0 <0\qquad\text{on $\Gamma(f_0)$}.
  \end{aligned}
  \end{equation*}
  As $p_+^0$ attains its positive maximum on $\G(f_0)$, we have that $\p_\nu p_0^+$ is negative at least at one point on this interface implying that
  \begin{equation}\label{CT1b2}
  \begin{aligned}
  \mu_+<\mu_-.
  \end{aligned}
  \end{equation}
  It is easy to see that if $b_0$ is positive and  constant and if also $f_0$ and $h_0$ are constant functions, then \eqref{CT1b2} is equivalent to \eqref{CT1}.
  Thus, if $b_0$ is positive and constant we find, as in \cite{SCH04}, that the Muskat problem is well-posed for small initial data -- that is, initial data close to constants in $\big(h^{2+\alpha}(\s)\big)^2 $ -- when the more viscous
  fluid expands into the less viscous one.\\

The proof of Theorem \ref{MT1} is postponed to the end of Section~\ref{Sec:6} as it requires several preparatory results that will be given in the subsequent sections.


\section{The evolution  equation}\label{Sec:3}
In order to solve problem \eqref{PB} we re-write it as an abstract evolution equation on the unit circle.
To do so we first transform system \eqref{eq:S} into a system of equations on fixed domains by using the unknown functions $(f,h).$
Let $\0_-:=\s\times (-1,0)$, $\0_+:=\s\times(0,1),$ and define for each $(f, h)\in\V$  the mappings
$ \phi_{ f}:\Omega_{-}\to\Omega(f)$ and $
\phi_{ (f,h)}:\Omega_{+}\to\Omega(f,h)$
by setting
\[
\phi_{f}(x,y):=(x,-dy+(1+y)f(x)) 
\qquad\ \text{and} \qquad \phi_{(f,h)}(x,y):=(x, y h(x)+(1-y)f(x)),
\]
respectively.
One easily checks that $\phi_f$ and $\phi_{(f,h)}$ are diffeomorphisms for all $(f, h)\in\V.$  
Each pair  $(f, h)\in\V$ induces linear strongly uniformly elliptic operators
\begin{align*}
&\A(f):\mbox{\it{h}}\,^{2+\alpha}(\Omega_-)\to \mbox{\it{h}}\,^{\alpha}(\Omega_-),\qquad v_-\mapsto\Delta(v_-\circ\phi_{f}^{-1})\circ\phi_{f},\\[1ex]
&\A(f,h):\mbox{\it{h}}\,^{2+\alpha}(\Omega_+)\to \mbox{\it{h}}\,^{\alpha}(\Omega_+),\qquad  v_+\mapsto\Delta(v_+\circ\phi_{(f,h)}^{-1})\circ\phi_{(f,h)},
\end{align*}
which depend, as bounded operators, real-analytically on $f$ and $h$ (see the formulae in the Appendix).
Denote by $\tr_0$ the trace operator with respect to $\Gamma_0:=\s\times\{0\}$.
We associate with problem \eqref{eq:S}  trace operators  on $\Gamma_0$, 
\begin{align*}
&\B(f)v_-:=k\mu_-^{-1}\tr_0 (\langle \nabla(v_-\circ\phi_{f}^{-1})|(-f',1)\rangle\circ\phi_{f}),\quad v_-\in \mbox{\it{h}}\,^{2+\alpha}(\Omega_-),\\[1ex]
&\B(f,h)v_+:=k\mu_+^{-1}\tr_0 (\langle \nabla(v_+\circ\phi_{(f,h)}^{-1})|(-f',1)\rangle\circ\phi_{(f,h)}),\quad v_+\in \mbox{\it{h}}\,^{2+\alpha}(\Omega_+),
\end{align*}
which, seen as bounded operators into $h^{1+\alpha}(\s)$,  depend real-analytically on $f$ and $h$ as well.
Lastly, we define  a boundary operator on $\Gamma_1$, where $\Gamma_{\pm1}:=\s\times\{\pm1\}$.
Given $(f,h)\in\V,$ we set
\[
\B_1(f,h)v_+:=k\mu_+^{-1}\tr_1 (\langle \nabla(v_+\circ\phi_{(f,h)}^{-1})|(-h',1)\rangle\circ \phi_{(f,h)}),\quad v_+\in \mbox{\it{h}}\,^{2+\alpha}(\Omega_+),
\]
where  $\tr_{\pm1}$ is the trace operator with respect to $\Gamma_{\pm1}$.

\begin{rem}\label{Obs1} Given $(f,h)\in\V,$ the mappings
\begin{align*}
 &[u_-\mapsto u_-\circ\phi_f]\, : \,h^{2+\alpha}(\0(f)) \to h^{2+\alpha}(\0_-),\\
 &[ u_+\mapsto u_+\circ\phi_{(f,h)}]\,:\,h^{2+\alpha}(\0(f,h))\to  h^{2+\alpha}(\0_+),
\end{align*}
are isomorphisms.
\end{rem}
\begin{proof}
 See, for instance, the proof of \cite[Lemma 1.2]{EM09}.
\end{proof}

In view of Remark \ref{Obs1}, it follows that 
 $(f,h,u_+,u_-)$ is a solution to \eqref{PB} if and only if $(f,h, v_+,v_-)$ with $v_+:=u_+\circ\phi_{(f,h)}$ and  $v_-:=u_-\circ\phi_f$
is a classical H\"older solution  to  
\begin{subequations}\label{eq:TS}
\begin{equation}\label{eq:TS1}
\left\{
\begin{array}{rllllll}
\A(f,h) v_+\!\!&=&\!\!0&\text{in $ \Omega_+$},\\
\A(f) v_-\!\!&=&\!\!0&\text{in $ \Omega_-$},\\
\B(f,h)v_+-\B(f)v_-\!\!&=&\!\!0 &\text{on $ \Gamma_{0}$},\\
v_+-v_-\!\!&=&\!\!g(\rho_+-\rho_-) f&\text{on $ \Gamma_0$},\\
v_+\!\!&=&\!\!g\rho_+h&\text{on $ \Gamma_1$},\\
 v_-\!\!&=&\!\!b&\text{on $\G_{-1}$},
\end{array}
\right.
\end{equation}
with 
\begin{equation}\label{eq:TS2}
\left\{
\begin{array}{rllllll}
\p_t h\!\!&=&\!\!-\B_1(f,h)v_+&\text{on $ \Gamma_1$},\\
\p_t f\!\!&=&\!\!-\B(f)v_- &\text{on $ \Gamma_{0}$},
\end{array}
\right.
\end{equation}
and 
\begin{equation}\label{eq:TS3}
f(0)=f_0,\qquad h(0)=h_0.
\end{equation}
\end{subequations}
The notion of  classical H\"older solution to \eqref{eq:TS} is defined analogously to that for problem \eqref{PB}.\bigskip

\noindent{\bf A diffraction problem in H\"older spaces.}\,
The system \eqref{eq:TS1} is an elliptic diffraction (or transmission)  problem, problems of this type being highly relevant in  many physical  situations such as the study of  multiphase dynamics.
However,  citable references on this topic are sparse.
The main goal in this part is to establish the following result on the existence, uniqueness, and real-analytic dependence of solutions to \eqref{eq:TS1} on given $(f,h)\in\V$ and $b\in h^{2+\alpha}(\s)$.

\begin{thm}\label{DP}
 Given $(f,h)\in \V $ and $b\in h^{2+\alpha}(\s)$, there exists a unique solution $$(v_+,v_-):=(v_+(f,h,b), v_-(f,h,b))\in \mbox{\it{h}}\,^{2+\alpha}(\Omega_+)\times \mbox{\it{h}}\,^{2+\alpha}(\Omega_-)$$
 to the diffraction problem \eqref{eq:TS1}.
 Moreover, it holds that
 \[\big[(f,h,b)\mapsto (v_+(f,h,b), v_-(f,h,b))\big]\in C^\omega\big(\V\times h^{2+\alpha}(\s), \mbox{\it{h}}\,^{2+\alpha}(\Omega_+)\times \mbox{\it{h}}\,^{2+\alpha}(\Omega_-)\big).\]
\end{thm}

In the H\"older setting  considered herein, problem \eqref{eq:TS1} can be accessed by using the celebrated Agmon-Douglis-Nirenberg  estimates  on solutions to elliptic boundary value systems presented in \cite{ADN64}.
To prove Theorem \ref{DP} we first consider a particular boundary value problem  for a  linear elliptic system  with coupled boundary conditions for which we establish the existence and uniqueness of solutions in the natural framework of H\"older spaces. 
This is the context of the next proposition.

In the following  $\p_1:=\p_x$, $\p_2:=\p_y$ and we identify both boundaries of $\0_+$ with the unit circle $\s.$ 

\begin{prop}\label{P:1}
Let $\kL_k:=a_{ij}^{(k)}\p_{ij}+b_{i}^{(k)}\p_i+c^{(k)}$ be strongly uniformly elliptic operators
with coefficients $a_{ij}^{(k)},$ $ b_{i}^{(k)},$ $  c
^{(k)} \in C^\alpha(\ov\0_+)$ for $i,j,k=1,2 $ and such that $c^{(k)}\leq0.$ 
Additionally, let  
$\B_k:=\beta_{i}^{(k)} \tr_0\p_i+\gamma^{(k)}$ be two boundary operators 
such that $\beta_{i}^{(k)}, \gamma^{(k)} \in C^{1+\alpha}(\s)$, $\beta_{2}^{(k)}>0,$ and $\gamma^{(k)}\leq 0,$ $k=1,2.$
Given $F_1,F_2\in C^{\alpha}(\ov\0_+)$, $\varphi_1\in C^{1+\alpha}(\s),$ and $\varphi_2, \varphi_3,\varphi_4\in C^{2+\alpha}(\s),$
the boundary value problem 
\begin{subequations}\label{TS1}
\begin{equation}\label{TS1a}
\left\{
\begin{array}{rlllllll}
\kL_1w_1\!\!&=&\!\!F_1 &\text{in $ \Omega_+,$}\\
 \kL_2w_2\!\!&=&\!\!F_2&\text{in $\Omega_+$},
\end{array}
\right.
\end{equation}
and 
\begin{equation}\label{TS1b}
\left\{
\begin{array}{rllllll}
\B_1w_1+\B_2w_2\!\!&=&\!\!\varphi_1&\text{on $\Gamma_{0}$},\\
w_1-w_2\!\!&=&\!\!\varphi_2 &\text{on $ \Gamma_{0},$}
\end{array}
\right.\qquad\left\{
\begin{array}{rllllll}
w_1\!\!&=&\!\!\varphi_3 &\text{on  $\Gamma_{1},$}\\
w_2\!\!&=&\!\!\varphi_4&\text{on $ \Gamma_{1},$}
\end{array}
\right.
\end{equation}
\end{subequations}
possesses a unique solution $(w_1,w_2)\in \big(C^{2+\alpha}(\ov\0_+)\big)^2$.
\end{prop}
\begin{proof}
We discuss uniqueness first. 
To this end,  let $(w_1,w_2)$ be a solution to \eqref{TS1} with   right-hand sides replaced all by zero.
If $\max_{\ov\0_+} w_1>0$ (the case $\min_{\ov\0_+} w_1<0$ is similar), the  weak elliptic maximum principle ensures that   $$\max_{\ov\0_+} w_1=\max_{\ov\0_+} w_2=w_1(x,0)=w_2(x,0)$$  for some $x\in\s.$
Applying Hopf's lemma at $(x,0)$ for both $w_1$ and $w_2$ yields  $\p_yw_1(x,0)<0$ and $\p_yw_2(x,0)<0.$ This is in contradiction to the equation $ \B_1w_1+\B_2w_2=0$ on $\Gamma_{0}$.
Hence,   $(w_1,w_2)$ has to be the zero solution, and therefore  \eqref{TS1} has at most one solution $(w_1,w_2)\in \big(C^{2+\alpha}(\ov\0_+)\big)^2$.

For the existence part, we consider a family of  operators
$\{\T_\tau:=(\T_\tau^1,\ldots,\T_\tau^6)\}_{\tau\in[0,1]}\subset\kL(\X, \Y)$, with $\X:=\big(C^{2+\alpha}(\ov\0_+)\big)^2$, $\Y:=\big(C^{\alpha}(\ov\0_+)\big)^2\times  C^{1+\alpha}(\s)\times \big(C^{2+\alpha}(\s)\big)^3,$
and 
\[\T_\tau(w_1,w_2):=\left(\begin{array}{cc}(1-\tau)\kL_1w_1+\tau  \Delta w_1\\(1-\tau)\kL_2w_2+\tau\Delta w_2\\ (1-\tau)(\B_1w_1+\B_2w_2)+\tau\tr_0\p_y(w_1+w_2)\\ \tr_0( w_1-w_2)\\ \tr_1 w_1\\ \tr_1 w_2\end{array}\right)\]
for $ \tau\in[0,1]$ and $(w_1,w_2)\in \big(C^{2+\alpha}(\ov\0_+)\big)^2$.
We observe  that $[\tau\mapsto \T_\tau]\in C([0,1], \kL(\X,\Y)).$
Moreover, by considering two suitable   Dirichlet problems the  invertibility  of $\T_1$ can be reduced to the solvability of the equation $\T_1(z_1,z_2)=(0,0,\varphi,0,0,0,0)$ for   arbitrary $\varphi\in C^{1+\alpha}(\s).$
Indeed, letting $(z_1,z_2)$ denote the solution to the equation  $\T_1(z_1,z_2)=(0,0,\varphi,0,0,0,0)$ for $\varphi:=\varphi_1-\tr_0\p_y(\wt w_1+\wt w_2),$ and setting
\[\wt w_1:=(\Delta,\tr_0 ,\tr_1)^{-1}(F_1,\varphi_2,\varphi_3)\qquad\text{and}\qquad \wt w_2:=(\Delta,\tr_0 ,\tr_1)^{-1}(F_2,0,\varphi_4),\]
 it is easy to see that $(w_1,w_2):=(z_1+\wt w_1, z_2+\wt w_2)\in \X$ satisfies $\T_1(w_1,w_2)=(F_1,F_2, \varphi_1,\varphi_2,\varphi_3,\varphi_4).$
 The solution to $\T_1(z_1,z_2)=(0,0,\varphi,0,0,0)$ is  $z_1=z_2:= (\Delta,\tr_0\p_y,\tr_1)^{-1}(0,\varphi/2,0)\in C^{2+\alpha}(\ov\0_+).$
Hence, we have shown that $\T_1$ is invertible.
If we find a constant $C>0$,  such that 
\begin{equation}\label{MC}
 \|(w_1,w_2)\|_{\X}\leq C\|\T_\tau (w_1,w_2)\|_{\Y} \qquad \text{for all $\tau\in[0,1] $ and  $(w_1,w_2)\in\X,$}
\end{equation}
then by the method of continuity, cf. e.g. \cite{GT01}, we conclude that $\T_0$ is an isomorphism, which is the   claim of the proposition. 

We are left to establish \eqref{MC}. 
To this end, we show that $\T_\tau$ corresponds to  a strongly uniformly elliptic system that satisfies 
the Complementing Condition in the sense of \cite{ADN64}  on both boundary components $\G_0$ and $\G_1$. 
Because the equations in $\0_+$ are decoupled, it is easy to see that $(\T_\tau^1,\T_\tau^2)$ defines a strongly uniformly  elliptic system for each $\tau\in[0,1]$.
Additionally, the boundary conditions defined by  $(\T_\tau^5,\T_\tau^6)$ are of Dirichlet type, and therefore the 
Complementing Condition on  $\G_1$ is straightforward.
To verify  the Complementing Condition on $\G_0$ we modify the operators $\T_\tau^i, 1\leq i\leq 4 $ as follows: we identify the principle parts $\T_\tau^{\pi,i}, 1\leq i\leq 4 $, freeze their  coefficients at an arbitrary $P\in\G_0$, and
replace $(\p_1,\p_2)$ by $(\xi, -i\p_t)$ with $0\neq \xi\in\R$.
Doing this, we arrive at the   initial value problem 
\begin{equation}\label{CBC}
 \left\{
 \begin{array}{rllll}
  v_1''-iA_{ 1}^{(1)}v_1'-A_{ 2}^{(1)}v_1\!\!&=&\!\!0&  \text{for $t>0$,}\\
   v_2''-iA_{ 1}^{(2)}v_2'-A_{ 2}^{(2)}v_2\!\!&=&\!\!0&  \text{for $t>0$,}\\
  v_1(0)\!\!&=&\!\!v_2(0),\\
  i\big[\big((1-\tau)\beta_{ 2}^{(1)}(P)+\tau\big)v_1'(0)+\big((1-\tau)\beta_{ 2}^{(2)}(P)+\tau\big)v_2'(0)\big]\\
  -(1-\tau)\xi\big( \beta_{ 1}^{(1)}(P)v_1(0)+\beta_{ 1}^{(2)}(P)v_2(0))\!\!&=&\!\!0,\\
 \end{array}
 \right.
\end{equation}
where 
\[A_{ 1}^{(k)}:=-\frac{2(1-\tau)a_{ 12}^{(k)}(P)\xi}{(1-\tau)a_{ 22}^{(k)}(P)+\tau},\quad A_{ 2}^{(k)}:=\frac{((1-\tau)a_{ 11}^{(k)}(P)+\tau)\xi^2}{(1-\tau)a_{ 22}^{(k)}(P)+\tau},\qquad k=1,2.\]
The Complementing Condition is satisfied if and only if the only bounded solution $(v_1,v_2)$ to \eqref{CBC} is the zero solution.
It is readily seen that  
\[v_{k}(t)=\gamma_{ 1}^{(k)}e^{i\delta_{ 1}^{(k)}t}e^{-\delta_{ 2}^{(k)}t}+\gamma_{ 2}^{(k)}e^{i\delta_{ 1}^{(k)}t}e^{\delta_{ 2}^{(k)}t},\qquad\text{$k=1,2,$}\]
with $\delta_{ 1}^{(k)}:=A_{ 1}^{(k)}/2$ and $\delta_{ 2}^{(k)}:=\sqrt{A_{ 2}^{(k)}-(A_{ 1}^{(k)})^2/4}>0$.
The boundedness of $v_1,v_2$ entails that $\gamma_{ 2}^{(1)}=\gamma_{ 2}^{(2)}=0.$  
Moreover, the equation $v_1(0)=v_2(0)$ implies  that $\gamma_{ 1}^{(1)}=\gamma_{ 1}^{(2)}.$
Finally, assuming $\gamma_{ 1}^{(1)}\neq0,$  we find from  the last equation of \eqref{CBC} that necessarily
\[  \delta_{ 2}^{(1)}\big((1-\tau)\beta_{ 2}^{(1)}(P)+\tau\big)+\delta_{ 2}^{(2)}\big((1-\tau)\beta_{ 2}^{(2)}(P)+\tau\big)=0.\]
However, this last equation cannot hold true as $\delta_{ 2}^{(k)}$ and $\beta_{ 2}^{(k)}(P), k=1,2,$ are positive constants.
Hence, the Complementing  Condition is also satisfied on $\G_0.$ 

We may use now Theorem 9.3 and argue similarly as in the subsequent  Remark  2 in \cite{ADN64} to conclude, together with the uniqueness result established at the beginning of the proof, 
 there exists a constant $C>0$ such that
the estimate \eqref{MC} holds. 
This completes the proof.
\end{proof}

Using Proposition \ref{P:1}, we obtain the unique solvability of certain diffraction problems within the natural H\"older spaces.
\begin{cor}\label{C:1}
Let $\kL_\pm:=a_{ij}^\pm\p_{ij}+b_{ i}^\pm\p_i+c^{\pm}$ be strongly uniformly elliptic operators
with coefficients $a_{ ij}^\pm, b_{ i}^\pm, c^{\pm} \in C^\alpha(\ov\0_\pm)$ and such that $c^{\pm}\leq0$ for  $i,j=1,2.$
Moreover, let  
$\B_\pm:=\beta_{ i} ^ \pm \tr_0\p_i+\gamma^\pm$ be  boundary operators 
such that $\beta_{ i}^\pm, \gamma^\pm \in C^{1+\alpha}(\s)$, $\beta_{ 2}^\pm>0,$  and $\gamma^\pm\leq 0$.
Then, given $F_\pm\in C^{\alpha}(\ov\0_\pm)$, $\varphi_1\in C^{1+\alpha}(\s),$ and $\varphi_2, \varphi_3, \varphi_4\in C^{2+\alpha}(\s),$ the diffraction problem
\begin{equation}\label{DPG}
\left\{
\begin{array}{rllllll}
\kL_+ v_+\!\!&=&\!\!F_+&\text{in $ \Omega_+$},\\
\kL_- v_-\!\!&=&\!\!F_-&\text{in $ \Omega_-$},\\
 \B_+v_+-\B_-v_-\!\!&=&\!\!\varphi_1 &\text{on $ \Gamma_{0}$},\\
v_+-v_-\!\!&=&\!\!\varphi_2&\text{on $\Gamma_0$},\\
v_+\!\!&=&\!\!\varphi_3&\text{on $ \Gamma_1,$}\\
 v_-\!\!&=&\!\!\varphi_4&\text{on $\G_{-1}$},
\end{array}
\right.
\end{equation}
possesses a unique solution $(v_+,v_-)\in C^{2+\alpha}(\ov\0_+)\times  C^{2+\alpha}(\ov\0_-).$
In particular, there exists a constant $C>0$ such that  
\begin{equation}\label{Schauder}
 \|v_+\|_{2+\alpha}+\|v_-\|_{2+\alpha}\leq C\left(\|F_+\|_{\alpha}+\|F_-\|_{\alpha}+\|\varphi_1\|_{1+\alpha}+\sum_{i=2}^4\|\varphi_i\|_{2+\alpha}\right). 
\end{equation}
\end{cor}
\begin{proof}
 The mapping   $\phi:\0_+\to\0_-$ defined by  $\phi(y)=-y$, $y\in\0_+$ is a  smooth diffeomorphism.
 Therefore, if $(v_+,v_-)$ solves the system \eqref{DPG}, then $w_1:=v_+$ and $w_2=v_-\circ \phi,$  
is a solution to \eqref{TS1} with 
$ \kL_1:=\kL_+,$ $\kL_2:=[w\mapsto \big(\kL_-(w\circ\phi^{-1})\big)\circ\phi],$  $ \B_1:=\B_+,$ $\B_2:=[w\mapsto-\B_-(w\circ\phi^{-1})],$
$F_1:=F_+, $ and $F_2:=F_-\circ \phi.$
The desired claim follows now directly from Proposition \ref{P:1}.
\end{proof}

 We are now in the position to prove Theorem \ref{DP}.
 \begin{proof}[Proof of Theorem \ref{DP}] Given $(f,h)\in\V,$
it follows from Corollary \ref{C:1}  that the mapping 
\begin{equation}\label{op}
(v_+,v_-)\mapsto
\begin{pmatrix} 
 \A(f,h) v_+ \\
\A(f) v_- \\
\B(f,h)v_+-\B(f)v_- \\
\tr_0(v_+-v_-) \\
\tr_1v_+,\\
 \tr_{-1}v_- 
\end{pmatrix}
\end{equation}
defines an isomorphism between  $C^{2+\alpha}(\ov\0_+)\times  C^{2+\alpha}(\ov\0_-)$ and $  C^{\alpha}(\ov\0_+)\times C^{\alpha}(\ov\0_-) \times C^{1+\alpha}(\s)\times\big( C^{2+\alpha}(\s)\big)^3.$ 
Since $(f,h)\in\V,$   a density argument shows that the operator defined by \eqref{op} is a isomorphism also when acting between the corresponding  small H\"older spaces
$h^{2+\alpha}(\0_+)\times  h^{2+\alpha}(\0_-)$ and $  h^{\alpha}(\0_+)\times h^{\alpha}(\0_-) \times h^{1+\alpha}(\s)\times\big(h^{2+\alpha}(\s)\big)^3.$
Because the differential operators and the right-hand sides of the equations of \eqref{eq:TS1} depend in a real-analytic way on $(f,h,b)\in \V\times h^{2+\alpha}(\s),$ the claim of Theorem \ref{DP} is now obvious.  
 \end{proof}

\noindent{\bf The evolution equation.}\, With the identification $\Gamma_i=\s$ for   $i\in\{-1,0, 1\}$ and by using Theorem~\ref{DP}, the problem \eqref{eq:TS1}-\eqref{eq:TS2} can  now be reformulated as an abstract fully nonlinear and nonlocal evolution equation
\begin{equation}\label{AEP}
 \p_t(f,h)=\Phi(t,(f,h)),
\end{equation}
where $\Phi:[0,T)\times \V \to (h^{1+\alpha}(\s))^2$ is the operator 
$\Phi=(\Phi_1,\Phi_2) $ defined by
\begin{equation}\label{PHI}
\begin{aligned}
 &\Phi_1(t,(f,h)):=-\B(f)v_-(f,h,b(t)),\\ 
 &\Phi_2(t,(f,h)):=-\B_1(f,h)v_+(f,h,b(t)),
 \end{aligned}
\end{equation}
with $(v_+,v_-)$ denoting the  solution operator introduced in Theorem \ref{DP}.
We note that 
\begin{equation}\label{reg1}
 \text{$\Phi\in C\big([0,T)\times \V, (h^{1+\alpha}(\s))^2\big)$\qquad and \qquad $\p_{(f,h)}\Phi\in C\big( [0,T)\times\V, \kL\big((h^{2+\alpha}(\s))^2, (h^{1+\alpha}(\s))^2\big)\big)$.}
\end{equation}
 Let $(f_0,h_0)\in\V$ and set  $$b_0:=b(0)\in h^{2+\alpha}(\s).$$
Our aim is to apply the existence  result  \cite[Theorem 8.4.1]{L95} to \eqref{AEP} for which we need to show that 
\begin{equation}\label{AG}
 -\p_{(f,h)}\Phi(0,(f_0,h_0))\in\kH\big((h^{2+\alpha}(\s))^2, (h^{1+\alpha}(\s))^2\big),
\end{equation}
 that is, $\p_{(f,h)}\Phi(0,(f_0,h_0))$ seen as an unbounded operator in $(h^{1+\alpha}(\s))^2$ with domain of definition $(h^{2+\alpha}(\s))^2$  is the generator of a strongly continuous analytic semigroup in 
 $\kL\big((h^{1+\alpha}(\s))^2\big)$.
This generator property will be established in the sections to follow for $(f_0,h_0)\in\V$ and  $b_0\in h^{2+\alpha}(\s)$ for which  \eqref{CT1} is satisfied.

Given $(f_*,h_*)\in\V$, the operator $\p_{(f,h)}\Phi(0,(f_*,h_*))\in\kL\big((h^{2+\alpha}(\s))^2,  (h^{1+\alpha}(\s))^2\big)$ can be written in matrix form
\[\p_{(f,h)}\Phi(0,(f_*,h_*))
=\begin{pmatrix}
  \p_{f}\Phi_1(0,(f_*,h_*)) & \p_{h}\Phi_1(0,(f_*,h_*))\\[1ex]
   \p_{f}\Phi_2(0,(f_*,h_*)) & \p_{h}\Phi_2(0,(f_*,h_*))
 \end{pmatrix},
\] 
where, according to the definition \eqref{PHI}, we have
\begin{equation}\label{DC}
\begin{aligned}
 \p_{f}\Phi_1(0,(f_*,h_*))[f]&=-\p_f\B(f_*)[f]v_-(f_*,h_*,b_0)-\B(f_*)\p_fv_-(f_*,h_*,b_0)[f],\\
 \p_{h}\Phi_1(0,(f_*,h_*))[h]&=- \B(f_*)\p_hv_-(f_*,h_*,b_0)[h],\\
 \p_{h}\Phi_2(0,(f_*,h_*))[h]&=-\p_h\B_1(f_*,h_*)[h]v_+(f_*,h_*,b_0)-\B_1(f_*,h_*)\p_hv_+(f_*,h_*,b_0)[h],
\end{aligned}
\end{equation}
for $(f,h)\in (h^{2+\alpha}(\s))^2.$
Additionally,  $(w_+[f],w_-[f]):=(\p_fv_+(f_*,h_*,b_0)[f],\p_fv_-(f_*,h_*,b_0)[f])$ is the solution to the
diffraction problem 
\begin{equation}\label{P2}
\left\{
\begin{array}{rllllll}
\A(f_*,h_*) w_+[f]\!\!&=&\!\!-\p_f\A(f_*,h_*)[f]v_+^* &\text{in $ \Omega_+$},\\
\A(f_*) w_-[f]\!\!&=&\!\!-\p_f\A(f_*)[f] v_-^*&\text{in $ \Omega_-$},\\
 \B(f_*,h_*)w_+[f]-\B(f_*)w_-[f]\!\!&=&\!\!-\p_f\B(f_*,h_*)[f]v_+^*+\p_f\B(f_*)[f]v_-^*&\text{on $\Gamma_{0}$},\\
w_+[f]-w_-[f]\!\!&=&\!\!g(\rho_+-\rho_-) f&\text{on $ \Gamma_0$},\\
w_+[f]\!\!&=&\!\!0&\text{on $ \Gamma_1$},\\
 w_-[f]\!\!&=&\!\!0&\text{on $\G_{-1}$},
\end{array}
\right.
\end{equation}
 and   $(W_+[h],W_-[h]):=(\p_hv_+^*(f_*,h_*,b_0)[h],\p_hv_-^*(f_*,h_*,b_0)[h])$  solves the
diffraction problem 
\begin{equation}\label{P3}
\left\{
\begin{array}{rllllll}
\A(f_*,h_*) W_+[h]\!\!&=&\!\!-\p_h\A(f_*,h_*)[h]v_+^* &\text{in $\Omega_+$},\\
\A(f_*) W_-[h]\!\!&=&\!\!0&\text{in $ \Omega_-$},\\
 \B(f_*,h_*)W_+[h]-\B(f_*)W_-[h]\!\!&=&\!\!-\p_h\B(f_*,h_*)[h]v_+^*&\text{on $ \Gamma_{0}$},\\
W_+[h]-W_-[h]\!\!&=&\!\!0&\text{on $ \Gamma_0$},\\
W_+[h]\!\!&=&\!\!g\rho_+h&\text{on $ \Gamma_1$},\\
 W_-[h]\!\!&=&\!\!0&\text{on $\G_{-1}$}.
\end{array}
\right.
\end{equation}
In \eqref{P2} and \eqref{P3} we set  
\begin{equation}\label{zero}
(v_+^*,v_-^*):=(v_+,v_-)(f_*,h_*,b_0)\,.
\end{equation}
According to  \cite[Theorem I.1.6.1 and Remark I.1.6.2]{Am95},   \eqref{AG} is satisfied provided  the diagonal operators satisfy
\begin{align}
& -\p_{f}\Phi_1(0,(f_0,h_0))\in\kH\big(h^{2+\alpha}(\s), h^{1+\alpha}(\s)\big),\label{Claim1}\\
& -\p_{h}\Phi_2(0,(f_0,h_0))\in\kH\big(h^{2+\alpha}(\s), h^{1+\alpha}(\s)\big),\label{Claim2}
\end{align}
and  provided  the following property holds for the off-diagonal operator: for each $\e>0$ there exists $K_0=K_0(\e)>0$ such that  
\begin{align}
&\|\p_{h}\Phi_1(0,(f_0,h_0))[h]\|_{1+\alpha}\leq \e\|h\|_{2+\alpha}+K_0\|h\|_{1+\alpha}\qquad\text{for all $h\in h^{2+\alpha}(\s)$}.\label{Claim3}
\end{align}
So, to establish \eqref{AG} various computations are needed.
In Section \ref{Sec:4} we first prove \eqref{Claim3} based on Schauder estimates for diffraction problems as presented in Corollary \ref{C:1}.
The proofs of \eqref{Claim1} and \eqref{Claim2} respectively, are given in Sections \ref{Sec:5} (see Theorem~\ref{T:c1}) and \ref{Sec:6} (see Theorem~\ref{T:c2})
for which we use localization techniques in the spirit of \cite{ES95} (see also \cite{E94, ES97, We14}). 
However, our localization techniques are quite different from those therein as we do not
consider problems in the halfplane, and our results are sharper, e.g. see Theorem \ref{T:c2} and \cite[Theorem 14]{We14}.

\section{An off-diagonal operator}\label{Sec:4}
The main goal of this section is to establish the property \eqref{Claim3}.
 This is a consequence of the following lemma where \eqref{Claim3} is established for general $(f_*,h_*)\in\V$. 

\begin{lemma}\label{L:NDO}
Let   $(f_*,h_*)\in\V$. Given $\e\in(0,1)$, there exists $K_0=K_0(\e)>0$ such that  
\begin{align}
&\|\p_{h}\Phi_1(0,(f_*,h_*))[h]\|_{1+\alpha}\leq \e\|h\|_{2+\alpha}+K_0\|h\|_{1+\alpha}\qquad\text{for all $h\in h^{2+\alpha}(\s)$}.\label{Claim3'}
\end{align} 
\end{lemma}
\begin{proof}
The proof is based on  estimates for boundary value problems for elliptic systems, cf. \eqref{Schauder} and~\cite{ADN64}.
Let $(f_n,h_n,b_n)$ be a sequence  in  $\big(C^\infty(\s)\big)^3$ which converges towards $(f_*,h_*,b_0)$ in  $\big(h^{2+\alpha}(\s)\big)^3$ and  such that $(f_n,h_n)\in\V$ for all $n.$ 
We then have
\begin{align}
&\|\p_{h}\Phi_1(0,(f_*,h_*))[h]\|_{1+\alpha}\leq C\|\tr_0(\nabla\p_hv_-(f_*,h_*,b_0)[h])\|_{1+\alpha}\nonumber\\[1ex]
&\leq C \|\tr_0\nabla\big( \p_hv_-(f_n,h_n,b_n)[h]-\p_hv_-(f_*,h_*,b_0)[h]\big)\|_{1+\alpha}+C\|\tr_0(\nabla\p_hv_-(f_n,h_n,b_n)[h])\|_{1+\alpha}\nonumber\\[1ex]
&\leq C\|\p_hv_-(f_n,h_n,b_n)-\p_hv_-(f_*,h_*,b_0)\|_{\kL\big(h^{2+\alpha}(\s), h^{2+\alpha}(\0_-)\big)}\|h\|_{2+\alpha}\nonumber\\[1ex]
&\hspace{0.5cm}+C\|\tr_0\nabla(\p_hv_-(f_n,h_n,b_n)[h])\|_{1+\alpha}=:I_1+I_2\label{I0}
\end{align}
for all $h\in h^{2+\alpha}(\s).$
Let $\e>0$ be given.
In view of Theorem \ref{DP}, we can choose  $n$ large enough to guarantee that
\begin{equation}\label{I1}
 I_1\leq \frac{\e}{2}\|h\|_{2+\alpha}.
\end{equation}
We are now left to estimate the term $I_2$ for a fixed $n$ such that \eqref{I1} holds.
To this end, we  let $(W_+^n[h], W_-^n[h]):=(\p_hv_+(f_n,h_n,b_n)[h], \p_hv_-(f_n,h_n,b_n)[h])$ denote  the solution  to \eqref{P3} when replacing $(f_*,h_*,b_0)$ by  $(f_n,h_n,b_n)$  and $v_+^*$ by $v_+^n:=v_+(f_n,h_n,b_n)$.
Because   $(f_n,h_n,b_n)$ is smooth, it follows from \cite[Theorem 9.3]{ADN64}, by arguing as in the proof of Theorem \ref{DP}, that $v_+^n$ is a smooth function up to the boundary of $\0_+.$ 
In particular, $v_+^n\in h^{3+\alpha}(\0_+).$ 
We now split the solution $(W_+^n[h], W_-^n[h])$ as 
\begin{equation}\label{I2}
 (W_+^n[h], W_-^n[h])=(W_+^{1}, W_-^{1})+(W_+^{2}, W_-^{2}),
\end{equation}
where $(W_+^{1}, W_-^{1})$ is the solution to 
\begin{equation}\label{P4}
\left\{
\begin{array}{rllllll}
\A(f_n,h_n) W_+^1\!\!&=&\!\!-\p_h\A(f_n,h_n)[h]v_+^n &\text{in  $\0_+$},\\
\A(f_n) W_-^1\!\!&=&\!\!0&\text{in $\0_-$} ,\\
 \B(f_n,h_n)W_+^1-\B(f_n)W_-^1\!\!&=&\!\!0&\text{on $ \Gamma_{0}$},\\
W_+^1-W_-^1\!\!&=&\!\!0&\text{on $ \Gamma_0$},\\
W_+^1\!\!&=&\!\!0&\text{on $ \Gamma_1$,}\\
 W_-^1\!\!&=&\!\!0&\text{on $\G_{-1},$}
\end{array}
\right.
\end{equation} 
respectively $(W_+^{2}, W_-^{2})$ solves
\begin{equation}\label{P5}
\left\{
\begin{array}{rllllll}
\A(f_n,h_n) W_+^2\!\!&=&\!\!0 &\text{in  $\0_+$},\\
\A(f_n) W_-^2\!\!&=&\!\!0&\text{in $\0_-$} ,\\
 \B(f_n,h_n)W_+^2-\B(f_n)W_-^2\!\!&=&\!\!-\p_h\B(f_n,h_n)[h]v_+^n&\text{on $ \Gamma_{0}$},\\
W_+^2-W_-^2\!\!&=&\!\!0&\text{on $ \Gamma_0$},\\
W_+^2\!\!&=&\!\!g\rho_+h&\text{on $ \Gamma_1$,}\\
 W_-^2\!\!&=&\!\!0&\text{on $\G_{-1}.$}
\end{array}
\right.
\end{equation} 
In order to estimate $W_-^2$ we note that 
\begin{align*}
\p_h \B(f_n,h_n)[h]=-k\mu_+^{-1}\frac{(1+{f'_n}^2)h}{(h_n- f_n)^2}\tr_0 \p_y,
\end{align*}
(see the formulas in the appendix), and therefore, the right-hand side of the third equation of \eqref{P5} belongs to $h^{2+\alpha}(\s).$ 
In virtue of \cite[Theorem 9.3]{ADN64}, there exists a constant $C>0$ (which depends on the previously fixed $n$) such that on the subdomain $(1/2)\0_-:=\s\times(-1/2,0)$ of $\0_-$ we have
\begin{equation}\label{improved}
 \|W_-^2\|_{2+\alpha}^{(1/2)\0_-}\leq C\big(\|\p_h\B(f_n,h_n)[h]v_+^n\|_{1+\alpha}+\|W_+^2\|_0+\|W_-^2\|_0\big)
\end{equation}
for all $h\in h^{2+\alpha}(\s)$.
Let  $\alpha'\in(0,\alpha)$ be fixed. 
Recalling \eqref{Schauder}, we have  
\begin{equation*} 
 \|W_+^2\|_0+\|W_-^2\|_0\leq \|W_+^2\|_{2+\alpha'}+\|W_-^2\|_{2+\alpha'} \leq C\big(\|\p_h\B(f_n,h_n)[h]v_+^n\|_{1+\alpha'}+\|h\|_{2+\alpha'}\big),
\end{equation*}
and together with \eqref{improved} we end up with 
\begin{equation}\label{improved1}
 \|W_-^2\|_{2+\alpha}^{(1/2)\0_-}\leq C\big(\|\p_h\B(f_n,h_n)[h]v_+^n\|_{1+\alpha}+\|h\|_{2+\alpha'}\big)\leq C \|h\|_{2+\alpha'}.
\end{equation}
Using the following interpolation property of the small H\"older spaces (e.g. see \cite{L95})
 \begin{equation}\label{interpolation}
(h^{r}(\s),h^{s}(\s))_\theta=h^{(1-\theta)r+\theta s}(\s)\qquad\text{for $ \theta\in(0,1)$ and $(1-\theta)r+\theta s\notin\N,$}
\end{equation}
where $(\cdot,\cdot)_\theta=(\cdot,\cdot)_{\theta,\infty}^0$ is the continuous interpolation functor introduced by Da Prato and Grisvard \cite{DG79}, we infer from \eqref{improved1} and Young's inequality there exists a constant $K(\e)$ such that 
\begin{align}
 \|\tr_0 \nabla W_-^2\|_{1+\alpha} &\leq\| W_-^2\|_{2+\alpha}^{(1/2)\0_-}\leq   \frac{\e}{4}\|h\|_{2+\alpha}+K(\e)\|h\|_{1+\alpha}\label{GN1}
\end{align}
for all $h\in h^{2+\alpha}(\s)$.

Thanks to \eqref{GN1}, we are left to prove a similar estimate for $\|\tr_0 \nabla W_-^1\|_{1+\alpha}$.
To this end, we choose $\delta\in(0,1)$ and a function $\chi:=\chi_\delta \in C^\infty([0,1])$ such that $0\leq \chi \leq 1,$ $\chi=1$ on $(0,\delta/8),$ and  $\chi=0$ on $[\delta/4,1].$
Note that \eqref{Schauder} and the regularity  $\A\in C^\omega\big(\V, \kL\big(h^{2+\alpha}(\s), h^{1+\alpha}(\s)\big)\big)$ implies that, for a fixed $\alpha'\in(0,\alpha),$
\begin{equation}\label{Chopin}
 \|W_-^1\|_{2+\alpha'}+\|W_+^1\|_{2+\alpha'}\leq C\|h\|_{2+\alpha'}
\end{equation}
for all $h\in h^{2+\alpha}(\s).$
Using the more compact notation $\A(f_n,h_n)=a_{ij}\p_{ij}+b_2\p_2$, the pair $(W_+^1\chi,W_-^1)$ solves the  system \eqref{P4},
but with the first equation replaced by  the elliptic equation 
\begin{align*}
 \A(f_n,h_n) (W_+^1\chi)=&-\chi\p_h\A(f_n,h_n)[h]v_+^n+2a_{12}\p_1W_+^1\chi'+a_{22}(\p_2W_+^1\chi'+W_+^1\chi'')+b_2W_+^1\chi'
\end{align*}
in $\0_+$.
We now use \eqref{Schauder} and \eqref{Chopin} to get
\begin{align}
 \|W_-^1\|_{2+\alpha}\leq& C\Big(\|\chi\p_h\A(f_n,h_n)[h]v_+^n\|_\alpha+\|(2a_{12}\p_1W_+^1+b_2W_+^1+a_{22}\p_2W_+^1)\chi'+a_{22}W_+^1\chi''\|_\alpha\Big)\nonumber\\
 \leq&C \|\chi\p_h\A(f_n,h_n)[h]v_+^n\|_\alpha+C(\delta)\| W_+^1 \|_{1+\alpha} \nonumber\\
 \leq&  C\|yh''\chi\|_{\alpha}+  C(\delta)\|h\|_{2+\alpha'}\nonumber\\
 \leq& C\|y\chi\|_{0}\|h\|_{2+\alpha}+  C(\delta)\|h\|_{2+\alpha'}. \label{Chopin2}
\end{align}
Choosing $\delta$ such that  $C\|y\chi\|_{0}<\e/4,$ we infer from \eqref{Chopin2}  there exists  a constant $K(\e)$ such that 
\begin{align}\label{GN2}
 \|\tr_0 \nabla W_-^1\|_{1+\alpha} \leq&\| W_-^1\|_{2+\alpha}^{\0_-}\leq (\e/4)\|h\|_{2+\alpha}+ K(\e)\|h\|_{1+\alpha}
\end{align}
for all $h\in h^{2+\alpha}(\s).$
Gathering now \eqref{I0}, \eqref{I1}, \eqref{I2}, \eqref{GN1}, and \eqref{GN2} we have established the desired estimate \eqref{Claim3'}.
\end{proof}


\section{The first diagonal operator}\label{Sec:5}

 In this section we prove that  $\p_{f}\Phi_1(0,(f_0,h_0)) $ is the generator of a strongly continuous and analytic semigroup, hence \eqref{Claim1},  when $(f_0,h_0)$ and $b_0$  are such that the first inequality in \eqref{CT1} holds.
 This is stated in Theorem~\ref{T:c1}.  
We start in Lemma \ref{L:LOP1}  by identifying the ``leading order part'' $\p_{f}\Phi_1^\pi(f_*,h_*)$ of $\p_{f}\Phi_1(0,(f_*,h_*))$  for a general $(f_*,h_*)\in\V$. 
This reduces our  task (see \eqref{Claim4} and \cite[Theorem I.1.3.1 (ii)]{Am95})  to showing  the generator property merely for $\p_{f}\Phi_1^\pi(f_0,h_0)$, and 
thus allows us  to neglect  several ``lower order terms'' in the quite involved computations to follow.  
Following this step,   we locally approximate  the principal part $\p_{f}\Phi_1^\pi(0,(f_*,h_*))$ by Fourier multipliers for $(f_*,h_*)$ sufficiently close to $(f_0,h_0)$ and possessing  additional regularity, cf. Theorem \ref{T1}.
These multipliers are then shown to be  generators of strongly continuous analytic semigroups, where the  constants in the resolvent estimates  are uniform with respect to certain variables, cf.  Lemma \ref{L:2}.
This uniformity  property is essential when establishing the  desired Theorem \ref{T:c1} by means of Lemma \ref{L:Ad1} and a continuity argument. 
\begin{lemma}\label{L:LOP1}
Given  $(f_*,h_*)\in\V$, let $\p_{f}\Phi_1^\pi(f_*,h_*)\in\kL\big(h^{2+\alpha}(\s), h^{1+\alpha}(\s)\big)$ denote the operator defined by
 \begin{equation}\label{LOP1}
 \p_{f}\Phi_1^\pi(f_*,h_*)[f]:=-\frac{k}{\mu_-}\Big(  \frac{2f_*'\tr_0 \p_yv_-^*}{ f_*-d} -\tr_0 \p_x v_-^* \Big)f'   -\B(f_*) w_-^{\pi}[f],\quad f\in h^{2+\alpha}(\s),
 \end{equation}
 with  $(v_+^*,v_-^*)$  defined  in \eqref{zero} and  where $(w_+^{\pi}[f],w_-^{\pi}[f])$  denotes the solution to the problem
 \begin{equation}\label{LOP1a}
\left\{
\begin{array}{rllllll}
\A^\pi_0(f_*,h_*) w_+^{\pi}[f]\!\!&=&\!\!\cfrac{\tr_0\p_yv_+^*}{h_*-f_*}f'' &\text{in $ \Omega_+$},\\
\A^\pi_0(f_*) w_-^{\pi}[f]\!\!&=&\!\!\cfrac{\tr_0\p_y v_-^* }{f_*-d}f''&\text{in $ \Omega_-$},\\
 \B(f_*,h_*)w_+^{\pi}[f]-\B(f_*)w_-^{\pi}[f]\!\!&=&\!\!-\p_f\B^\pi(f_*,h_*)[f]v_+^*+\p_f\B^\pi(f_*)[f]v_-^*&\text{on $ \Gamma_{0}$},\\
w_+^{\pi}[f]-w_-^{\pi}[f]\!\!&=&\!\!g(\rho_+-\rho_-) f&\text{on $\Gamma_0$},\\
w_+^{\pi}[f]\!\!&=&\!\!0&\text{on $ \Gamma_1$},\\
 w_-^{\pi}[f]\!\!&=&\!\!0&\text{on $\G_{-1}$},
\end{array}
\right.
\end{equation}
with
\begin{equation}\label{eqs}
\begin{aligned}
&\A^\pi_0(f_*):=\p_{xx}-\frac{2f_*'}{f_*-d}\p_{xy}+
\frac{f_*'^2+1}{(f_*-d)^2}\p_{yy},\quad \A^\pi_0(f_*,h_*):=\p_{xx}-\frac{2f_*'}{ h_*-f_*}\p_{xy}+\frac{f_*'^2+1}{(h_*-f_*)^2}\p_{yy},\\
&\p_f\B^\pi(f_*)[f]:=\frac{k}{\mu_-} \Big(\frac{2f_*'}{ f_*-d}  \tr_0 \p_y- \tr_0 \p_x\Big)f',\quad \p_f\B^\pi(f_*,h_*)[f]:=\frac{k}{\mu_+} \Big( \frac{2f_*'}{h_*- f_*}  \tr_0\p_y-\tr_0 \p_x\Big)f'.
\end{aligned}
\end{equation}
Then, given  $\e\in(0,1)$, there exists $K_1=K_1(\e)>0$ such that  
\begin{align}
&\|\p_{f}\Phi_1(0,(f_*,h_*))[f]-\p_{f}\Phi_1^\pi(f_*,h_*)[f]\|_{1+\alpha}\leq \e\|f\|_{2+\alpha}+K_1\|f\|_{1+\alpha}\qquad\text{for all $f\in h^{2+\alpha}(\s)$}.\label{Claim4}
\end{align}
 Moreover, $ \p_{f}\Phi_1^\pi \in C^\omega\big(\V, \kL\big(h^{2+\alpha}(\s), h^{1+\alpha}(\s)\big)\big)$.
\end{lemma}
\begin{proof} The regularity property follows by using similar arguments as in the proof of Theorem \ref{DP}. In order to prove \eqref{Claim4},
let $\e\in(0,1)$ be given and $\alpha'\in(0,\alpha)$ be fixed. 
We choose $\delta\in(0,1)$ and a cut-off function $\chi:=\chi_\delta\in C^\infty([-1,1])$ such that   $0\leq \chi \leq 1,$
$\chi=1$ for $|y|\leq\delta/2,$ and   $\chi=0$ for $|y|\geq\delta.$
Using  the same notation as in Section \ref{Sec:3} (see \eqref{P2}) we have from \eqref{DC}, \eqref{LOP1}, and the formula for $\p_f\B(f_*)$ in the Appendix \ref{App}
\begin{align}
\|\p_{f}\Phi_1(0,(f_*,h_*))[f]-\p_{f}\Phi_1^\pi(f_*,h_*)[f]\|_{1+\alpha}\leq&  C(\|f\|_{1+\alpha}+\|\B(f_*)\big(w_-[f]-w_-^{\pi}[f]\big)\|_{1+\alpha}\big)\nonumber\\
\leq&  C(\|f\|_{1+\alpha}+\|\tr_0\nabla\big(w_-[f]-w_-^{\pi}[f]\big)\|_{1+\alpha}\big)\nonumber\\
\leq&  C(\|f\|_{1+\alpha}+\|\chi\big(w_-[f]-w_-^{\pi}[f]\big)\|_{2+\alpha}^{\0_-}\big),\label{LOP1b}
\end{align}
where $C$ is independent of $\delta $. 
We now observe that $$(u_+,u_-):= (w_+^{\pi}[f],w_-^{\pi}[f])-(w_+[f],w_-[f])$$ solves a diffraction problem of the form
\begin{equation*}\left\{
\begin{array}{rllllll}
\A^\pi_0(f_*,h_*)u_+\!\!&=&\!\! a_0^+f+a_1^+f'+ya_2^+((1-y)\p_yv_+^*-\tr_0\p_yv_+^*)f''\\
&&+y(b_0^+\p_{xy}w_+[f]+b_1^+\p_{yy}w_+[f])+b_2^+\p_yw_+[f] &\text{in $ \Omega_+$},\\
\A^\pi_0(f_*) u_-\!\!&=&\!\! a_0^-f+a_1^-f'+ya_2^-+((1+y)\p_yv_-^*-\tr_0\p_yv_-^*)f''\\
&&+y(b_0^-\p_{xy}w_-[f]+b_1^-\p_{yy}w_-[f])+b_2^-\p_yw_-[f]&\text{in $ \Omega_-$},\\
 \B(f_*,h_*)u_+-\B(f_*)u_-\!\!&=&\!\!c f&\text{on $ \Gamma_{0}$},\\
u_+-u_- \!\!&=&\!\!0  &\text{on $\Gamma_0$},\\
u_+\!\!&=&\!\!0&\text{on $ \Gamma_1$},\\
 u_-\!\!&=&\!\!0&\text{on $\G_{-1}$},
\end{array}
\right.
\end{equation*}
with functions $a_i^\pm,b_i^\pm\in h^\alpha(\0_\pm), 0\leq i\leq2,$ and $c\in h^{1+\alpha}(\s).$
Therefore, $(u_+^\chi,u_-^\chi):=\chi(u_+,u_-)$ is the solution to
\begin{equation}\label{LOP1c}\left\{
\begin{array}{rllllll}
\A^\pi_0(f_*,h_*)u_+^\chi\!\!&=&\!\! \chi\A^\pi_0(f_*,h_*)u_+-\frac{2f_*'}{ h_*-f_*}\chi'\p_xu_{+}+\frac{f_*'^2+1}{(h_*-f_*)^2}(2\chi'\p_yu_{+}+\chi''u_+) &\text{in $ \Omega_+$},\\
\A^\pi_0(f_*) u_-^\chi\!\!&=&\!\! \chi\A^\pi_0(f_*) u_--\frac{2f_*'}{f_*-d}\chi'\p_xu_{-}+\frac{f_*'^2+1}{(f_*-d)^2}(2\chi'\p_yu_{-}+\chi''u_-)&\text{in $ \Omega_-$},\\
 \B(f_*,h_*)u_+^\chi-\B(f_*)u_-^\chi\!\!&=&\!\!c f&\text{on $ \Gamma_{0}$},\\
u_+^\chi-u_-^\chi \!\!&=&\!\!0  &\text{on $\Gamma_0$},\\
u_+^\chi\!\!&=&\!\!0&\text{on $ \Gamma_1$},\\
 u_-^\chi\!\!&=&\!\!0&\text{on $\G_{-1}$}.
\end{array}
\right.
\end{equation}
We can estimate the solutions to the systems   \eqref{LOP1a} and \eqref{LOP1c}  by using   \eqref{Schauder} and  obtain that 
\begin{align*}
 \|u_-^\chi\|_{2+\alpha}^{\0_-}\leq&C\Big(\|\chi\A^\pi_0(f_*,h_*)u_+\|_\alpha^{\0_+}+\|\chi\A^\pi_0(f_*)u_-\|_\alpha^{\0_-}+\|f\|_{1+\alpha}\Big)+C(\delta)\big(\|u_+\|_{2+\alpha'}+\|u_-\|_{2+\alpha'}\big)\\
 \leq &C\Big(\|\chi((1-y)\p_yv_+^*-\tr_0\p_yv_+^*)f''\|_\alpha^{\0_+}+\|\chi((1+y)\p_yv_-^*-\tr_0\p_yv_-^*)f''\|_\alpha^{\0_-}\\
 &\quad+\|\chi y(b_0^+\p_{xy}w_+[f]+b_1^+\p_{yy}w_+[f])\|_\alpha^{\0_+}+\|\chi y(b_0^-\p_{xy}w_-[f]+b_1^-\p_{yy}w_-[f])\|_\alpha^{\0_-}\Big)\\
 &+C(\delta)\|f\|_{2+\alpha'}\\
 \leq &C\Big(\|\chi((1-y)\p_yv_+^*-\tr_0\p_yv_+^*)\|_0^{\0_+}+\|\chi((1+y)\p_yv_-^*-\tr_0\p_yv_-^*)\|_0^{\0_-}\\
 &\quad+\|\chi y\|_0^{\0_+}+\|\chi y\|_0^{\0_-}\Big)\|f\|_{2+\alpha}+C(\delta)\|f\|_{2+\alpha'}.
\end{align*}
Recalling the definition of $\chi=\chi_\delta$ and choosing $\delta>0$ such that 
$$\|\chi((1-y)\p_yv_+^*-\tr_0\p_yv_+^*)\|_0^{\0_+}+\|\chi((1+y)\p_yv_-^*-\tr_0\p_yv_-^*)\|_0^{\0_-}+\|\chi y\|_0^{\0_+}+\|\chi y\|_0^{\0_-}<\frac{\e}{2C},$$ the desired estimate \eqref{Claim4} follows from the interpolation  property  \eqref{interpolation}.
\end{proof}

Let $(f_0,h_0)$ and $b_0$  be such that the first inequality of \eqref{CT1} holds. We are now left to show that the principal part   $\p_{f}\Phi_1^\pi(f_0,h_0)$ of $\p_f\Phi_1(0,(f_0,h_0))$ is the generator of a strongly continuous and analytic semigroup in $\kL(h^{1+\alpha}(\s))$.
In view of  the classical result \cite[Theorem I.1.2.2]{Am95}, one has $- \p_{f}\Phi_1^\pi(f_0,h_0)\in\kH\big(h^{2+\alpha}(\s), h^{1+\alpha}(\s)\big)$ if 
and only if there exist constants $\kappa_1\geq 1$ and $\omega_1>0$ such that 
$-\p_{f}\Phi_1^\pi(f_0,h_0)\in \kH\big(h^{2+\alpha}(\s), h^{1+\alpha}(\s),\kappa_1,\omega_1\big)$, that is, $$\omega_1- \p_{f}\Phi_1^\pi(f_0,h_0) \in{\rm Isom}\big(h^{2+\alpha}(\s), h^{1+\alpha}(\s)\big)$$ and 
\[
 \qquad \kappa_1^{-1}\leq\frac{\|(\lambda- \p_{f}\Phi_1^\pi(f_0,h_0))f\|_{1+\alpha}}{|\lambda|\cdot \|f\|_{1+\alpha}+\|f\|_{2+\alpha}}\leq \kappa_1\qquad\text{for all $\re\lambda\geq\omega_1$ and $0\neq f\in h^{2+\alpha}(\s)$.}
 \]
To this end, for $\sigma>0$, let $S_\sigma$ denote the set consisting of those $(f_*,h_*)\in\V$ satisfying the inequalities
\begin{equation*}
 \begin{array}{lll}
 (1) & \sigma< \min\{f_*-d, h_*-f_*\}, \qquad\|f_*\|_2+\|h_*\|_2< \sigma^{-1}\,,\\[2ex]
  (2) &\displaystyle g(\rho_--\rho_+)> \sigma+\frac{\tr_0\p_yv_-^*}{h_*-f_*}-\frac{\tr_0\p_yv_+^*}{f_*-d}\,,\\[3ex]
  (3)  &   \|\tr_0\p_yv_+^*\|_0+\|\tr_0\p_yv_-^*\|_0< \sigma^{-1} \,,
 \end{array}
\end{equation*}
where $(v_+^*,v_-^*)$ is defined in \eqref{zero}.
Since the functions $(f_0,h_0)$ and $b_0$ are chosen such that the first inequality in \eqref{CT1} is satisfied, we may choose $\sigma$ such that  $(f_0,h_0)\in S_\sigma$.

\begin{lemma}\label{L:Ad1}
Let $\sigma>0$ be such that  $(f_0,h_0)\in S_\sigma$.
Assume  there exists a constant $\wt \kappa_1:=\wt \kappa_1(\sigma)$ and for each $(f_*,h_*)\in  \big(h^{3+\alpha}(\s)\big)^2\cap S_\sigma\cap \mathbb{B}_{(h^{2+\alpha}(\s))^2}((f_0,h_0),\sigma)$ there exists a further constant $\wt \omega_1>0$ with the property that
$$-\p_{f}\Phi_{1}^{\pi}(f_*,h_*)\in \kH\big(h^{2+\alpha}(\s), h^{1+\alpha}(\s),\wt \kappa_1,\wt \omega_1\big).$$
Then
$$-\p_{f}\Phi_{1}^{\pi}(f_0,h_0)\in \kH\big(h^{2+\alpha}(\s), h^{1+\alpha}(\s)\big).$$
\end{lemma}
\begin{proof} Using the regularity assertions in Theorem \ref{DP} and Lemma \ref{L:LOP1} together with the density of $h^{3+\alpha}(\s)$ in $h^{2+\alpha}(\s),$ we find  $(f_*,h_*)\in  \big(h^{3+\alpha}(\s)\big)^2\cap S_\sigma\cap \mathbb{B}_{(h^{2+\alpha}(\s))^2}((f_0,h_0),\sigma)$ 
such that 
\[\|\p_{f}\Phi_{1}^{\pi}(f_*,h_*)-\p_{f}\Phi_{1}^{\pi}(f_0,h_0)\|_{\kL\big(h^{2+\alpha}(\s), h^{1+\alpha}(\s)\big)}<1/2\wt\kappa_1.\]
The perturbation result \cite[Theorem I.1.3.1 (i)]{Am95} implies
$$-\p_{f}\Phi_{1}^{\pi}(f_0,h_0)\in \kH\big(h^{2+\alpha}(\s), h^{1+\alpha}(\s),2\wt \kappa_1,\wt \omega_1\big),$$
which yields the desired claim.
\end{proof}
\begin{rem}\label{R:ad}
 We will see in the proof of Theorem \ref{T:c1} that,  in contrast to $\kappa_1,$   the constants $\omega_1$ in Lemma \ref{L:Ad1} appear to  depend on the $\|\cdot\|_{3+\alpha}$-norm  of $(f_*,h_*)$.
 Whence, for $(f_*,h_*)$ close to $(f_0,h_0)$, these constants may become large. 
\end{rem}

We are now left to  establish the assumptions of  Lemma \ref{L:Ad1} for some sufficiently small $\sigma$.
Therefore, we pick an arbitrary $(f_*,h_*)\in  \big(h^{3+\alpha}(\s)\big)^2\cap S_\sigma$
and  introduce a parameter $\tau\in[0,1] $ which will enable us to continuously transform
the leading order part  $\p_{f}\Phi_1^{\pi*}:= \p_{f}\Phi_1^\pi(f_*,h_*)$ of $\p_f\Phi_1(0,(f_*,h_*))$ into a (negative) Dirichlet-Neumann map.
This will allow us to use the continuation method and prove, by relaying on the properties of this  Dirichlet-Neumann map, that large positive real numbers belong to the resolvent set of $\p_{f}\Phi_1^\pi(f_*,h_*)$.
 We emphasize that this construction uses to a large extent  the additional regularity of $(f_*,h_*)$, as  the mappings $(v_+^*,v_-^*)$ introduced in \eqref{zero}  possess additional regularity close to the boundary $\Gamma_0$ in this case. 
More precisely, for each $\tau\in[0,1]$ we introduce the operator
$\p_{f}\Phi_{1,\tau}^{\pi*}\in\kL\big(h^{2+\alpha}(\s), h^{1+\alpha}(\s)\big)$  by  
 \begin{equation}\label{HO1}
 \p_{f}\Phi_{1,\tau}^{\pi*}[f]:=-\frac{\tau k}{\mu_-}\Big( \frac{2f'_{*}\tr_0 \p_yv_-^*}{ f_*-d} -\tr_0 \p_x v_-^*  \Big)f'  -\B(f_*) w_{-\tau}^\pi[f],\quad f\in h^{2+\alpha}(\s),
 \end{equation}
 with   $(v_+^*,v_-^*)$ being defined in \eqref{zero} and  $(w_{+\tau}^\pi[f],w_{-\tau}^\pi[f])$  denoting the solution to 
  \begin{equation}\label{HO1a}
\left\{
\begin{array}{rllllll}
\A^\pi_0(f_*,h_*) w_{+\tau}^\pi[f]\!\!&=&\!\!\cfrac{\tau\tr_0\p_yv_+^*}{h_*-f_*}f'' &\text{in $ \Omega_+$},\\
\A^\pi_0(f_*) w_{-\tau}^\pi[f]\!\!&=&\!\!\cfrac{\tau\tr_0\p_y v_-^* }{f_*-d}f''&\text{in $ \Omega_-$},\\
 \B(f_*,h_*)w_{+\tau}^\pi[f]-\B(f_*)w_{-\tau}^\pi[f]\!\!&=&\!\!-\tau\p_f\B^\pi(f_*,h_*)[f]v_+^*+\tau\p_f\B^\pi(f_*)[f]v_-^*&\text{on $ \Gamma_{0}$},\\
w_{+\tau}^\pi[f]-w_{-\tau}^\pi[f]\!\!&=&\!\!-\Big[g(\rho_--\rho_+)+(1-\tau)\Big(\cfrac{\tr_0\p_yv_+^*}{h_*-f_*}-\cfrac{\tr_0\p_yv_-^*}{f_*-d}\Big) \Big]f&\text{on $\Gamma_0$},\\
w_{+\tau}^\pi[f]\!\!&=&\!\!0&\text{on $ \Gamma_1$},\\
 w_{-\tau}^\pi[f]\!\!&=&\!\!0&\text{on $\G_{-1}$}.
\end{array}
\right.
\end{equation}
 As $(f_*,h_*)\in \big(h^{3+\alpha}(\s)\big)^2\cap S_\sigma$, a density argument together with \cite[Theorem 9.3]{ADN64} implies that 
 the transformed potentials $(v_+^*,v_-^*)$  satisfy
 \begin{equation}\label{adre}
   \frac{\tr_0\p_yv_-^*}{f_*-d}-\frac{\tr_0\p_yv_+^*}{h_*-f_*}\in h^{2+\alpha}(\s),
 \end{equation}
 and therefore $(w_{+\tau}^\pi[f],w_{-\tau}^\pi[f])$ is well-defined, cf. Corollary \ref{C:1}.
For $\tau=1$ we recover the   leading order part  $\p_{f}\Phi_{1,1}^{\pi*}= \p_{f}\Phi_1^\pi(f_*,h_*)$ of $\p_{f}\Phi_1(0,(f_*,h_*)) $, while $\p_{f}\Phi_{1,0}^{\pi*} $ is the above mentioned Dirichlet-Neumann map.
Our method uses localization techniques based on a suitable partition of unity allowing us to keep up the setting of periodic functions.\bigskip

\noindent{\bf Partition of unity.}
For each $p\geq 3$ there exists a family of functions  $\{\Pi_j^p\}_{1\leq j\leq 2^{p+1}}\subset  C^\infty(\s,[0,1])$ such that
\begin{itemize}
\item[$(i)$] $\supp \Pi_j^p=\cup_{n\in\Z} \big(2\pi n+ I_j^p\big)$, where $I_j^p:=[j-5/3,j-1/3] \pi/2^p;$
 \item[$(ii)$] $\sum_{j=1}^{2^{p+1}}\Pi_j^p=1 $ in $C(\s)$.
\end{itemize}
The center  of the interval $I_j^p$ for $1\leq j\leq 2^{p+1}$ is denoted by $x_j^p:=(j-1)\pi/2^p$. 
We call the family $\{\Pi_j^p\}_{1\leq j\leq 2^{p+1}}$ a {\em $p$-partition of unity}.
\medskip

Moreover, let $\{\chi_j^p\}_{1\leq j\leq 2^{p+1}}\subset  C^\infty(\s,[0,1])$ be a corresponding family of functions such that 
\begin{itemize}
\item[$(a)$] $\supp \chi_j^p=\cup_{n\in\Z} \big(2\pi n+ J_j^p\big)$ and $I_j^p\subset J_j^p$;
 \item[$(b)$] $\chi_j^p=1$ on $I_j^p$.
\end{itemize}
We can achieve that $J_j^p=I_{j-1}^p\cup I_{j}^p\cup I_{j+1}^p,$ where  $I_0^p:=-2\pi +I_{2^{p+1}}^p$ and $I_{2^{p+1}+1}^p:=2\pi+I_1^p$, so that $I_j^p$ and $J_j^p$ have the same center $x_j^p.$\medskip

The following remark is a simple exercise.

\begin{rem}\label{R:1}
 Given $k, p\in\N$ with $ p\geq 3$ and $\alpha\in(0,1),$ the mapping
$$f\mapsto \max_{1\leq j\leq 2^{p+1}} \|\Pi_j^p f\|_{k+\alpha}$$
defines a norm on $h^{k+\alpha}(\s)$ which is  equivalent to the $\|\cdot\|_{k+\alpha}$- norm. 
\end{rem}

The following  perturbation type result lies at the core  of our analysis.

\begin{thm}\label{T1} 
Let $\sigma>0$ be such that  $(f_0,h_0)\in S_\sigma$ and  let    $\mu>0$ and  $\alpha'\in(0,\alpha)$ be given.
Then, given $(f_*,h_*)\in  \big(h^{3+\alpha}(\s)\big)^2\cap S_\sigma$, there exist an integer $p\geq 3$, a $p$-partition of unity $\{\Pi_j^p\}_{1\leq j\leq 2^{p+1}}$, and a constant $K_2=K_2(p)$, and for each $\tau\in[0,1]$ and $1\leq j\leq 2^{p+1}$ there 
are bounded operators $\bA_{\tau,j}\in\kL\big(h^{2+\alpha}(\s), h^{1+\alpha}(\s)\big)$
 such that 
 \begin{equation}\label{DE}
  \|\Pi_j^p\p_f\Phi^{\pi*}_{1,\tau}[f]-\bA_{\tau,j}[\Pi^p_j f]\|_{1+\alpha}\leq \mu \|\Pi_j^pf\|_{2+\alpha}+K_2\|f\|_{2+\alpha'}
 \end{equation}
 for all $f\in h^{2+\alpha}(\s)$. The operators $\bA_{\tau,j}$ are defined  by the formula
  \begin{equation}\label{HOj}
 \bA_{\tau,j}[f]:=-\frac{\tau k}{\mu_-}\Big( \frac{2f'_{0}\tr_0\p_yv_-^*}{f_*-d} - \tr_0\p_xv_-^*\Big)\Big|_{x_j^p}f'  - \frac{k}{\mu_-}\Big(\frac{1+f_*'^2}{ f_*-d}\Big|_{x_j^p}\tr_0 \p_yw_{-\tau}^{\pi,j}[f]-f'_0(x_j^p)\tr_0 \p_xw_{-\tau}^{\pi,j}[f]\Big),
 \end{equation}
 where 
 $( w_{+\tau}^{\pi,j}[f], w_{-\tau}^{\pi,j}[f])$ denotes the solution to the problem
 \begin{equation}\label{HO1b}
\left\{
\begin{array}{rllllll}
\A^\pi_{0,j}(f_*,h_*) w_{+\tau}^{\pi,j}[f]\!\!&=&\!\!\tau A_+f'' &\text{in $ \Omega_+$},\\
\A^\pi_{0,j}(f_*) w_{-\tau}^{\pi,j}[f]\!\!&=&\!\!\tau A_-f''&\text{in $ \Omega_-$},\\
 \B_j(f_*,h_*)w_{+\tau}^{\pi,j}[f]-\B_j(f_*)w_{-\tau}^{\pi,j}[f]\!\!&=&\!\!\tau B f'&\text{on $ \Gamma_{0}$},\\
w_{+\tau}^{\pi,j}[f]-w_{-\tau}^{\pi,j}[f]\!\!&=&\!\!-\big[\Delta_\rho+(1-\tau)\Delta_A] f&\text{on $\Gamma_0$},\\
w_{+\tau}^{\pi,j}[f]\!\!&=&\!\!0&\text{on $ \Gamma_1$},\\
 w_{-\tau}^{\pi,j}[f]\!\!&=&\!\!0&\text{on $\G_{-1}$},
\end{array}
\right.
\end{equation}
and
\begin{align*}
&A_+:=\cfrac{ \tr_0\p_yv_+^*}{h_* -f_* }\Big|_{x_j^p},\qquad A_-:=\cfrac{  \tr_0\p_yv_-^* }{f_*-d}\Big|_{x_j^p},\qquad \Delta_\rho:=g(\rho_--\rho_+),\qquad \Delta_A:=A_+-A_-,\\
&B:=\frac{ k}{\mu_-}\Big(\cfrac{2f_*'\tr_0\p_yv_-^*}{f_*-d}-\tr_0\p_xv_-^*\Big)\Big|_{x_j^p}- \frac{k}{\mu_+}\Big(\cfrac{2f_*'\tr_0\p_yv_+^*}{h_* -f_* }-\tr_0\p_xv_+^*\Big)\Big|_{x_j^p}.
\end{align*}
Moreover,  $\A^\pi_{0,j}(f_*,h_*),$ $ \A^\pi_{0,j}(f_*),$ $ \B_j(f_*,h_*),$ and $\B_j(f_*)$ are the operators obtained from  $\A^\pi_0 (f_*,h_*), $ $\A^\pi_0 (f_*),$  $\B (f_*,h_*),$ and $\B (f_*)$, respectively, when evaluating their coefficients at  $x_j^p$.
\end{thm}
\begin{proof} Let  $\alpha'\in(0,\alpha)$ and $\mu>0$ be fixed. Given $p\geq 3$, a $p$-partition of unity $\{\Pi_j^p\}_{1\leq j\leq 2^{p+1}}$, and $\tau\in[0,1]$ we decompose the difference 
$\Pi_j^p\p_f\Phi^{\pi*}_{1,\tau}[f]-\bA_{\tau,j}[\Pi^p_j f]=T_1+T_2+T_3$ by setting
\begin{align*}
 &T_1:=\frac{\tau k}{\mu_-}\Big( \frac{2f_*'\tr_0\p_yv_-^*}{f_*-d} - \tr_0\p_xv_-^*\Big)\Big|_{x_j^p}(\Pi^p_j f)' -\frac{\tau k}{\mu_-}\Big( \frac{2f_*'\tr_0 \p_yv_-^*}{ f_*-d} -\tr_0 \p_x v_-^*  \Big)\Pi^p_j f',\\
 &T_2:=   \frac{k}{\mu_-}\frac{1+f_*'^2}{f_*-d}\Big|_{x_j^p}\tr_0 \p_yw_{-\tau}^{\pi,j}[\Pi^p_jf]-\frac{k}{\mu_-}\frac{1+f_*'^2 }{f_* -d}\Pi^p_j\tr_0 \p_yw_{-\tau}^\pi[f],\\
 &T_3:=\frac{k}{\mu_-}f_*'\Pi^p_j\tr_0 \p_xw_{-\tau}^\pi[f]-\frac{k}{\mu_-}f_*'(x_j^p)\tr_0 \p_xw_{-\tau}^{\pi,j}[\Pi^p_jf]
\end{align*}
and estimate each of these terms separately. 
In the following we shall denote  constants which are independent of  $p$ (and, of course, of $f\in h^{2+\alpha}(\s)$, $\tau\in[0,1],$ and $j\in\{1,\ldots, 2^{p+1}\}$)  by $C$.  \medskip

\noindent{\bf The estimate for $T_1$.} 
Noticing that $\|\Pi_j^p f''\|_{\alpha}\leq \|(\Pi_j^p f)''\|_{\alpha}+K\| f\|_{2}$ and using the fact that $\chi_j^p\Pi_j^p=\Pi_j^p$, we have
\begin{align*}
 \|T_1\|_{1+\alpha}\leq & 
 K\|f\|_{1+\alpha}+C\Big\|\chi_j^p\Big[\Big( \frac{2f_*'\tr_0\p_yv_-^*}{f_*-d} - \tr_0\p_xv_-^*\Big)\Big|_{x_j^p}-\Big( \frac{2f_*'\tr_0 \p_yv_-^*}{ f_*-d} -\tr_0 \p_x v_-^*  \Big)\Big]\Pi_j^p f'\Big\|_{1+\alpha}\\
 \leq &C\Big[\Big\|\chi_j^p\Big( \frac{f_*'\tr_0\p_yv_-^*}{f_*-d}\Big|_{x_j^p} - \frac{f_*'\tr_0\p_yv_-^*}{f_*-d}\Big)\Big\|_{0}+ \|\chi_j^p (  \tr_0 \p_x v_-^* -  \p_x v_-^*(x_j^p,0) \|_{0}\Big]\|\Pi_j^p f\|_{2+\alpha}\\
 &+K\|f\|_{2}.
\end{align*}
Choosing $p$ sufficiently large and using the uniform continuity of  the functions in the square brackets, we are led   to the estimate
\begin{align}
 \|T_1\|_{1+\alpha}
 \leq &(\mu/3)\|\Pi_j^p f\|_{2+\alpha}+K\|f\|_{2}.\label{T1_D}
\end{align}

\noindent{\bf The estimate for $T_2$ and $T_3$.} The terms  $T_2$ and $T_3$ are estimated in a similar way. However, due to the nonlocal character of these expressions, the arguments are more
involved than in the previous step.
It is easy to see that
\begin{align}
\|T_3\|_{1+\alpha}\leq &  C\|f_*'\Pi^p_j \tr_0 \p_xw_{-\tau}^\pi[f]-f_*'(x_j^p)\tr_0 \p_xw_{-\tau}^{\pi,j}[\Pi^p_jf]\|_{1+\alpha}\nonumber\\
\leq& C\|\chi_j^pf_*' \tr_0 \p_x\big(\Pi^p_j w_{-\tau}^\pi[f]\big)-f_*'(x_j^p)\tr_0 \p_xw_{-\tau}^{\pi,j}[\Pi^p_jf]\|_{1+\alpha}+K\|f\|_{2+\alpha'}\nonumber\\
\leq &C\|\chi_j^pf_*' \tr_0 \p_x\big(\Pi^p_j w_{-\tau}^\pi[f]\big)-\chi_j^pf_*'(x_j^p)\tr_0 \p_xw_{-\tau}^{\pi,j}[\Pi^p_jf]\|_{1+\alpha}\nonumber\\
&+ C\|(1-\chi_j^p)f_*'(x_j^p)\tr_0 \p_xw_{-\tau}^{\pi,j}[\Pi^p_jf]\|_{1+\alpha}+K\|f\|_{2+\alpha'}\nonumber\\
=:&T_{31}+T_{32}+K\|f\|_{2+\alpha'},\label{T3_D12}
\end{align}
where we used  once more the fact that $\chi_j^p\Pi_j^p=\Pi_j^p$.
We first note that
\begin{align*}
 T_{32}\leq &C\|(1-\chi_j^p) \tr_0 \p_xw_{-\tau}^{\pi,j}[\Pi^p_jf]\|_{1+\alpha}\leq K\|f\|_{2+\alpha'}+C\|  (1-\chi_j^p)w_{-\tau}^{\pi,j}[\Pi^p_jf] \|_{2+\alpha}^{\0_-}.
\end{align*}
The pair $(u_+,u_-):=(1-\chi_j^p)\big(w_{+\tau}^{\pi,j}[\Pi^p_jf],w_{-\tau}^{\pi,j}[\Pi^p_jf] \big)$
is the solution to
\begin{equation*} 
\left\{
\begin{array}{rllllll}
\A^\pi_{0,j}(f_*,h_*)u_+\!\!&=&\!\!\tau A_+(1-\chi_j^p)(\Pi_j^pf)''-(\chi_j^p)''w_{+\tau}^{\pi,j}[\Pi^p_jf]\\
&& -2(\chi_j^p)'\p_xw_{+\tau}^{\pi,j}[\Pi^p_jf]+\cfrac{2f_*'}{ h_*-f_*}\Big|_{x_j^p}(\chi_j^p)'\p_yw_{+\tau}^{\pi,j}[\Pi^p_jf]&\text{in $ \Omega_+$},\\
\A^\pi_{0,j}(f_*) u_-\!\!&=&\!\!\tau A_-(1-\chi_j^p)(\Pi_j^pf)''-(\chi_j^p)''w_{-\tau}^{\pi,j}[\Pi^p_jf]\\
&&-2(\chi_j^p)'\p_xw_{-\tau}^{\pi,j}[\Pi^p_jf]+\cfrac{2f_*'}{ f_*-d}\Big|_{x_j^p}(\chi_j^p)'\p_yw_{-\tau}^{\pi,j}[\Pi^p_jf]&\text{in $ \Omega_-$},\\
 \B_j(f_*,h_*)u_+-\B_j(f_*)u_-\!\!&=&\!\!\tau B (1-\chi_j^p)(\Pi_j^pf)'+k\mu_+^{-1} (\chi_j^p)'f_*'(x_j^p)\tr_0w_{+\tau}^{\pi,j}[\Pi^p_jf]\\
 &&-k\mu_-^{-1} (\chi_j^p)'f_*'(x_j^p)\tr_0 w_{-\tau}^{\pi,j}[\Pi^p_jf])&\text{on $ \Gamma_{0}$},\\
u_+-u_-\!\!&=&\!\!-\big[\Delta_\rho+(1-\tau)\Delta_A] (1-\chi_j^p)\Pi_j^pf&\text{on $\Gamma_0$},\\
u_+\!\!&=&\!\!0&\text{on $ \Gamma_1$},\\
 u_-\!\!&=&\!\!0&\text{on $\G_{-1}$},
\end{array}
\right.
\end{equation*}
and, because of $(1-\chi_j^p)\Pi_j^p=0,$ we end up with
\begin{align}
 T_{32}\leq  K\|f\|_{2+\alpha'}+C\|  u_- \|_{2+\alpha}^{\0_-}\leq K\|f\|_{2+\alpha'}.\label{T3_D1}
\end{align}
The remaining term $T_{31}$ can be estimated as 
\begin{align}
 T_{31}\leq &C\|\chi_j^p\big(f_*' -f_*'(x_j^p)\big)\tr_0 \p_xw_{-\tau}^{\pi,j}[\Pi^p_jf]\|_{1+\alpha}+C\|\chi_j^p \tr_0 \p_x\big(\Pi^p_j w_{-\tau}^\pi[f] - w_{-\tau}^{\pi,j}[\Pi^p_jf]\big)\|_{1+\alpha}\nonumber\\
 \leq&(\mu/6)\|\Pi^p_jf\|_{2+\alpha}+K\|f\|_{2+\alpha'}+C\|\chi_j^p\tr_0 \p_x\big( \Pi^p_j w_{-\tau}^\pi[f] -w_{-\tau}^{\pi,j}[\Pi^p_jf]\big)\|_{1+\alpha}\label{T3_D1a}
\end{align}
if $p$ is sufficiently large.
We are left to estimate the last term in \eqref{T3_D1a}. We note that
\begin{align}
 \|\chi_j^p \tr_0 \p_x\big( \Pi^p_j w_{-\tau}^\pi[f] -w_{-\tau}^{\pi,j}[\Pi^p_jf]\big)\|_{1+\alpha}\leq & K\|f\|_{2+\alpha'}+\|\chi_j^p\big( w_{-\tau}^{\pi,j}[\Pi^p_jf]-\Pi^p_j w_{-\tau}^\pi[f] \big)\|_{2+\alpha}^{\0_-}.\label{T3_D1b}
\end{align}
We now infer from the fact that the pair $$(z_+,z_-):= \chi_j^{p}\big((w_{+\tau}^{\pi,j}[\Pi^p_jf],w_{-\tau}^{\pi,j}[\Pi^p_jf] )-\Pi^p_j(w_{+\tau}^\pi[f],w_{-\tau}^\pi[f] )\big)$$
solves the diffraction problem
 \begin{equation*} 
\left\{
\begin{array}{rllllll}
\A^\pi_0 (f_*,h_*) z_+\!\!&=&\!\!(\A^\pi_0 (f_*,h_*)-\A^\pi_{0,j}(f_*,h_*))[\chi_j^pw_{+\tau}^{\pi,j}[\Pi^p_jf]]\\
&&+\tau A_+\chi_j^p(\Pi_j^pf)''-\frac{\tau\tr_0\p_yv_+^*}{h_*-f_*}\Pi_j^pf''\\
&&+(\chi_j^p)''w_{+\tau}^{\pi,j}[\Pi^p_jf]-(\Pi_j^p)''w_{+\tau}^\pi[f]\\
&&+2(\chi_j^p)'\p_xw_{+\tau}^{\pi,j}[\Pi^p_jf]-2(\Pi_j^p)'\p_xw_{+\tau}^\pi[f] \\
&&+\frac{2f_*'}{h_*-f_*}(\Pi_j^p)'\p_yw_{+\tau}^\pi[f]-\frac{2f_*'}{h_*-f_*}\Big|_{x_j^p}(\chi_j^p)'\p_yw_{+\tau}^{\pi,j}[\Pi^p_jf]&\text{in $ \Omega_+$},\\
\A^\pi_0 (f_*) z_-\!\!&=&\!\!(\A^\pi_0 (f_*)-\A^\pi_{0,j}(f_*))[\chi_j^pw_{-\tau}^{\pi,j}[\Pi^p_jf]]\\
&&+\tau A_-\chi_j^p(\Pi_j^pf)''-\frac{\tau\tr_0\p_yv_-^*}{f_*-d}\Pi_j^pf''\\
&&+(\chi_j^p)''w_{-\tau}^{\pi,j}[\Pi^p_jf]-(\Pi_j^p)''w_{-\tau}^\pi[f]\\
&&+2(\chi_j^p)'\p_xw_{-\tau}^{\pi,j}[\Pi^p_jf]-2(\Pi_j^p)'\p_xw_{-\tau}^\pi[f] \\
&&+\frac{2f_*'}{f_*-d}(\Pi_j^p)'\p_yw_{-\tau}^\pi[f]-\frac{2f_*'}{f_*-d}\Big|_{x_j^p}(\chi_j^p)'\p_yw_{-\tau}^{\pi,j}[\Pi^p_jf]&\text{in $ \Omega_-$},\\
 \B (f_*,h_*)z_+-\B (f_*)z_-\!\!&=&\!\!(\B(f_*,h_*)-\B_j(f_*,h_*))[\chi_j^pw_{+\tau}^{\pi,j}[\Pi^p_jf]]\\
 &&-(\B(f_*)-\B_j(f_*))[\chi_j^pw_{-\tau}^{\pi,j}[\Pi^p_jf]]+\tau B \chi_j^p(\Pi_j^pf)'\\
 &&+\tau\Pi_j^p\p_f\B^\pi(f_*,h_*)[f]v_+^*-\tau\Pi_j^p\p_f\B^\pi(f_*)[f]v_-^*\\
&&-k\mu_+^{-1}\big( (\chi_j^p)'f_*'(x_j^p)\tr_0w_{+\tau}^{\pi,j}[\Pi^p_jf]-(\Pi_j^p)'f_*'\tr_0w_{+\tau}^\pi[f]\big)\\
 &&+k\mu_-^{-1}\big( (\chi_j^p)'f_*'(x_j^p)\tr_0 w_{-\tau}^{\pi,j}[\Pi^p_jf]-(\Pi_j^p)'f_*'\tr_0w_{-\tau}^\pi[f]\big)&\text{on $ \Gamma_{0}$},\\
z_+-z_-\!\!&=&\!\!(1-\tau)\chi_j^p\Big(\frac{\tau\tr_0\p_yv_+^*}{h_*-f_*}-\frac{\tau\tr_0\p_yv_-^*}{f_*-d}-\Delta_A\Big)\Pi_j^p f&\text{on $\Gamma_0$},\\
z_+\!\!&=&\!\!0&\text{on $ \Gamma_1$},\\
 z_-\!\!&=&\!\!0&\text{on $\G_{-1}$}
\end{array}
\right.
\end{equation*}
 and the Schauder estimate  \eqref{Schauder}  that
 \begin{align*}
\|z_-\|_{2+\alpha}^{\0_-}\leq &  C\Big(\|\big(\A^\pi_0(f_*,h_*)-\A^\pi_{0,j}(f_*,h_*)\big)
\big[\chi_j^p w_{+\tau}^{\pi,j}[\Pi^p_jf]\big]\|_\alpha+\Big\|\chi_j^p\Big(\cfrac{\tr_0\p_yv_+^*}{h_*-f_*}-A_+\Big)\Big\|_0\|\Pi_j^pf''\|_\alpha\\
 &\quad+\|\big(\A^\pi_0(f_*)-\A^\pi_{0,j}(f_*)\big)\big[\chi_j^p w_{-\tau}^{\pi,j}[\Pi^p_jf]\big]\|_\alpha
 +\Big\|\chi_j^p\Big(\cfrac{\tr_0\p_yv_-^*}{f_*-d}-A_-\Big)\Big\|_0\|\Pi_j^pf''\|_\alpha\\
 &\quad+\|\big(\B(f_*,h_*)-\B_j(f_*,h_*)\big)\big[\chi_j^p w_{+\tau}^{\pi,j}[\Pi^p_jf]\big]\|_{1+\alpha}\\
 &\quad+\|\big(\B(f_*)-\B_j(f_*)\big)\big[\chi_j^p w_{-\tau}^{\pi,j}[\Pi^p_jf]\big]\|_{1+\alpha}\\
 &\quad+ \|\tau B \chi_j^p(\Pi_j^pf)'+\tau\Pi_j^p\p_f\B^\pi(f_*,h_*)[f]v_+^*-\tau\Pi_j^p\p_f\B^\pi(f_*)[f]v_-^*\|_{1+\alpha}\\
 &\quad+\Big\|\chi_j^p\Big(\frac{\tau\tr_0\p_yv_+^*}{h_*-f_*}-\frac{\tau\tr_0\p_yv_-^*}{f_*-d}-\Delta_A\Big)\Big\|_{0}\|\Pi_j^p f\|_{2+\alpha}\Big)\\
 &+K\|f\|_{2+\alpha'}.
 \end{align*}
Using the same arguments as when estimating $T_1$, we find for $p$ sufficiently large 
\[\|\chi_j^p\big( w_{-\tau}^{\pi,j}[\Pi^p_jf]-\Pi^p_j w_{-\tau}^\pi[f] \big)\|_{2+\alpha}^{\0_-}=\| z_-\|_{2+\alpha}^{\0_-}\leq (\mu/6)\|\Pi_j^p f\|_{2+\alpha}+K\|f\|_{2+\alpha'},\]
which yields, together with \eqref{T3_D1}-\eqref{T3_D1b}, that
\begin{align}
 \|T_3\|_{1+\alpha}
 \leq &(\mu/3)\|\Pi_j^p f\|_{2+\alpha}+K\|f\|_{2+\alpha'}.\label{T3_D}
\end{align}
Similar arguments show that 
\begin{align}
 \|T_2\|_{1+\alpha}
 \leq &(\mu/3)\|\Pi_j^p f\|_{2+\alpha}+K\|f\|_{2+\alpha'}.\label{T2_D}
\end{align}
Gathering \eqref{T1_D}, \eqref{T3_D}, and \eqref{T2_D}, we have established the desired estimate \eqref{DE}.
\end{proof}\bigskip

\noindent{\bf Fourier analysis: The symbol of $\bA_{\tau,j}$.} 
In this step we use Fourier analysis arguments and ODE techniques  to represent  the operators $\bA_{\tau,j}$ introduced in \eqref{HOj} as Fourier multipliers.
Subsequently we will use a Marcinkiewicz type Fourier multiplier theorem 
to prove that all these  Fourier multipliers $\bA_{\tau,j}$ are generators of strongly continuous and analytic semigroups, cf. Lemma \ref{L:2}.\medskip

We fix now  $\tau\in[0,1]$  and $p,j\in\N$ with $p\geq 3$ and $1\leq j\leq 2^{p+1}$  arbitrary.
Given $f\in h^{2+\alpha}(\s)$, we consider its Fourier expansion
\[f=\sum_{m\in\Z}f_me^{imx}\]
and look for the corresponding expansion of $\bA_{\tau,j}[f].$
In view of \eqref{HOj}, we needed to determine the  expansions for $\tr_0\nabla( w_{+\tau}^{\pi,j}[f], w_{-\tau}^{\pi,j}[f]).$
Hence, we let 
\begin{align}\label{AEW}
  ( w_{+\tau}^{\pi,j}[f], w_{-\tau}^{\pi,j}[f])=\sum_{m\in\Z}\big(A_m^+, A_m^-\big)f_m e^{imx}
\end{align}
with   functions $(A_m^+, A_m^-)=(A_m^+, A_m^-)(y)$ still to be determined.
Recalling \eqref{eqs} and the formulae in   Appendix \ref{App}, we   use the summation convention and write
\begin{align*}
 &\A^\pi_{0,j}(f_*,h_*)=c_{pl}^+\p_{pl},\qquad \A^\pi_{0,j}(f_*)=c_{pl}^-\p_{pl},\qquad \B_j(f_*,h_*)=\beta_l^+\tr_0\p_l,\qquad \B_j(f_*)=\beta_l^-\tr_0\p_l,
\end{align*} 
and we set
\begin{align*}
&a_\pm:=c_{12}^\pm/c_{22}^\pm,\quad b_\pm:=1/c_{22}^\pm,\quad D_\pm:=\sqrt{b_\pm-(a_\pm)^2}
\end{align*}
noticing that $b_\pm-(a_\pm)^2>0$.
From \eqref{AEW} and \eqref{HO1b} we deduce that for  fixed $m\in\Z,$ the pair $(A_m^+, A_m^-)$ solves the following problem 
 \begin{equation}\label{Hobit}
\left\{
\begin{array}{rllllll}
(A_m^+)''+i2ma_+(A_m^+)'-b_+m^2A_m^+\!\!&=&\!\!-\tau b_+ A_+m^2  &\text{in $ 0<y<1$},\\
(A_m^-)''+i2ma_-(A_m^-)'-b_-m^2A_m^-\!\!&=&\!\!-\tau b_-A_-m^2 &\text{in $ -1<y<0$},\\
 im \big(\beta_1^+ A_m^+(0)-\beta_1^- A_m^-(0)\big)+\beta_2^+ (A_m^+)'(0) -\beta_2^- (A_m^-)'(0)\!\!&=&\!\!i\tau m B,\\
A_m^+(0)-A_m^-(0)\!\!&=&\!\!-\big[\Delta_\rho+(1-\tau)\Delta_A\big],\\
A_m^+(1)\!\!&=&\!\!0,\\
 A_m^-(-1)\!\!&=&\!\!0,
\end{array}
\right.
\end{equation}
with constants $A_\pm, B, \Delta_\rho, \Delta_A$ defined in Theorem \ref{T1}.
We now set
\begin{align*}
 &\cos_\pm(y):=\cos(a_\pm my), && \sin_\pm(y):=\sin(a_\pm my),\\
 &\cosh_\pm(y):=\cosh(D_\pm my) ,&& \sinh_\pm(y):=\sinh(D_\pm my).
\end{align*}
With this notation, it is easy to verify that the general solutions to the first two equations of \eqref{Hobit} are given by the formula
\begin{equation}\label{BFor1}
 A_m^\pm:=u_m^\pm+i v_m^\pm
\end{equation}
with real parts 
\begin{align*}
 u_m^\pm(y)=&\xi_1^\pm\Big(\cos_\pm(y)\cosh_\pm(y)+\frac{a_\pm}{D_\pm}\sin_\pm(y)\sinh_\pm(y)\Big)+\frac{\xi_2^\pm}{D_\pm m}\cos_\pm(y)\sinh_\pm(y)\\
 &+\xi_3^\pm\Big(\sin_\pm(y)\cosh_\pm(y)-\frac{a_\pm}{D_\pm}\cos_\pm(y)\sinh_\pm(y)\Big)+\frac{\xi_4^\pm}{D_\pm m}\sin_\pm(y)\sinh_\pm(y)+\tau A_\pm
\end{align*}
and imaginary parts
\begin{align*}
 v_m^\pm(y)=&\xi_1^\pm\Big(-\sin_\pm(y)\cosh_\pm(y)+\frac{a_\pm}{D_\pm}\cos_\pm(y)\sinh_\pm(y)\Big)-\frac{\xi_2^\pm}{D_\pm m}\sin_\pm(y)\sinh_\pm(y)\\
 &+\xi_3^\pm\Big(\cos_\pm(y)\cosh_\pm(y)+\frac{a_\pm}{D_\pm}\sin_\pm(y)\sinh_\pm(y)\Big)+\frac{\xi_4^\pm}{D_\pm m}\cos_\pm(y)\sinh_\pm(y).
\end{align*}
The real constants $\{\xi_i^\pm\,:\, 1\leq i\leq 4\}$ are to be determined such that  the last four equations of \eqref{Hobit} are also satisfied by $(A_m^+, A_m^-)$.
It can be shown by explicit, but tedious computations that such constants can be uniquely determined to obtain $(A_m^+, A_m^-)$.
Since 
\[
\tr_0 \p_yw_{-\tau}^{\pi,j}[f]=\sum_{m\in\Z} (\xi_2^-+i\xi_4^-)f_me^{imx},\qquad \tr_0 \p_xw_{-\tau}^{\pi,j}[f]=\sum_{m\in\Z} im(\xi_1+\tau A_-+i\xi_3)f_me^{imx},
\]
and using the definition \eqref{HOj} together with the explicit expressions for  $\{\xi_i^\pm\,:\, 1\leq i\leq 4\}$, it follows
that $\bA_{\tau,j}$ is  the Fourier multiplier 
\begin{align*} 
 \bA_{\tau,j}\Big[\sum_m  f_me^{imx}\Big]=&\sum_m \lambda_m f_me^{imx}
 \end{align*}
with symbol $(\lambda_m)_{m\in\Z}$ defined by 
 \begin{align}
  \re\lambda_m
 :=& -  \Big(\frac{\tanh(D_+m)}{\beta_2^+D_+m}+\frac{\tanh(D_-m)}{\beta_2^-D_-m}\Big)^{-1}\Big[\Delta_\rho+\Delta_A+\frac{\tau A_-\cos(D_-m)}{\cosh(D_-m)}- \frac{\tau A_+\cos(D_+m)}{\cosh(D_+m)}\Big],\label{FMFR}\\
  \im\lambda_m:=& \frac{\tau k }{ \mu_-}\Big( \frac{ a_- \p_yv_-^*(x_j^p,0)}{a_-^2+D_-^2} +\p_xv_-^*(x_j^p,0)\Big)m\nonumber\\
  &+ \tau \Big(\frac{\tanh(D_+m)}{\beta_2^+D_+m}+\frac{\tanh(D_-m)}{\beta_2^-D_-m}\Big)^{-1} \Big[  \frac{ A_-\sin(D_-m)}{\cosh(D_-m)}+\frac{ A_+\sin(D_+m)}{\cosh(D_+m)}\Big]\nonumber\\
  &- \tau\Big(\frac{\tanh(D_+m)}{\beta_2^+D_+m}+\frac{\tanh(D_-m)}{\beta_2^-D_-m}\Big)^{-1} \frac{\tanh(D_+m)}{\beta_2^+D_+ } \big[ \beta_1^+A_+-\beta_1^-A_-- B\big].\label{FMFI}
 \end{align} 
We refrain from presenting here the detailed computations that are used to determine the constants $\{\xi_i^\pm\,:\, 1\leq i\leq 4\}$, \eqref{FMFR}, and \eqref{FMFI}, as they are quite long and only the outcome, that is,  the explicit 
formula for the symbol $(\lambda_m)_{m\in\Z}$, 
is of importance for the further analysis.
 
 We are now in the position to prove that the operators  $\bA_{\tau,j}$ are generators of strongly continuous and analytic semigroups.
To this end we will use a Marcinkiewicz type Fourier multiplier theorem \cite[Theorem 2.1]{JL12}, which generalizes a result from \cite{AB04} (see also \cite{EM09} for a similar result) and states that a Fourier multiplier
 \[\sum_{m\in\Z}f_me^{imx}\mapsto\sum_{m\in\Z}\Lambda_m f_me^{imx}\]
 belongs to $\kL(C^{r+\alpha}(\s), C^{s+\alpha}(\s))$, $r,s\in \N$ and $\alpha\in(0,1)$, provided that 
 \begin{align*}
  s_1:=\sup_{m\in\Z\setminus\{0\}}|m|^{s-r}|\Lambda_m|<\infty \qquad\text{and}\qquad s_2:=\sup_{m\in\Z\setminus\{0\}}|m|^{s-r+1}|\Lambda_{m+1}-\Lambda_m|<\infty.
 \end{align*}
Additionally, the norm of the Fourier multiplier as a bounded linear operator from $C^{r+\alpha}(\s)$ to $ C^{s+\alpha}(\s)$ can
be bounded in terms of the constants $s_1$ and $s_2$ alone. Bearing this result in mind we can derive the following lemma.

\bigskip

\begin{lemma}\label{L:2}
Let $\sigma>0$ be such that  $(f_0,h_0)\in S_\sigma$.
Then, there exist  constants  $\kappa_1\geq 1$ and   $\omega_1>0$ depending only on $\sigma,$  such that for  any    $(f_*,h_*)\in  \big(h^{3+\alpha}(\s)\big)^2\cap S_\sigma$, 
any $p$-partition of unity $\{\Pi_j^p\}_{1\leq j\leq 2^{p+1}}$ with $p\geq 3$ and any $\tau\in[0,1]$,
the operators  $\bA_{\tau,j}$,     $1\leq j\leq 2^{p+1}$,   defined by \eqref{HOj}  satisfy
 \begin{align}\label{13}
&\lambda-\bA_{\tau,j}\in{\rm Isom}(h^{2+\alpha}(\s),h^{1+\alpha}(\s)),\\[1ex]
\label{14}
& \kappa_1\|(\lambda-\bA_{\tau,j})[f]\|_{1+\alpha}\geq  |\lambda|\cdot\|f\|_{1+\alpha}+\|f\|_{2+\alpha}
\end{align}
for all $f\in h^{2+\alpha}(\s)$ and $\lambda\in\C$ with $\re \lambda\geq \omega_1$.
\end{lemma}

\begin{proof} 
 We write 
  \begin{align*}
  \re\lambda_m=-\mu_m-\nu_m,
 \end{align*}
 where
 \begin{align*}
\nu_m&:=    \tau\Big(\frac{\tanh(D_+m)}{\beta_2^+D_+m}+\frac{\tanh(D_-m)}{\beta_2^-D_-m}\Big)^{-1}\Big[\frac{ A_-\cos(D_-m)}{\cosh(D_-m)}- \frac{A_+\cos(D_+m)}{\cosh(D_+m)}\Big],\\
\mu_m&:=    \Big(\frac{\tanh(D_+m)}{\beta_2^+D_+m}+\frac{\tanh(D_-m)}{\beta_2^-D_-m}\Big)^{-1}\big(\Delta_\rho+\Delta_A\big).
 \end{align*}
Recalling the definition of $S_\sigma$, there exists a   constant  $C_0=C_0(\sigma)>1$ such that 
 \begin{equation}\label{plus}
C_0^{-1}\leq\Delta_\rho+\Delta_A\leq C_0.
\end{equation}
 In fact, it is not difficult to see that we can choose $C_0$ large enough to guarantee that
 \begin{equation}\label{pup}
 \begin{aligned}
  &C_0^{-1}\leq\frac{\mu_m}{|m|}\leq C_0,\quad m\in\Z\setminus\{0\},\, \tau\in[0,1]\,,\\
  &\sup_{m\in\Z\setminus\{0\}}\sup_{\tau\in[0,1]} \Big(|\nu_m|+\frac{|\re\lambda_m|+|\im\lambda_m|}{|m|}\Big)\leq C_0,\\
 & \sup_{m\in\Z\setminus\{0\}}\sup_{\tau\in[0,1]}\Big(|\re\lambda_{m+1}-\re\lambda_m|+|\im\lambda_{m+1}-\im\lambda_m|\Big)\leq C_0.
 \end{aligned}
 \end{equation}
 
Let now $\lambda=\zeta_1+i\zeta_2\in\C $  be given such that $\re\lambda=\zeta_1\geq\omega_1:= 2C_0.$
Since $\zeta_1-\re\lambda_m\geq C_0$, the (formal) inverse  $(\lambda-\bA_{\tau,j})^{-1}$  of $\lambda-\bA_{\tau,j}$ is the Fourier multiplier
 \begin{equation*}
 \sum_{m\in\Z}f_me^{imx}\mapsto\sum_{m\in\Z} \Lambda_m f_me^{imx}
\end{equation*}
  with symbol
 \[
 \Lambda_m:= (\lambda-\lambda_{m})^{-1}=\frac{\zeta_1-\re\lambda_m-i(\zeta_2-\im\lambda_m)}{(\zeta_1-\re\lambda_m)^2+(\zeta_2-\im\lambda_m)^2}.
 \]
\noindent{\bf Step 1.}  We first prove that $(\lambda-\bA_{\tau,j})^{-1}$ belongs to $\kL(C^{1+\alpha}(\s), C^{2+\alpha}(\s))$, the norm being independent of $p\geq 3, j\in\{1,\ldots, 2^{p+1}\}, \tau\in[0,1]$, and $\lambda\in\C$ with $ \re\lambda\geq\omega_1.$
More precisely, we show that\footnote{The operator $\bA_{\tau,j}$  depends 
on the partition $\{\Pi_j^p\}_{1\leq j\leq 2^{p+1}}$ through the middle point $x_j^p$ of the interval  $I_j^p$ only, and therefore $\bA_{\tau,j}$ makes sense when replacing $x_j^p$ by any $x\in\s$.
In view of this fact, when taking the supremum over ${x_j^p\in\s}$ in \eqref{cond1} we consider a larger set  than when taking the supremum over all middle point $x_j^p$ and all $p$-partitions. }
 \begin{align}\label{cond1}
\sup_{\zeta_1\geq \omega_1}\sup_{x_j^p\in\s} \sup_{\tau\in[0,1]}\sup_{m\in\Z^*}\big(|m\Lambda_m|+|m|^2|\Lambda_{m+1}-\Lambda_{m}|\big)<\infty.
 \end{align}
Note that \eqref{pup} implies
\begin{align}\label{EEE1}
\begin{aligned}
&|\Lambda_0|\leq C_0,\\
&|m|^2 |\Lambda_m|^2=&\frac{m^2}{(\zeta_1-\re\lambda_m)^2+(\zeta_2-\im\lambda_m)^2}\leq\frac{m^2}{(\zeta_1-\re\lambda_m)^2}\leq\frac{C_0^2 m^2}{m^2+C_0^4}\leq C_0^2.
\end{aligned}
\end{align}
Setting $R_m:=\zeta_1-\re\lambda_m$ and $I_m:=\zeta_2-\im\lambda_m,$ we have $R_m>0$ and
\begin{align*}
 |\Lambda_{m+1}-\Lambda_{m}|\leq T_1+T_2+T_3+T_4,
\end{align*}
where
\begin{align*}
 T_1:=\frac{R_mR_{m+1}|R_{m+1}-R_m|}{(R_m^2+I_m^2)(R_{m+1}^2+I_{m+1}^2)},\qquad T_2:=\frac{ |R_mI_{m+1}^2-R_{m+1}I_m^2|}{(R_m^2+I_m^2)(R_{m+1}^2+I_{m+1}^2)},\\
  T_3:=\frac{|I_mI_{m+1}||I_{m+1}-I_m|}{(R_m^2+I_m^2)(R_{m+1}^2+I_{m+1}^2)},\qquad T_4:=\frac{ |I_mR_{m+1}^2-I_{m+1}R_m^2|}{(R_m^2+I_m^2)(R_{m+1}^2+I_{m+1}^2)}.
\end{align*}
The estimates for $m^2T_3$ and $m^2T_4$ are similar to those for $m^2T_1$ and $m^2T_2$, and therefore we shall present only  those for the latter.
Recalling \eqref{pup} and \eqref{EEE1},  we get
\begin{align*}
 m^2T_1\leq m^2|\Lambda_m|\cdot|\Lambda_{m+1}|\cdot|R_{m+1}-R_m|\leq 2C_0^3
\end{align*}
and
\begin{align*}
 m^2 T_2\leq&m^2\frac{|R_m|(|I_{m+1}|+|I_m|)|I_{m+1}-I_m|+|I_m^2||R_{m+1}-R_m|}{(R_m^2+I_m^2)(R_{m+1}^2+I_{m+1}^2)}\leq  C_0m^2|\Lambda_{m+1}|\big( |\Lambda_m|+2|\Lambda_{m+1}|\big)\\
 \leq& 10 \, C_0^3.
\end{align*}
Proceeding similarly with  $m^2T_3$ and $m^2T_4$, we arrive at \eqref{cond1}. In view of \cite[Theorem 2.1]{JL12} we additionally know that  $(\lambda-\bA_{\tau,j})^{-1}\in\kL\big(C^{n+\alpha}(\s), C^{n+1+\alpha}(\s)\big)$ for all $n\in\N$, and since the closure of $C^{n+1+\alpha}(\s)$ in 
$C^{n+\alpha}(\s)$ is exactly  $h^{n+\alpha}(\s)$, a density argument leads us to the desired property \eqref{13}.
Moreover, our arguments  show  there exists a constant $\kappa_1$ depending only on $C_0$ such that   
 \begin{align}\label{15}
& \kappa_1\|(\lambda-\bA_{\tau,j})[f]\|_{1+\alpha}\geq \|f\|_{2+\alpha}\qquad\text{for all $f\in h^{2+\alpha}(\s)$ and   $\re \lambda\geq \omega_1$.}
\end{align}
\noindent{\bf Step 2.}
We are left to prove that we can choose $\kappa_1$  large enough to guarantee that 
 \begin{align*}
& \kappa_1\|(\lambda-\bA_{\tau,j})[f]\|_{1+\alpha}\geq  |\lambda|\cdot\|f\|_{1+\alpha} \qquad\text{for all $f\in h^{2+\alpha}(\s)$ and   $\re \lambda\geq \omega_1$.}
\end{align*}
For this it suffices to show that   
\begin{align}\label{cond2}
\sup_{\zeta_1\geq \omega_1}\sup_{x_j^p\in\s} \sup_{\tau\in[0,1]}\sup_{m\in\Z^*}|\lambda|\big(|\Lambda_m|+|m|\cdot|\Lambda_{m+1}-\Lambda_{m}|\big)<\infty.
 \end{align}
Using \eqref{pup}, we  first note that 
\begin{align*}
\frac{\zeta_1^2}{(\zeta_1-\re\lambda_m)^2+(\zeta_2-\im\lambda_m)^2}\leq \frac{\zeta_1^2}{(\zeta_1-\re\lambda_m)^2}\leq \frac{\zeta_1^2}{(\zeta_1/2)^2}=4.
\end{align*}
Additionally, \eqref{pup} and \eqref{EEE1} lead us to
\begin{align*}
\frac{\zeta_2^2}{(\zeta_1-\re\lambda_m)^2+(\zeta_2-\im\lambda_m)^2}\leq& \frac{2(\zeta_2-\im\lambda_m)^2+2(\im\lambda_m)^2}{(\zeta_1-\re\lambda_m)^2+(\zeta_2-\im\lambda_m)^2}\leq 2\Big[1+\frac{|\im\lambda_m|^2}{m^2}\big(m^2|\Lambda_m|^2\big)\Big]\\
\leq&2(1+C_0^4),
\end{align*}
and we thus have established that 
\begin{align}\label{EEE4}
\sup_{\zeta_1\geq \omega_1}\sup_{x_j^p\in\s} \sup_{\tau\in[0,1]}\sup_{m\in\Z^*}|\lambda|\cdot|\Lambda_m|\leq 2(2+C_0^2).
 \end{align}
To finish the proof, we recall that $|\Lambda_{m+1}-\Lambda_{m}|\leq T_1+T_2+T_3+T_4$.
Combining \eqref{pup}, \eqref{EEE1}, and \eqref{EEE4}, we find that 
\begin{align*}
 |\lambda m|T_1\leq |\lambda m|\cdot|\Lambda_m|\cdot|\Lambda_{m+1}|\cdot|R_{m+1}-R_m|\leq 4C_0^2(2+C_0^2)
\end{align*}
and 
\begin{align*}
 |\lambda m|T_2\leq& C_0|\lambda m|\cdot |\Lambda_{m+1}|\big( |\Lambda_m|+2|\Lambda_{m+1}|\big)\leq 10C_0^2(2+C_0^2).
\end{align*}
Proceeding similarly with the remaining terms $|\lambda m|T_3$ and $|\lambda m|T_4$, we obtain \eqref{cond2}.
Together with \eqref{15} we deduce \eqref{14}, and the proof is completed.
\end{proof}

With these preliminary results, we can now  prove that the assumptions of Lemma \ref{L:Ad1} hold. In this way we establish  \eqref{Claim1}.

\begin{thm}\label{T:c1}
Let $(f_0,h_0)) \in\V$ and $b_0\in h^{2+\alpha}(\s)$ be given such that the first inequality in \eqref{CT1}  is satisfied. Then 
 \begin{align*}
 -\p_{f}\Phi_1(0,(f_0,h_0))\in\kH\big(h^{2+\alpha}(\s), h^{1+\alpha}(\s)\big).
\end{align*}
\end{thm}
\begin{proof} 
Recalling \eqref{Claim4} (with $(f_*,h_*)=(f_0,h_0)$) and the perturbation result \cite[Theorem I.1.3.1 (ii)]{Am95}, we are left  to show that $-\p_f\Phi^\pi_1(f_0,h_0)\in\kH\big(h^{2+\alpha}(\s), h^{1+\alpha}(\s)\big).$ 
To this end, it suffices to verify the assumptions of Lemma \ref{L:Ad1} for some $\sigma>0$.

 Let $\sigma>0$ be such that  $(f_0,h_0)\in S_\sigma$,  let $\kappa_1$ and $\omega_1$ be the constants found in Lemma \ref{L:2}, and pick  $\alpha'\in(0,\alpha)$.
 Given $(f_*,h_*)\in  \big(h^{3+\alpha}(\s)\big)^2\cap S_\sigma$, Theorem \ref{T1} implies the existence of a  $p$-partition of unity  $\{\Pi_j^p\}_{1\leq j\leq 2^{p+1}}$  and of a constant $K_2$ such that  
 \begin{equation}\label{DEa}
  \|\Pi_j^p\p_f\Phi^{\pi*}_{1,\tau}[f]-\bA_{\tau,j}[\Pi^p_j f]\|_{1+\alpha}\leq  \frac{1}{2\kappa_1} \|\Pi_j^pf\|_{2+\alpha}+K_2\|f\|_{2+\alpha'}
 \end{equation}
 for all $f\in h^{2+\alpha}(\s)$, $\tau\in[0,1],$ and $j\in\{1,\ldots, 2^{p+1}\}$.
In view of \eqref{14}, we get
\begin{align}\label{DEb}
& \kappa_1\|(\lambda-\bA_{\tau,j})[\Pi_j^pf]\|_{1+\alpha}\geq  |\lambda|\cdot\|\Pi_j^pf\|_{1+\alpha}+\|\Pi_j^pf\|_{2+\alpha}
\end{align}
for all $\lambda\in\C$ with $\re \lambda \geq\omega_1$, $\tau\in[0,1]$, $f\in h^{2+\alpha}(\s)$, and $j\in\{1,\ldots, 2^{p+1}\}.$
Combining \eqref{DEa} and \eqref{DEb} gives
\begin{align*}
 \kappa_1\|\Pi_j^p(\lambda-\p_f\Phi^{\pi*}_{1,\tau})[f]\|_{1+\alpha}\geq& \kappa_1\|\lambda\Pi_j^pf-\bA_{\tau,j}[\Pi^p_j f]\|_{1+\alpha}-\kappa_1\|\bA_{\tau,j}[\Pi^p_j f]-\Pi_j^p\p_f\Phi^{\pi*}_{1,\tau}[f]\|_{1+\alpha}\\
 \geq&   |\lambda|\cdot\|\Pi_j^pf\|_{1+\alpha}+\|\Pi_j^pf\|_{2+\alpha}-\frac{1}{2} \|\Pi_j^pf\|_{2+\alpha}-\kappa_1K_2\|f\|_{2+\alpha'},
\end{align*}
hence
\begin{align*}
 \kappa_1K_2\|f\|_{2+\alpha'}+\kappa_1\|\Pi_j^p(\lambda-\p_f\Phi^{\pi*}_{1,\tau})[f]\|_{1+\alpha}\geq&   |\lambda|\cdot\|\Pi_j^pf\|_{1+\alpha}+\frac{1}{2}\|\Pi_j^pf\|_{2+\alpha} 
\end{align*}
for all $\re \lambda \geq\omega_1$, $f\in h^{2+\alpha}(\s)$, $\tau\in[0,1]$,  and $j\in\{1,\ldots, 2^{p+1}\}.$
Together with Remark \ref{R:1}, we conclude there exists $\kappa_1'>0$ such that 
\begin{align*}
 \kappa_1' K_2\|f\|_{2+\alpha'}+\kappa_1\|(\lambda-\p_f\Phi^\pi_{1,\tau})[f]\|_{1+\alpha}\geq&   |\lambda|\cdot\|f\|_{1+\alpha}+\frac{1}{2}\|f\|_{2+\alpha} 
\end{align*}
for all $\re \lambda \geq\omega_1$, $f\in h^{2+\alpha}(\s)$, and $\tau\in[0,1]$.
Using  the interpolation property \eqref{interpolation}  and Young's inequality, we find  a constant $\wt \omega_1>0$ depending on $K_2$ such that 
\begin{align}\label{FE1}
 4\kappa_1'\|(\lambda-\p_f\Phi^{\pi*}_{1,\tau})[f]\|_{1+\alpha}\geq&   |\lambda|\cdot\|f\|_{1+\alpha}+\|f\|_{2+\alpha} 
\end{align}
for all $\lambda\in\C$ with $\re \lambda \geq\wt\omega_1$,  $f\in h^{2+\alpha}(\s)$, and $\tau\in[0,1]$.
Additionally, choosing $\sigma$ sufficiently small we obtain from the definition of  $S_\sigma$ that
\[
\wt\kappa_1:=\max\Big\{4\kappa_1', \sup\big\{\|\p_f\Phi^{\pi*}_{1,1})\|_{\kL\big(h^{2+\alpha}(\s), h^{1+\alpha}(\s)\big)}\,:\, (f_*,h_*)\in S_\sigma\cap \mathbb{B}_{(h^{2+\alpha}(\s))^2}((f_0,h_0),\sigma)\big\}\Big\}<\infty.
\]
Setting  $\tau=1$, we conclude that  if $(\wt\omega_1-\p_f\Phi^{\pi*}_{1,1})=(\wt\omega_1-\p_f\Phi^{\pi}_{1}(f_*,h_*))\in{\rm Isom}(h^{2+\alpha}(\s),h^{1+\alpha}(\s)),$ 
then 
$$-\p_f\Phi^{\pi}_{1}(f_*,h_*) \in \kH\big(h^{2+\alpha}(\s), h^{1+\alpha}(\s),\wt \kappa_1,\wt \omega_1\big),$$
 where $\wt\kappa_1$ is independent of $(f_*,h_*)\in  \big(h^{3+\alpha}(\s)\big)^2\cap S_\sigma \cap \mathbb{B}_{(h^{2+\alpha}(\s))^2}((f_0,h_0),\sigma)$. 
 
 Hence, the assumptions of Lemma \ref{L:Ad1}  hold true for sufficiently  small $\sigma$ if we can show that   $(\wt\omega_1-\p_f\Phi^{\pi*}_{1,1})\in{\rm Isom}(h^{2+\alpha}(\s),h^{1+\alpha}(\s)).$
Using  the  method of continuity, cf. e.g. \cite[Theorem 5.2]{GT01}, and \eqref{FE1} it follows that $(\lambda-\p_f\Phi^{\pi*}_{1,1})$ is onto for $\re\lambda\geq\wt\omega_1$, provided that $(\lambda-\p_f\Phi^{\pi*}_{1,0})$ is onto.
But the surjectivity of   $(\lambda-\p_f\Phi^{\pi*}_{1,0})$ for positive $\lambda$  follows  as in the proof of Proposition \ref{L:DN1} below, and we thus have verified the assumptions of Lemma \ref{L:Ad1}.
\end{proof}

We present now  a generation result for the Dirichlet-Neumann  operator $\p_f\Phi^{\pi*}_{1,0}$ in a slightly more general context. 
\begin{prop}\label{L:DN1}
 Let $(f_*,h_*) \in\V$  be given and $a\in h^{2+\alpha}(\s)$ be a positive function. The linear operator 
  \begin{equation*}
 \bA[f]:=   \B(f_*) w_{-0}^\pi[f],\quad f\in h^{2+\alpha}(\s),
 \end{equation*}
 with  $(w_{+0}^\pi[f],w_{-0}^\pi[f])$  denoting the solution to 
  \begin{equation*} 
\left\{
\begin{array}{rllllll}
\A^\pi_0(f_*,h_*) w_{+0}^\pi[f]\!\!&=&\!\!0 &\text{in $ \Omega_+$},\\
\A^\pi_0(f_*) w_{-0}^\pi[f]\!\!&=&\!\!0&\text{in $ \Omega_-$},\\
 \B(f_*,h_*)w_{+0}^\pi[f]-\B(f_*)w_{-0}^\pi[f]\!\!&=&\!\!0&\text{on $ \Gamma_{0}$},\\
w_{+0}^\pi[f]-w_{-0}^\pi[f]\!\!&=&\!\!af&\text{on $\Gamma_0$},\\
w_{+0}^\pi[f]\!\!&=&\!\!0&\text{on $ \Gamma_1$},\\
 w_{-0}^\pi[f]\!\!&=&\!\!0&\text{on $\G_{-1}$}.
\end{array}
\right.
\end{equation*}
satisfies
$$
 -\bA\in\kH\big(h^{2+\alpha}(\s), h^{1+\alpha}(\s)\big).
$$
\end{prop}
\begin{proof} A short inspection of the proof of Theorem \ref{T1}, Lemma \ref{L:2}, and Theorem \ref{T:c1} reveals that 
there exist constants $\kappa\geq1$ and $\omega>0$ such that 
\begin{align*}\label{FE1}
 \kappa\|(\lambda-\bA)[f]\|_{1+\alpha}\geq&   |\lambda|\cdot\|f\|_{1+\alpha}+\|f\|_{2+\alpha} 
\end{align*}
for all $\lambda\in\C$ with $\re \lambda \geq\wt\omega$ and all  $f\in h^{2+\alpha}(\s)$.
 To finish the proof, it suffices to prove that  for all positive  $\lambda,$ the operator $(\lambda-\bA)\in\kL\big(h^{2+\alpha}(\s), h^{1+\alpha}(\s)\big)$ is onto.
 Let thus $\lambda>0$  and $F\in h^{1+\alpha}(\s)$ be given, denote by $(z_+,z_-)\in h^{2+\alpha}(\0_+)\times h^{2+\alpha}(\0_-)$ the solution to the diffraction problem
 \begin{equation} \label{DP1}
\left\{
\begin{array}{rllllll}
\A^\pi_0(f_*,h_*) z_+\!\!&=&\!\!0&\text{in $\Omega_+,$}\\
\A^\pi_0(f_*) z_-\!\!&=&\!\!0&\text{in $\Omega_-$},\\
 \B(f_*,h_*)z_+-\B(f_*)z_-\!\!&=&\!\!0&\text{on $ \Gamma_{0},$}\\
-\lambda a^{-1}(z_+-z_-)+\B(f_*)z_-\!\!&=&\!\! F&\text{on $\Gamma_0,$}\\
z_+\!\!&=&\!\!0&\text{on $ \Gamma_1,$}\\
 z_-\!\!&=&\!\!0&\text{on $\G_{-1},$}
\end{array}
\right.
\end{equation}
and set $f:=-a^{-1}\tr_0 (z_+-z_-) \in h^{2+\alpha}(\s).$
It is easy to see that $(w_{+0}^{\pi},w_{-0}^{\pi})[f]=(z_+,z_-).$
Therefore,
\begin{align*}
 (\lambda-\bA)[f]=\lambda f+\B(f_*)w_{-0}^{\pi}[f]=-\lambda a^{-1}(z_+-z_-)+\B(f_*)z_-=F.
\end{align*}
In the remaining part of the proof, we establish that for each $F\in h^{1+\alpha}(\s)$  and $\lambda>0$, 
the diffraction problem \eqref{DP1} possesses a unique solution  $(z_+,z_-)\in h^{2+\alpha}(\0_+)\times h^{2+\alpha}(\0_-).$ 
In fact, due to the arguments in the proof of Theorem \ref{DP}, it suffices to prove 
that the mapping 
\begin{equation}\label{opa}
(z_+,z_-)\mapsto
\begin{pmatrix} 
 \A^\pi_0(f_*,h_*) z_+\\
\A^\pi_0(f_*) z_-\\
 \B(f_*,h_*)z_+-\B(f_*)z_-\\
-\lambda a^{-1}\tr_0(z_+-z_-)+\B (f_*)z_-\\
\tr_1z_+ \\
 \tr_{-1}z_- 
\end{pmatrix}
\end{equation}
is an isomorphism between  $C^{2+\alpha}(\ov\0_+)\times  C^{2+\alpha}(\ov\0_-)$ and $  C^{\alpha}(\ov\0_+)\times C^{\alpha}(\ov\0_-) \times \big(C^{1+\alpha}(\s)\big)^2\times\big( C^{2+\alpha}(\s)\big)^2.$ 
It is easily seen that for $\lambda=0$ the operator \eqref{opa} is an isomorphism between these spaces (as the third and fourth operators defined by \eqref{opa} lead to  decoupled equations).
As the mapping $$(z_+,z_-)\mapsto \big(0,0,0,-\lambda a^{-1}\tr_0(z_+-z_-),0,0\big)$$ is compact, we conclude that \eqref{opa} defines a Fredholm operator of index zero.
We are left to show that \eqref{opa} defines for each $\lambda>0$ an operator which is one-to-one.
So, let $(z_+,z_-)$ be a solution to \eqref{DP1} corresponding to $F=0$ and assume that $$\max_{\ov\0_-} z_-=z_-(x_0,0)>0.$$
If $z_-\not\equiv 0,$ Hopf's lemma ensures that $$\big(\B(f_*)z_-\big)(x_0)=\lambda(\Delta_\rho)^{-1}[z_+(x_0,0)-z_-(x_0,0)]>0,$$ hence $\max_{\ov\0_+} z_+>\max_{\ov\0_-} z_-.$ 
 On the other hand,  if $\max_{\ov\0_+} z_+=z_+(x_1,0),$ then Hopf's lemma implies $\big(\B(f_*,h_*)z_+\big)(x_1)<0$, whence $(z_+-z_-)(x_1,0)<0.$ 
 This contradicts the inequality $\max_{\ov\0_+} z_+>\max_{\ov\0_-} z_-,$ meaning that  $z_-\equiv0.$
 But then also $z_+\equiv0, $ and    the proof is complete.
  \end{proof}


\section{The second diagonal operator}\label{Sec:6}
In this section we prove that the Fr\'echet derivative $\p_h\Phi_2(0,(f_0,h_0))$ is the generator of a strongly continuous and analytic 
semigroup as stated in \eqref{Claim2} when $(f_0,h_0)$ and $b_0$ are such that the second inequality of \eqref{CT1} is satisfied.
To derive the corresponding Theorem~\ref{T:c2} we proceed in a similar way as in Section \ref{Sec:5} and first identify  the ``leading order part''  $\p_h\Phi_2^\pi(f_0,h_0)$ of $\p_h\Phi_2(0,(f_0,h_0))$. 
In Lemma \ref{L:LOP2} it is shown, however, that the latter is related to the solution operator of a Dirichlet problem which differs from the case considered in Section \ref{Sec:5} 
where the leading order part $\p_f\Phi_1^\pi$ was related to a diffraction problem.
 The arguments that follow are thus somewhat simpler than those in Section~\ref{Sec:5}.
 \begin{lemma}\label{L:LOP2}
 Let $(f_*,h_*)\in\V$ and let $\p_{h}\Phi_2^\pi(f_*,h_*)\in\kL\big(h^{2+\alpha}(\s), h^{1+\alpha}(\s)\big)$ denote the operator defined by
 \begin{equation}\label{LOP2}
 \p_{h}\Phi_2^\pi(f_*,h_*)[h]:=-\frac{k}{\mu_+}\Big(\frac{2h_*' \tr_1\p_yv_+^*}{h_*- f_*} -  \tr_1\p_xv_+^*\Big)h'   -\B_1(f_*, h_*) W_+^{\pi}[h],\quad h\in h^{2+\alpha}(\s)\,,
 \end{equation}
where  $(v_+^*,v_-^*):=(v_+^*,v_-^*)(f_*,h_*,b_0)$ is defined in \eqref{zero} and  where $ W_+^{\pi}[h]$  denotes the solution to the Dirichlet problem
 \begin{equation}\label{P3'}
\left\{
\begin{array}{rllllll}
\A^\pi_1(f_*,h_*) W_+^{\pi}[h]\!\!&=&\!\!\cfrac{\tr_1\p_yv_+^*}{h_*-f_*}h''&\text{in $\Omega_+$},\\
W_+^{\pi}[h]\!\!&=&\!\!0&\text{on $ \Gamma_0$},\\
W_+^{\pi}[h]\!\!&=&\!\!g\rho_+h&\text{on $ \Gamma_1$},
\end{array}
\right.
\end{equation}
with
\begin{equation}\label{eqs2}
\begin{aligned}
&\A^\pi_1(f_*,h_*):=\p_{xx}-\frac{2h_*'}{ h_*-f_*}\p_{xy}+\frac{h_*'^2+1}{(h_*-f_*)^2}\p_{yy}.
\end{aligned}
\end{equation}
Then, given $\e\in(0,1)$, there exists $K_3=K_3(\e)>0$ such that  
\begin{align}
&\|\p_{h}\Phi_2(0,(f_*,h_*))[h]-\p_{h}\Phi_2^\pi(f_*,h_*)[h]\|_{1+\alpha}\leq \e\|h\|_{2+\alpha}+K_3\|h\|_{1+\alpha}\qquad\text{for all $h\in h^{2+\alpha}(\s)$}.\label{Claim5}
\end{align}
Moreover, $ \p_{h}\Phi_2^\pi \in C^\omega\big(\V, \kL\big(h^{2+\alpha}(\s), h^{1+\alpha}(\s)\big)\big)$.
\end{lemma}
\begin{proof} The regularity assertion follows by using Theorem \ref{DP}. As for \eqref{Claim5}, let $\e\in(0,1)$ be given and $\alpha'\in(0,\alpha)$ be fixed. 
Given $\delta\in(0,1)$, we pick a cut-off function $\chi:=\chi_\delta\in C^\infty([-1,1])$ such that   $0\leq \chi \leq 1,$ $\chi=0$ for $y\leq 1-\delta,$ and   $\chi=1$ for $y\geq1-\delta/2.$
  We then obtain from \eqref{LOP2}, \eqref{DC}, and the Appendix that
 \begin{align*}
  \|\p_{h}\Phi_2(0,(f_*,h_*))[h]-\p_{h}\Phi_2^\pi (f_*,h_*)[h]\|_{1+\alpha}\leq &C\big(\|h\|_{1+\alpha}+\|\B_1(f_*,h_*)(W_+[h]-W_+^{\pi}[h])\|_{1+\alpha}\big)\\
  \leq &C\big(\|h\|_{1+\alpha}+\|\tr_1 \partial_y(W_+[h]-W_+^{\pi}[h])\|_{1+\alpha}\big)\\
   \leq &C\big(\|h\|_{1+\alpha}+\|\chi(W_+[h]-W_+^{\pi}[h])\|_{2+\alpha}^{\0_+}\big),
 \end{align*}
where $(W_+[h],W_-[h])=(\p_hv_+^*(f_*,h_*,b_0)[h],\p_hv_-^*(f_*,h_*,b_0)[h])$ is the solution to \eqref{P3} and $C$ is independent of $\delta$.
 We now notice that the pair $(u_+,u_-):=(W_+^{\pi}[h],0)-(W_+[h],W_-[h])$  solves  according to \eqref{P3'}, \eqref{P3}, and the formulas for $\mathcal{A}(f_*,h_*)$ and $\partial_h \mathcal{A}(f_*,h_*)$ from the Appendix a diffraction problem of the form
 \begin{equation*}
\left\{
\begin{array}{rllllll}
\A^\pi_1(f_*,h_*)u_+\!\!&=&\!\!a_0^+h+a_1^+h'+(y\p_yv_+^*-\tr_1\p_yv_+^*)a_2^+h''\\
&&+(y-1)(b_0^+\partial_{xy} W_+[h]+b_1^+\partial_{yy}W_+[h])+b_2^+\partial_yW_+[h] &\text{in $\Omega_+$},\\
\A(f_*) u_-\!\!&=&\!\!0&\text{in $ \Omega_-$},\\
 \B(f_*,h_*)u_+-\B(f_*)u_-\!\!&=&\!\!\p_h\B(f_*,h_*)[h]v_+^*+\B(f_*,h_*)W_+^\pi[h]&\text{on $ \Gamma_{0}$},\\
u_+-u_-\!\!&=&\!\!0&\text{on $ \Gamma_0$},\\
u_+\!\!&=&\!\!0&\text{on $ \Gamma_1$},\\
 u_-\!\!&=&\!\!0&\text{on $\G_{-1}$},
\end{array}
\right.
\end{equation*}
with $a_i^+, b_i^+\in h^{\alpha}(\0_+)$, $0\leq i\leq2.$
Therefore,  $(u_+^\chi,u_-^\chi):= \chi(u_+,u_-)$ solves
  \begin{equation*} 
\left\{
\begin{array}{rllllll}
\A^\pi_1(f_*,h_*)u_+^\chi\!\!&=&\!\!\chi\A^\pi_1(f_*,h_*)u_+-\frac{2h_*'}{ h_*-f_*}\chi'\partial_xu_{+}+\frac{h_*'^2+1}{(h_*-f_*)^2}(2\chi'\partial_yu_{+}+\chi''u_+) &\text{in $\Omega_+$},\\
\A(f_*) u_-^\chi\!\!&=&\!\!0&\text{in $ \Omega_-$},\\
 \B(f_*,h_*)u_+^\chi-\B(f_*)u_-^\chi\!\!&=&\!\! 0&\text{on $ \Gamma_{0}$},\\
u_+^\chi-u_-^\chi\!\!&=&\!\!0&\text{on $ \Gamma_0$},\\
u_+^\chi\!\!&=&\!\!0&\text{on $ \Gamma_1$},\\
 u_-^\chi\!\!&=&\!\!0&\text{on $\G_{-1}$}.
\end{array}
\right.
\end{equation*}
Recalling the Schauder estimate \eqref{Schauder}, we obtain that
\begin{align*}
\|u_+^\chi\|_{2+\alpha}^{\0_+}\leq& C \|\chi\A^\pi_1(f_*,h_*)u_+\|_{\alpha}^{\0_+}+C(\delta)\|u_+\|_{2+\alpha'}^{\0_+} \\
 \leq & C\big(\|\chi(y-1)( b_0^+\partial_{xy}W_+[h]+b_1^+\partial_{yy}W_+[h])\|_{\alpha}^{\0_+}+\|\chi(y\p_yv_+^*-\tr_1\p_yv_+^*) a_2^+h''\|_{\alpha}^{\0_+}\big)\\
& +C(\delta)\|h\|_{2+\alpha'} \\
 \leq & C\big(\|\chi(y-1)\|_0^{\0_+}+ \|\chi(y\p_yv_+^*-\tr_1\p_yv_+^*)\|_0^{\0_+}\big)\|h\|_{2+\alpha}+C(\delta)\|h\|_{2+\alpha'},
\end{align*}
and choosing $\delta>0$ such that $\|\chi(y-1)\|_0^{\0_+}+ \|\chi(y\p_yv_+^*-\tr_1\p_yv_+^*)\|_0^{\0_+}<\e/2C$, interpolation properties of small H\"older spaces lead us to the desired estimate \eqref{Claim5}. 
\end{proof}

As in the previous section  we introduce for $\sigma>0$  the set $R_\sigma$ consisting of those $(f_*,h_*)\in\V$ satisfying the inequalities
\begin{equation*}
 \begin{array}{lll}
 (1) \! & \sigma< \min\{f_*-d, h_*-f_*\}, \qquad \|f_*\|_2+\|h_*\|_2< \sigma^{-1}\,,\\[2ex]
 (2) \! &\displaystyle g\rho_+> \sigma+\frac{\tr_1\p_yv_+^*}{h_*-f_*}\,,\\[3ex]
  (3) \! &   \|\tr_1\p_yv_+^*\|_0< \sigma^{-1} \,,
 \end{array}
\end{equation*}
where $(v_+^*,v_-^*)$ is defined in \eqref{zero}.
Since the functions $(f_0,h_0)$ and $b_0$ are chosen such that the second inequality in \eqref{CT1} is satisfied, we may choose $\sigma$ such that  $(f_0,h_0)\in R_\sigma$. 

\begin{lemma}\label{L:Ad1'}
Let $\sigma>0$ be such that  $(f_0,h_0)\in R_\sigma$.
Assume  there exists a constant $\wt \kappa_2:=\wt \kappa_2(\sigma)$ and for each $(f_*,h_*)\in  \big(h^{3+\alpha}(\s)\big)^2\cap R_\sigma \cap \mathbb{B}_{(h^{2+\alpha}(\s))^2}((f_0,h_0),\sigma)$ there exists a further constant $\wt \omega_2>0$ with the property that
$$-\p_{h}\Phi_{2}^{\pi}(f_*,h_*)\in \kH\big(h^{2+\alpha}(\s), h^{1+\alpha}(\s),\wt \kappa_2,\wt \omega_2\big).$$
It then holds
$$-\p_{h}\Phi_{2}^{\pi}(f_0,h_0)\in \kH\big(h^{2+\alpha}(\s), h^{1+\alpha}(\s)\big).$$
\end{lemma}
\begin{proof}
 The proof is similar to that of Lemma \ref{L:Ad1}.
\end{proof}

We are thus left  to prove that there exists $\sigma>0$ such that  $\p_{h}\Phi_{2}^{\pi}(f_*,h_*)$ is a analytic  generator   for each $(f_*,h_*)\in  \big(h^{3+\alpha}(\s)\big)^2\cap R_\sigma \cap \mathbb{B}_{(h^{2+\alpha}(\s))^2}((f_0,h_0),\sigma),$
with the constant $\wt\kappa_2$ depending only on $\sigma$. 
To this end we use the additional regularity  of $(f_*,h_*)\in  \big(h^{3+\alpha}(\s)\big)^2\cap R_\sigma$ to introduce again a parameter $\tau\in[0,1] $ which 
enables us to continuously transform the  leading order part  $\p_{h}\Phi_2^{\pi*}:=\p_{h}\Phi_2^\pi(f_*,h_*)$ into a (negative) Dirichlet-Neumann map.
More precisely, for each $\tau\in[0,1]$ we define the operator
$\p_{h}\Phi_{2,\tau}^\pi\in\kL\big(h^{2+\alpha}(\s), h^{1+\alpha}(\s)\big)$  by the formula
 \begin{equation}\label{HO2}
\p_{h}\Phi_{2,\tau}^{\pi*}[h]:=-\frac{\tau k}{\mu_+}\Big(\frac{2h_*' \tr_1\p_yv_+^*}{h_*- f_*} -  \tr_1\p_xv_+^*\Big)h'   -\B_1(f_*, h_*) W_{+\tau}^{\pi}[h], \qquad h\in h^{2+\alpha}(\s),
 \end{equation}
 with  $W_{+\tau}^{\pi}[h]$ denoting the solution to 
 \begin{equation}\label{HO2a}
\left\{
\begin{array}{rllllll}
\A^\pi_1(f_*,h_*) W_{+\tau}^{\pi}[h]\!\!&=&\!\!\cfrac{\tau\tr_1\p_yv_+^*}{h_*-f_*}h''&\text{in $\Omega_+$},\\
W_{+\tau}^{\pi}[h]\!\!&=&\!\!0&\text{on $ \Gamma_0$},\\
W_{+\tau}^{\pi}[h]\!\!&=&\!\!\Big[g\rho_+ -(1-\tau)\cfrac{\tr_1\p_yv_+^*}{h_*-f_*}\Big]h&\text{on $ \Gamma_1$}.
\end{array}
\right.
\end{equation}
For $\tau=1$ we see that $\p_{h}\Phi_{2,1}^{\pi*}=\p_{h}\Phi_{2}^\pi(f_*,h_*),$ while for $\tau=0$ we obtain a Dirichlet-Neumann operator.

We now prove the following perturbation  result.

\begin{thm}\label{T2} Let $\sigma>0$ be such that  $(f_0,h_0)\in R_\sigma$ and let  $\mu>0$ and  $\alpha'\in(0,\alpha)$ be given.
Then, given $(f_*,h_*)\in  \big(h^{3+\alpha}(\s)\big)^2\cap R_\sigma$,  there exist an integer $p\geq 3$, a $p$-partition of unity $\{\Pi_j^p\}_{1\leq j\leq 2^{p+1}}$, and a constant $K_4=K_4(p)$, and for each $\tau\in[0,1]$ and $1\leq j\leq 2^{p+1}$ 
there are bounded operators $\bB_{\tau,j}\in\kL\big(h^{2+\alpha}(\s), h^{1+\alpha}(\s)\big)$
 such that 
 \begin{equation}\label{DE2}
  \|\Pi_j^p\p_h\Phi^{\pi*}_{2,\tau}[h]-\bB_{\tau,j}[\Pi^p_j h]\|_{1+\alpha}\leq \mu \|\Pi_j^ph\|_{2+\alpha}+K_4\|h\|_{2+\alpha'} 
 \end{equation}
 for all $h\in h^{2+\alpha}(\s)$.
The operators $\bB_{\tau,j}$ are defined  by the formula
  \begin{equation}\label{HOjb}
 \bB_{\tau,j}[h]:=-\frac{\tau k}{\mu_+}\Big(\frac{2h_*' \tr_1\p_yv_+^*}{h_*- f_*} -  \tr_1\p_xv_+^*\Big)\Big|_{x_j^p}h'
 - \frac{k}{\mu_+}\Big(\frac{1+h_*'^2}{h_*-f_*}\Big|_{x_j^p}\tr_1 \p_yW_{+\tau}^{\pi,j}[h]-h_*'(x_j^p)\tr_1 \p_xW_{+\tau}^{\pi,j}[h]\Big),
 \end{equation}
where $ W_{+\tau}^{\pi,j}[h] $ denotes the solution to the problem
  \begin{equation}\label{HO2b}
\left\{
\begin{array}{rllllll}
\A^\pi_{1,j}(f_*,h_*) W_{+\tau}^{\pi,j}[h]\!\!&=&\!\!\cfrac{\tau\tr_1\p_yv_+^*}{h_*-f_*}\Big|_{x_j^p}h''&\text{in $\Omega_+$},\\
W_{+\tau}^{\pi,j}[h]\!\!&=&\!\!0&\text{on $ \Gamma_0$},\\
W_{+\tau}^{\pi,j}[h]\!\!&=&\!\! \Big[g\rho_+-(1-\tau)\cfrac{\tr_1\p_yv_+^*}{h_*-f_*}\Big|_{x_j^p}\Big] h&\text{on $ \Gamma_1$},\\
\end{array}
\right.
\end{equation}
with $\A^\pi_{1,j}(f_*,h_*)$  
being the operator obtained from  $\A^\pi_1(f_*,h_*) $ when evaluating its coefficients at  $x_j^p$.
\end{thm}
\begin{proof}
 Let  $\mu>0$ and  $\alpha'\in(0,\alpha)$ be fixed. 
Given $p\geq 3$, a $p$-partition of unity $\{\Pi_j^p\}_{1\leq j\leq 2^{p+1}}$, and an associated family $\{\chi_j^p\}_{1\leq j\leq 2^{p+1}}$ (see Section~\ref{Sec:5})
 we decompose
$$\Pi_j^p\p_h\Phi^{\pi*}_{2,\tau}[h]-\bB_{\tau,j}[\Pi^p_j h]=S_1+S_2+S_3$$ with 
\begin{align*}
 &S_1:=\frac{ (1+\tau)kg\rho_+}{\mu_+}\big(\Pi_j^ph_*'h'-h_*'(x_j^p)(\Pi_j^ph)'\big),\\
 &S_2:=\frac{ (1+\tau) k}{\mu_+}\Big[\frac{h_*' \tr_1\p_yv_+^*}{h_*- f_*}\Big|_{x_j^p}(\Pi_j^ph)'-\frac{h_*' \tr_1\p_yv_+^*}{h_*- f_*}\Pi_j^ph'\Big],\\
&S_3:=   \frac{k}{\mu_+}\frac{1+h_*'^2}{h_*-f_*}\Big|_{x_j^p}\tr_1 \p_yW_{+\tau}^{\pi,j}[\Pi_j^ph]-\frac{k}{\mu_+}\frac{1+h_*'^2}{h_*-f_*} \Pi_j^p\tr_1 \p_yW_{+\tau}^{\pi}[h].
\end{align*}
Still, we  use $C$ for constants which are independent of  $p$, constants depending on $p$ being denoted by $K$.
Recalling that $\chi_j^p\Pi_j^p=\Pi_j^p$, we estimate $S_1$ as 
\begin{align*}
 \|S_1\|_{1+\alpha}\leq& C\|\Pi_j^ph_*'h'-h_*'(x_j^p)(\Pi_j^ph)'\|_{1+\alpha}\leq C\|\Pi_j^ph'\chi_j^p(h_*'-h_*'(x_j^p))\|_{1+\alpha}+ K\|h\|_{1+\alpha}\\
 \leq &  C\|\chi_j^p(h_*'-h_*'(x_j^p))\|_{0}\,\|\Pi_j^ph\|_{2+\alpha}+ K\|h\|_{2},
\end{align*}
and similarly
 \begin{align*}
 \|S_2\|_{1+\alpha}\leq&   C\Big\|\chi_j^p\Big(\frac{2h_*' \tr_1\p_yv_+^*}{h_*- f_*}\Big|_{x_j^p}-\frac{2h_*' \tr_1\p_yv_+^*}{h_*- f_*}\Big)\Big\|_{0}\,\|\Pi_j^ph\|_{2+\alpha}+ K\|h\|_{2}.
\end{align*}
Choosing $p$ sufficiently large, we find that
\begin{align}\label{Mozard12}
 \|S_1\|_{1+\alpha}+\|S_2\|_{1+\alpha}\leq&   \frac{\mu}{2}\|\Pi_j^ph\|_{2+\alpha}+ K\|h\|_{1+\alpha}.
\end{align}
We are left to estimate $S_3$.
For this we note that
\begin{align*}
 \|S_3\|_{1+\alpha}\leq &C\Big\|\frac{1+h_*'^2}{h_*-f_*}\Big|_{x_j^p}\tr_1 \p_yW_{+\tau}^{\pi,j}[\Pi_j^ph]- \frac{1+h_*'^2}{h_*-f_*} \Pi_j^p\tr_1 \p_y W_{+\tau}^{\pi}[h] \Big\|_{1+\alpha}\\
 \leq &C\Big\|(1-\chi_j^p)\frac{1+h_*'^2}{h_*-f_*}\Big|_{x_j^p}\tr_1 \p_yW_{+\tau}^{\pi,j}[\Pi_j^ph]\Big\|_{1+\alpha}\\
 &+\Big\|\chi_j^p \frac{1+h_*'^2}{h_*-f_*}\Big|_{x_j^p}\tr_1 \p_yW_{+\tau}^{\pi,j}[\Pi_j^ph]-\chi_j^p\frac{1+h_*'^2}{h_*-f_*} \tr_1 \p_y\big(\Pi_j^pW_{+\tau}^{\pi}[h]\big)\Big\|_{1+\alpha}=:S_{31}+S_{32},
\end{align*}
where
\begin{align*}
 \|S_{31}\|_{1+\alpha}\leq &C\Big\|\tr_1 \p_y\big((1-\chi_j^p)W_{+\tau}^{\pi,j}[\Pi_j^ph]\big)\Big\|_{1+\alpha}\leq C\|(1-\chi_j^p)W_{+\tau}^{\pi,j}[\Pi_j^ph]\|_{2+\alpha}^{\0_+}.
\end{align*}
As $ z_+ :=(1-\chi_j^p) W_{+\tau}^{\pi,j}[\Pi_j^ph] $ solves  according to \eqref{HO2b} and \eqref{eqs2} the  problem
 \begin{equation*} 
\left\{
\begin{array}{rllllll}
\A^\pi_{1,j}(f_*,h_*) z_+\!\!&=&\!\!\cfrac{\tau\tr_1\p_yv_+^*}{h_*-f_*}\Big|_{x_j^p}(1-\chi_j^p)(\Pi_j^ph)''-(\chi_j^p)''W_{+\tau}^{\pi,j}[\Pi_j^ph]\\
&& -2(\chi_j^p)'\p_xW_{+\tau}^{\pi,j}[\Pi_j^ph]+\cfrac{2h_*'}{ h_*-f_*}\Big|_{x_j^p}(\chi_j^p)'\p_yW_{+\tau}^{\pi,j}[\Pi_j^ph]&\text{in $ \Omega_+$},\\
z_+\!\!&=&\!\!0&\text{on $ \Gamma_0$},\\
z_+\!\!&=&\!\!\Big[g\rho_+-(1-\tau)\cfrac{\tr_1\p_yv_+^*}{h_*-f_*}\Big|_{x_j^p}\Big](1-\chi_j^p)\Pi_j^ph&\text{on $ \Gamma_1$},
\end{array}
\right.
\end{equation*}
we infer from $(1-\chi_j^p)\Pi_j^p=0$ and  Schauder estimates for elliptic Dirichlet problems that 
\begin{align}\label{Mozard31}
 \|S_{31}\|_{1+\alpha}\leq&     K\|h\|_{2+\alpha'}.
\end{align}
Finally, 
\begin{align*}
 \|S_{32}\|_{1+\alpha}\leq &C\Big\|\chi_j^p \Big(\frac{1+h_*'^2}{h_*-f_*}\Big|_{x_j^p}-\frac{1+h_*'^2}{h_*-f_*}\Big)\Big\|_0\,\|\Pi_j^ph\|_{2+\alpha}+C\| \chi_j^p   \big(W_{+\tau}^{\pi,j}[\Pi_j^ph]-\Pi_j^pW_{+\tau}^{\pi}[h]\big)\|_{2+\alpha}^{\0_+},
\end{align*}
and for $p$ large enough we have
\begin{align}\label{Mozard32a}
\Big\|\chi_j^p \Big(\frac{1+h_*'^2}{h_*-f_*}\Big|_{x_j^p}-\frac{1+h_*'^2}{h_*-f_*}\Big)\Big\|_0\leq \frac{\mu}{4C}.
\end{align}
On the other hand, $ u_+ :=\chi_j^p   \big(W_{+\tau}^{\pi,j}[\Pi_j^ph]-\Pi_j^pW_{+\tau}^{\pi}[h]\big)$ solves the Dirichlet problem
\begin{equation*} 
\left\{
\begin{array}{rllllll}
\A^\pi_{1}(f_*,h_*) u_+\!\!&=&\!\!\big(\A^\pi_1(f_*,h_*)-\A^\pi_{1,j}(f_*,h_*)\big)\big[\chi_j^p W_{+\tau}^{\pi,j}[\Pi_j^ph]\big]\\
&&+\cfrac{\tau\tr_1\p_yv_+^*}{h_*-f_*}\Big|_{x_j^p}\chi_j^p(\Pi_j^ph)''-\cfrac{\tau\tr_1\p_yv_+^*}{h_*-f_*}\Pi_j^ph''\\
&&+(\chi_j^p)''W_{+\tau}^{\pi,j}[\Pi_j^ph]-(\Pi_j^p)''W_{+\tau}^{\pi}[h]\\
&&+2(\chi_j^p)'\p_xW_{+\tau}^{\pi,j}[\Pi_j^ph]-2(\Pi_j^p)'\p_xW_{+\tau}^{\pi}[h]\\
&&-\cfrac{2h_*'}{ h_*-f_*}\Big|_{x_j^p}(\chi_j^p)'\p_yW_{+\tau}^{\pi,j}[\Pi_j^ph]
+\cfrac{2h_*'}{ h_*-f_*}(\Pi_j^p)'\p_yW_{+\tau}^{\pi}[h]&\text{in $ \Omega_+$},\\
u_+\!\!&=&\!\! (1-\tau )\chi_j^p\Big(\cfrac{\tr_1\p_yv_+^*}{h_*-f_*}-\cfrac{\tr_1\p_yv_+^*}{h_*-f_*}\Big|_{x_j^p}\Big)\Pi_j^ph&\text{on $ \Gamma_0$},\\
u_+\!\!&=&\!\!0&\text{on $ \Gamma_1$},
\end{array}
\right.
\end{equation*}
and therefore
\begin{align*}
 \|u_+\|_{2+\alpha}^{\0_+}\leq & C\Big(\|\big(\A^\pi_1(f_*,h_*)-\A^\pi_{1,j}(f_*,h_*)\big)\big[\chi_j^p W_{+\tau}^{\pi,j}[\Pi_j^ph]\big]\|_\alpha\\
 &\quad+\Big\|\chi_j^p\Big(\cfrac{\tr_1\p_yv_+^*}{h_*-f_*}-\cfrac{\tr_1\p_yv_+^*}{h_*-f_*}\Big|_{x_j^p}\Big)\Big\|_0\|\Pi_j^ph''\|_\alpha\\
 &\quad+ \Big\|\chi_j^p\Big(\cfrac{\tr_1\p_yv_+^*}{h_*-f_*}-\cfrac{\tr_1\p_yv_+^*}{h_*-f_*}\Big|_{x_j^p}\Big)\Big\|_0\|\Pi_j^ph\|_{2+\alpha}\Big)\\
 &+K\|h\|_{2+\alpha'}.
\end{align*}
For $p$ sufficiently large, we obtain together with \eqref{Mozard32a} that 
\begin{align}\label{Mozard32}
 \|S_{32}\|_{1+\alpha}\leq & \frac{\mu}{4} \|\Pi_j^ph\|_{2+\alpha}+C\| u_+\|_{2+\alpha}^{\0_+}\leq \frac{\mu}{2}\|\Pi_j^ph\|_{2+\alpha}+K\|h\|_{2+\alpha'}.
\end{align}
Gathering \eqref{Mozard12}, \eqref{Mozard31}, and \eqref{Mozard32}, the desired estimate \eqref{DE2} follows.
\end{proof}
\bigskip

\noindent{\bf Fourier analysis: The symbol of $\bB_{\tau,j}$.} 
Let $\tau\in[0,1],$ $p,j\in\N$ with  $p\geq 3$ and $1\leq j\leq 2^{p+1}$ be arbitrary.
Given $h\in h^{2+\alpha}(\s)$, we consider its Fourier expansion
\[h=\sum_{m\in\Z}h_me^{imx}\]
and look for the corresponding expansion of $\bB_{\tau,j}[ h].$
Recalling the definition of $\bB_{\tau,j}$ from \eqref{HOjb}, we first determine the expansions for  $\tr_1\nabla  W_{+\tau}^{\pi,j}[h].$
Let 
\begin{align}\label{BAEW}
   W_{+\tau}^{\pi,j}[h] =\sum_{m\in\Z} B_m h_m e^{imx}
\end{align}
with functions $B_m = B_m (y)$ to be determined.
Recalling  \eqref{eqs2},  we  let (using the summation convention)
\begin{align*}
 &\A^\pi_{1,j}(f_*,h_*)=c_{pl}\p_{pl}
\end{align*} 
and set
\begin{align*}
&a:=c_{12}/c_{22},\quad b:=1/c_{22},\quad D:=\sqrt{b-a^2}
\end{align*}
noticing that $b-a^2>0$.
 Since $ W_{+\tau}^{\pi,j}[h]$ is the solution to \eqref{HO2b}, it follows from \eqref{BAEW}  that for each $m\in\Z,$ the function $ B_m $ solves the problem 
 \begin{equation}\label{Hobit2}
\left\{
\begin{array}{rllllll}
(B_m)''+2mia(B_m)'-bm^2B_m\!\!&=&\!\!-\tau bV m^2  &\text{in $ 0<y<1$},\\
B_m(0)\!\!&=&\!\!0,\\
B_m(1)\!\!&=&\!\!g\rho_+-(1-\tau)V,
\end{array}
\right.
\end{equation}
where 
\begin{equation}\label{V}
V:= \frac{\tr_1\p_yv_+}{h_*-f_*}\Big|_{x_j^p}.
\end{equation}

The general solution  to the first equation  of \eqref{Hobit2} is given by the formula
\begin{equation}\label{For1}
 B_m:=u_m+i v_m
\end{equation}
with real part  
\begin{align*}
 u_m(y)=&\zeta_1\Big(\cos(amy)\cosh(Dmy)+\frac{a}{D}\sin(amy)\sin(Dmy)\Big)+\frac{\zeta_2}{D m}\cos(amy)\sin(Dmy)\\
 &+\zeta_3\Big(\sin(amy)\cosh(Dmy)-\frac{a}{D}\cos(amy)\sin(Dmy)\Big)+\frac{\zeta_4}{D m}\sin(amy)\sin(Dmy)+  \tau V
\end{align*}
and imaginary part 
\begin{align*}
 v_m(y)=&\zeta_1\Big(-\sin(amy)\cosh(Dmy)+\frac{a}{D}\cos(amy)\sin(Dmy)\Big)-\frac{\zeta_2}{D m}\sin(amy)\sin(Dmy)\\
 &+\zeta_3\Big(\cos(amy)\cosh(Dmy)+\frac{a}{D}\sin(amy)\sin(Dmy)\Big)+\frac{\zeta_4}{D m}\cos(amy)\sin(Dmy).
\end{align*}
The real constants $\{\zeta_i\,:\, 1\leq i\leq 4\}$ are to be determined such that  the last  two equations of \eqref{Hobit2} are also satisfied by $ B_m $. 
  Explicit  computations  (much simpler compared to those  in Section~\ref{Sec:5}) show that such constants can be uniquely determined to obtain $ B_m $.
Furthermore, using the expressions for $\zeta_i$
it follows again by explicit computations that the operator $\bB_{\tau,j}$ from \eqref{HOjb} is a Fourier multiplier with symbol $(\varphi_m)_{m}$, the real part being given by
\begin{align}\label{FMHR}
 \re\varphi_m:=-\mu_m-\nu_m,
\end{align}
where
\begin{align*}
\mu_m:=\frac{k}{\mu_+}\frac{ g\rho_+- V}{\tanh(Dm)/m}\qquad\text{and}\qquad
 \nu_m:=\frac{\tau kV}{\mu_+}\frac{  \cos(am)}{\sinh(Dm)/m}. 
\end{align*}
The imaginary part is given by
\begin{align}
 \im\varphi_m
 :=&\frac{\tau k}{\mu_+}\Big[\frac{    a( (2-\tau) V-    g\rho_+)}{D}m+  \frac{V\sin(am)}{\sinh(Dm)/m}\Big].\label{FMHI}
\end{align}
The representation of $\bB_{\tau,j}$ allows us now to prove the following generation result.

\begin{lemma}\label{L:2h}
Let $\sigma>0$ be such that   $(f_0,h_0)\in R_\sigma$.
Then, there exist  constants  $\kappa_2\geq 1$ and $\omega_2>0$, depending only $\sigma$ such that for all $(f_*,h_*)\in  \big(h^{3+\alpha}(\s)\big)^2\cap R_\sigma$,   any $p$-partition 
of unity $\{\Pi_j^p\}_{1\leq j\leq 2^{p+1}}$,  $p\in\N$ with $p\geq 3$, and any $\tau\in[0,1]$,
the operators  $\bB_{\tau,j}$,     $1\leq j\leq 2^{p+1}$,   defined by \eqref{HOjb}  satisfy
 \begin{align}\label{13h}
&\lambda-\bB_{\tau,j}\in{\rm Isom}(h^{2+\alpha}(\s),h^{1+\alpha}(\s)),\\[1ex]
\label{14h}
& \kappa_2\|(\lambda-\bB_{\tau,j})[h]\|_{1+\alpha}\geq  |\lambda|\cdot\|h\|_{1+\alpha}+\|h\|_{2+\alpha}
\end{align}
for all $h\in h^{2+\alpha}(\s)$ and $\lambda\in\C$ with $\re \lambda\geq \omega_2$.
\end{lemma}

\begin{proof}
 Recalling \eqref{V} and the definition of $R_\sigma$, we find a constant  
  $C_0>0$ depending on $\sigma$ such that 
 \[C_0^{-1}\leq g\rho_+-V\leq C_0,\]
and such that the relations \eqref{pup}
 are satisfied when replacing $(\lambda_m)$ by $(\varphi_m).$
With this observation, the desired result follows along the lines of  the proof of  Lemma \ref{L:2}.  
 \end{proof}

We are now in a position to establish \eqref{Claim2}.

\begin{thm}\label{T:c2}
 Let $(f_0,h_0)) \in\V$ and $b_0\in h^{2+\alpha}(\s)$ be given such that the second inequality in \eqref{CT1}  is satisfied.  
Then 
 \begin{align*}
 -\p_{h}\Phi_2(0,(f_0,h_0))\in\kH\big(h^{2+\alpha}(\s), h^{1+\alpha}(\s)\big).
\end{align*}
\end{thm}
\begin{proof}  Recalling \eqref{Claim5}  (with $(f_*,h_*)=(f_0,h_0)$) and   \cite[Theorem I.1.3.1 (ii)]{Am95}, we are left  to show that $-\p_h\Phi^\pi_2(f_0,h_0)\in\kH\big(h^{2+\alpha}(\s), h^{1+\alpha}(\s)\big).$
For this, we prove that  the assumptions of Lemma \ref{L:Ad1'} are satisfied provided that $\sigma$ is sufficiently small.

Let $\sigma>0$ be such that  $(f_0,h_0)\in R_\sigma$,  let $\kappa_2$ and $\omega_2$ be the constants found in Lemma \ref{L:2h}, and pick  $\alpha'\in(0,\alpha)$.
Given $(f_*,h_*)\in  \big(h^{3+\alpha}(\s)\big)^2\cap R_\sigma$ and using the same arguments as in the proofs of Theorem \ref{T:c1}, Theorem \ref{T2}, and Lemma \ref{L:2h}, it follows 
 there exists a constant  $\wt\omega_2>1$ such that 
\begin{align}\label{FE1h}
 4\kappa_2\|(\lambda-\p_h\Phi^{\pi*}_{2,\tau})[h]\|_{1+\alpha}\geq&   |\lambda|\cdot\|h\|_{1+\alpha}+\|h\|_{2+\alpha} 
\end{align}
for $\lambda\in \C$ with $\re\lambda\ge \wt\omega_2$, $h\in h^{2+\alpha}(\s)$, and $\tau\in [0,1]$.
As in  the proof of Theorem \ref{T:c1} we may now choose $\sigma$ small to find a constant $\wt\kappa_2\geq1$ depending only on $\sigma$ such that, for $(f_*,h_*)$ belonging additionally to the ball $\mathbb{B}_{(h^{2+\alpha}(\s))^2}((f_0,h_0),\sigma)$, we have
$$-\p_h\Phi^{\pi}_{2}(f_*,h_*) \in \kH\big(h^{2+\alpha}(\s), h^{1+\alpha}(\s),\wt \kappa_2,\wt \omega_2\big),$$
provided $(\wt\omega_2-\p_h\Phi^{\pi*}_{2,1})=(\wt\omega_2-\p_h\Phi^{\pi}_{2}(f_*,h_*))$ is an isomorphism.

We now establish this last property.
Note that, due to \eqref{FE1h},     $(\wt\omega_2-\p_h\Phi^{\pi*}_{2,1})$ is an isomorphism  if $(\wt\omega_2-\p_2\Phi^{\pi*}_{2,0})$ is onto.
To prove the latter let
$\lambda>0$  and $H\in h^{1+\alpha}(\s)$ be given. We  let $z\in h^{2+\alpha}(\0_+)$  be the unique solution to the elliptic boundary value problem
 \begin{equation} \label{DP2}
\left\{
\begin{array}{rllllll}
\A^\pi_1(f_*,h_*) z\!\!&=&\!\!0&\text{in $\Omega_+,$}\\
z \!\!&=&\!\!0 &\text{on $\Gamma_0,$}\\
\lambda\Big(g\rho_+-\cfrac{\tr_1\p_yv_+^*}{h_*-f_*}\Big)^{-1} z+\B_1(f_*,h_*)z\!\!&=&\!\!H&\text{on $ \Gamma_1.$}
\end{array}
\right.
\end{equation}
We then set $$h:=\Big(g\rho_+-\cfrac{\tr_1\p_yv_+^*}{h_*-f_*}\Big)^{-1}\tr_1 z \in h^{2+\alpha}(\s).$$
Recalling the definitions \eqref{HO2} and \eqref{HO2a}, it follows that $W_{+0}^{\pi}[h]=z$
and
\begin{align*}
 (\lambda-\p_h\Phi^{\pi*}_{2,0})[h]=\lambda h+\B_1(f_*,h_*)W_{+0}^{\pi}[h]= \lambda\Big(g\rho_+-\cfrac{\tr_1\p_yv_+^*}{h_*-f_*}\Big)^{-1}\tr_1 z+\B_1(f_*,h_*)z=H.
\end{align*}
In view of Lemma \ref{L:Ad1'} the proof is complete.
  \end{proof}

Finally, we are ready to prove Theorem \ref{MT1}.

\begin{proof}[{\bf Proof of Theorem \ref{MT1}}]
 Because of the equivalence of the problems \eqref{PB} and \eqref{eq:TS} and the reduction of the latter to \eqref{AEP}, we are left to investigate \eqref{AEP}. For the existence and uniqueness result we shall employ an abstract result for fully nonlinear problems \cite[Theorem 8.4.1]{L95}. Let us first note that 
\eqref{AG} is implied by Lemma \ref{L:NDO}, Theorems \ref{T:c1}, \ref{T:c2}, and 
 \cite[Theorem I.1.6.1 and Remark I.1.6.2]{Am95}. 
Thus,  the assumptions of \cite[Theorem 8.4.1]{L95} are satisfied owing to \eqref{reg1}, \eqref{AG}, and the interpolation property~\eqref{interpolation}.
This proves the existence and uniqueness part. 

 Finally, the continuous dependence of the solution on the initial data follows from \cite[Theorem 8.4.4]{L95}.
\end{proof}

\section{The Muskat problem with surface tension effects}\label{Sec:7}

In this section we investigate the Muskat problem introduced in Section \ref{Sec:2} when allowing for surface tension effects in the presence or absence of gravity.
More precisely, instead of being continuous along the interfaces $\Gamma(f)$ and $\Gamma(h)$,
we assume that the pressure jump along an interface is proportional to the curvature of the respective interface, i.e., the pressure
obeys the Laplace-Young equation
\begin{equation}\label{LYE}
\begin{aligned}
& p_--p_+=-\gamma_f\kappa(f)\quad\text{on $\Gamma(f)$,}\\
&p_+=-\gamma_h\kappa(h)\quad\text{on $\Gamma(h),$}
\end{aligned}
\end{equation}
where $\gamma_f$ [resp. $\gamma_h$] is the surface tension coefficient at the interface $\Gamma(f)$ [resp. $\Gamma(h)$] and where for each $\zeta\in C^2(\s)$ the function
\[\kappa(\zeta):=\frac{\zeta''}{(1+\zeta'^2)^{3/2}}\]
is the curvature of the graph $[y=\zeta]$.

Hence, instead of \eqref{eq:S} we consider the system
\begin{equation}\label{eq:SK}
\left\{\begin{array}{rllllll}
\Delta u_+\!\!&=&\!\!0&\text{in}& \Omega(f,h), \\
\Delta u_-\!\!&=&\!\!0&\text{in}& \Omega(f), \\
{\p_th}\!\!&=&\!\!-k\mu_+^{-1}\sqrt{1+h'^2}\p_\nu u_+&\text{on}& \Gamma(h),\\
u_+\!\!&=&\!\!g\rho_+h-\gamma_h\kappa(h)&\text{on}&\Gamma(h),\\
u_-\!\!&=&\!\!b&\text{on}&\G_d,\\
u_+-u_-\!\!&=&\!\!g(\rho_+-\rho_-)f+\gamma_f\kappa(f)&\text{on}& \Gamma(f),\\
{\p_tf}\!\!&=&\!\!-k{\mu_\pm^{-1}}\sqrt{1+f'^2}\p_\nu u_\pm &\text{on}& \Gamma(f),
\end{array}
\right.
\end{equation}
supplemented with the initial conditions \eqref{eq:S1}.

The main result regarding problem \eqref{eq:SK} is following theorem. Its proof is given at the end of this section.

\begin{thm}\label{MT2}
   Let $g\ge 0$, $(f_0,h_0)\in\V\cap (h^{4+\alpha}(\s))^2$, and  $b$ be given such that \eqref{CL} holds.
  Then, there exist a maximal existence time $T_0:=T_0(f_0,h_0)\in(0,T]$ and a unique classical  H\"older solution\footnote{The notion of classical solution for \eqref{eq:SK} and \eqref{eq:S1} is the same as in Section \ref{Sec:2} with the modification 
  that we require additionally to \eqref{CS} that $(f,h)\in C\big([0,T_0),(h^{4+\alpha}(\s))^2\big).$  
  } $(f,h,u_+,u_-)$ to \eqref{eq:SK} and \eqref{eq:S1} on $[0,T_0).$ Additionally, the solutions depend continuously on the initial data.
\end{thm}

In order to prove Theorem \ref{MT2}, we first recast the problem as a nonlinear and nonlocal evolution equation.
To this end,  we infer from the proof of Theorem \ref{DP} that for each pair of functions $(f,h)\in\V\cap (h^{4+\alpha}(\s))^2$ and $b\in h^{2+\alpha}(\s)$
there exists a unique solution $$(v_+,v_-):=(v_+(f,h,b), v_-(f,h,b))\in \mbox{\it{h}}\,^{2+\alpha}(\Omega_+)\times \mbox{\it{h}}\,^{2+\alpha}(\Omega_-)$$
to the diffraction problem 
 \begin{equation}\label{TK}
\left\{
\begin{array}{rllllll}
\A(f,h) v_+\!\!&=&\!\!0&\text{in $ \Omega_+$},\\
\A(f) v_-\!\!&=&\!\!0&\text{in $ \Omega_-$},\\
\B(f,h)v_+-\B(f)v_-\!\!&=&\!\!0 &\text{on $ \Gamma_{0}$},\\
v_+-v_-\!\!&=&\!\!g(\rho_+-\rho_-) f+\gamma_f\kappa(f)&\text{on $ \Gamma_0$},\\
v_+\!\!&=&\!\!g\rho_+h-\gamma_h\kappa(h)&\text{on $ \Gamma_1$},\\
 v_-\!\!&=&\!\!b&\text{on $\G_{-1}$},
\end{array}
\right.
\end{equation}
  the mapping 
 \[\big[(f,h,b)\mapsto (v_+(f,h,b), v_-(f,h,b))\big]\in C^\omega\big(\big(\V\cap (h^{4+\alpha}(\s))^2\big)\times h^{2+\alpha}(\s), \mbox{\it{h}}\,^{2+\alpha}(\Omega_+)\times \mbox{\it{h}}\,^{2+\alpha}(\Omega_-)\big)\]
 being real-analytic  as a consequence of $\kappa\in C^\omega\big(h^{4+\alpha}(\s),h^{2+\alpha}(\s)\big)$.
 Using this observation, we deduce that the problem \eqref{eq:SK} and \eqref{eq:S1} is equivalent to the evolution equation
\begin{equation}\label{AK}
 \p_t(f,h)=\Phi(t,(f,h)),
\end{equation}
where $\Phi:[0,T)\times \big(\V\cap (h^{4+\alpha}(\s))^2\big)\subset \R\times(h^{4+\alpha}(\s))^2\to (h^{1+\alpha}(\s))^2$ is the operator 
$\Phi=(\Phi_1,\Phi_2) $ defined by
\begin{equation}\label{PHIK}
\begin{aligned}
 &\Phi_1(t,(f,h)):=-\B(f)v_-(f,h,b(t)),\\ 
 &\Phi_2(t,(f,h)):=-\B_1(f,h)v_+(f,h,b(t)),
 \end{aligned}
\end{equation}
and where $(v_+,v_-)$ denotes now the  solution operator for \eqref{TK}.
Since we have 
\begin{equation}\label{K1}
\begin{aligned}
&\Phi\in C\big([0,T)\times \big(\V\cap (h^{4+\alpha}(\s))^2\big), (h^{1+\alpha}(\s))^2\big),\\
&\Phi(t,\cdot )\in C^\omega\big(  \V\cap (h^{4+\alpha}(\s))^2, (h^{1+\alpha}(\s))^2\big) \qquad\text{for all $t\in[0,T)$,}
\end{aligned}
\end{equation}
our next goal is to show that
\begin{equation}\label{GK}
 -\p_{(f,h)}\Phi(0,(f_0,h_0))\in\kH((h^{4+\alpha}(\s))^2, (h^{1+\alpha}(\s))^2).
\end{equation}

In the remaining part we let $(v_+,v_-)$ denote the solution to \eqref{TK} determined by the tuple $(f_0,h_0,b_0)$ with $b_0:=b(0)$. 
As in the first part, the derivative $\p_{(f,h)}\Phi(0,(f_0,h_0))$ is a matrix operator, three of its components being given by \eqref{DC}  (with $(f_*,h_*)$ replaced by $(f_0,h_0)$), but with 
$$(w_+[f],w_-[f]):=(\p_fv_+(f_0,h_0,b_0)[f],\p_fv_-(f_0,h_0,b_0)[f])$$ being the solution to the
diffraction problem 
\begin{equation}\label{P2K}
\left\{
\begin{array}{rllllll}
\A(f_0,h_0) w_+[f]\!\!&=&\!\!-\p_f\A(f_0,h_0)[f]v_+ &\text{in $ \Omega_+$},\\
\A(f_0) w_-[f]\!\!&=&\!\!-\p_f\A(f_0)[f] v_-&\text{in $ \Omega_-$},\\
 \B(f_0,h_0)w_+[f]-\B(f_0)w_-[f]\!\!&=&\!\!-\p_f\B(f_0,h_0)[f]v_++\p_f\B(f_0)[f]v_-&\text{on $\Gamma_{0}$},\\
w_+[f]-w_-[f]\!\!&=&\!\!g(\rho_+-\rho_-) f+\gamma_f\p_f\kappa(f_0)[f]&\text{on $ \Gamma_0$},\\
w_+[f]\!\!&=&\!\!0&\text{on $ \Gamma_1$},\\
 w_-[f]\!\!&=&\!\!0&\text{on $\G_{-1}$},
\end{array}
\right.
\end{equation}
 and with  $(W_+[h],W_-[h]):=(\p_hv_+(f_0,h_0,b_0)[h],\p_hv_-(f_0,h_0,b_0)[h])$  solving
\begin{equation}\label{P3K}
\left\{
\begin{array}{rllllll}
\A(f_0,h_0) W_+[h]\!\!&=&\!\!-\p_h\A(f_0,h_0)[h]v_+ &\text{in $\Omega_+$},\\
\A(f_0) W_-[h]\!\!&=&\!\!0&\text{in $ \Omega_-$},\\
 \B(f_0,h_0)W_+[h]-\B(f_0)W_-[h]\!\!&=&\!\!-\p_h\B(f_0,h_0)[h]v_+&\text{on $ \Gamma_{0}$},\\
W_+[h]-W_-[h]\!\!&=&\!\!0&\text{on $ \Gamma_0$},\\
W_+[h]\!\!&=&\!\!g\rho_+h-\gamma_h\p_h\kappa(h_0)[h]&\text{on $ \Gamma_1$},\\
 W_-[h]\!\!&=&\!\!0&\text{on $\G_{-1}$}.
\end{array}
\right.
\end{equation}

In the following we   prove that the diagonal entries of $\p_{(f,h)}\Phi(0,(f_0,h_0))$ generate analytic semigroups when seen as unbounded operators in $h^{1+\alpha}(\s)$ with dense domain $h^{4+\alpha}(\s),$
while one of the off-diagonal entries is in some sense small (see relation \eqref{Claim3K} below).
The analysis is not so involved as in the previous sections: on the one hand, the highest derivatives of $f$ and $h$ in \eqref{P2K} and \eqref{P3K} are hidden in the Fr\'echet derivative of the curvature operators 
(the evolution equation \eqref{AK} has a quasilinear structure now), and, on the other hand, we can rely on previous arguments and computations. Therefore, we do not prove all our statements in detail. 
\bigskip

\noindent{\bf The off-diagonal entry.} 
We claim that for each $\e\in(0,1)$, there exists $K_0=K_0(\e)>0$ such that  
\begin{align}
&\|\p_{h}\Phi_1(0,(f_0,h_0))[h]\|_{1+\alpha}\leq \e\|h\|_{4+\alpha}+K_0\|h\|_{1+\alpha}\qquad\text{for all $h\in h^{2+\alpha}(\s)$}.\label{Claim3K}
\end{align}
Recalling \eqref{DC} and observing that 
\[\p_h\kappa(h_0)[h]=\frac{h''}{(1+h_0'^2)^{3/2}}-\frac{3h_0'h_0''h'}{(1+h_0'^2)^{5/2}}\qquad\text{for all $h\in h^{4+\alpha}(\s)$},\] 
it is useful to split the  solution   $(W_+[h], W_-[h])$  of \eqref{P3K} as
\begin{equation*}
 (W_+[h], W_-[h])=(W_+^{1}, W_-^{1})+(W_+^{2}, W_-^{2}),
\end{equation*}
where $(W_+^{1}, W_-^{1})$ is the solution to 
\begin{equation}\label{P4K}
\left\{
\begin{array}{rllllll}
\A(f_0,h_0) W_+^1\!\!&=&\!\!-\p_h\A(f_0,h_0)[h]v_+ &\text{in $\Omega_+$},\\
\A(f_0) W_-^1\!\!&=&\!\!0&\text{in $ \Omega_-$},\\
 \B(f_0,h_0)W_+^1-\B(f_0)W_-^1\!\!&=&\!\!-\p_h\B(f_0,h_0)[h]v_+&\text{on $ \Gamma_{0}$},\\
W_+^1-W_-^1\!\!&=&\!\!0&\text{on $ \Gamma_0$},\\
W_+^1\!\!&=&\!\!g\rho_+h+3\gamma_hh_0'h_0''(1+h_0'^2)^{-5/2}h'&\text{on $ \Gamma_1$},\\
 W_-^1\!\!&=&\!\!0&\text{on $\G_{-1}$},
\end{array}
\right.
\end{equation} 
while $(W_+^{2}, W_-^{2})$ solves
\begin{equation}\label{P5K}
\left\{
\begin{array}{rllllll}
\A(f_0,h_0) W_+^2\!\!&=&\!\!0 &\text{in $\Omega_+$},\\
\A(f_0) W_-^2\!\!&=&\!\!0&\text{in $ \Omega_-$},\\
 \B(f_0,h_0)W_+^2-\B(f_0)W_-^2\!\!&=&\!\!0&\text{on $ \Gamma_{0}$},\\
W_+^2-W_-^2\!\!&=&\!\!0&\text{on $ \Gamma_0$},\\
W_+^2\!\!&=&\!\!-\gamma_h (1+h_0'^2)^{-3/2}h''&\text{on $ \Gamma_1$},\\
 W_-^2\!\!&=&\!\!0&\text{on $\G_{-1}$}.
\end{array}
\right.
\end{equation} 
In view of \eqref{DC} we get 
\begin{align}
\|\p_{h}\Phi_1(0,(f_0,h_0))[h]\|_{1+\alpha}&\leq  C \big(\|\tr_0\nabla W_-^1\|_{1+\alpha}+\|\tr_0 \nabla W_-^2\|_{1+\alpha}\big)\nonumber\\
&\leq  C \big(\| W_-^1\|_{2+\alpha}^{\0_-}+\|  W_-^2\|_{2+\alpha}^{(1/2)\0_-}\big),\label{I0K}
\end{align}
where we again use the notation $(1/2)\0_-=\s\times(-1/2,0).$ 
Due to \eqref{Schauder} we have
\begin{equation}\label{I1K}
\| W_-^1\|_{2+\alpha}^{\0_-}\leq C\|h\|_{3+\alpha}.
\end{equation}
Finally, we infer from \cite[Theorem 9.3]{ADN64} and the arguments in the proof of Proposition \ref{P:1} that 
\begin{equation}\label{I2K}
 \|W_-^2\|_{2+\alpha}^{(1/2)\0_-}\leq C\big(\|W_+^2\|_{0}^{\0_+}+\|W_-^2\|_0^{\0_-}\big)\leq C\big(\|W_+^2\|_{2+\alpha/2}^{\0_+}+\|W_-^2\|_{2+\alpha/2}^{\0_-}\big)\leq C\|h\|_{4+\alpha/2}
\end{equation}
for all $h\in h^{2+\alpha}(\s).$
The relations \eqref{I0K}-\eqref{I2K} and the interpolation property \eqref{interpolation} lead us the desired estimate \eqref{Claim3K}.\bigskip


\noindent{\bf The first diagonal entry.} 
We start by identifying the ``leading order part''  $ \p_{f}\Phi_1^\pi\in\kL\big(h^{4+\alpha}(\s), h^{1+\alpha}(\s)\big)$ of $\p_{f}\Phi_1(0,(f_0,h_0))$. 
This is defined as 
 \begin{equation*}
 \p_{f}\Phi_1^\pi[f]:= -\B(f_0) w_-^{\pi}[f],\qquad f\in h^{4+\alpha}(\s),
 \end{equation*}
    where $( w_+^{\pi}[f],w_-^{\pi}[f])$  denotes the solution to the problem
 \begin{equation*}
\left\{
\begin{array}{rllllll}
\A^\pi_0(f_0,h_0) w_+^{\pi}[f]\!\!&=&\!\!0 &\text{in $ \Omega_+$},\\
\A^\pi_0(f_0) w_-^{\pi}[f]\!\!&=&\!\!0&\text{in $ \Omega_-$},\\
 \B(f_0,h_0)w_+^{\pi}[f]-\B(f_0)w_-^{\pi}[f]\!\!&=&\!\!0&\text{on $ \Gamma_{0}$},\\
w_+^{\pi}[f]-w_-^{\pi}[f]\!\!&=&\!\!\gamma_f(1+f_0'^2)^{-3/2} f'' &\text{on $\Gamma_0$},\\
w_+^{\pi}[f]\!\!&=&\!\!0&\text{on $ \Gamma_1$},\\
 w_-^{\pi}[f]\!\!&=&\!\!0&\text{on $\G_{-1}$}.
\end{array}
\right.
\end{equation*}
Here, $\A^\pi_0(f_0)$ and $\A^\pi_0(f_0,h_0)$ are the operators defined by \eqref{eqs} (with $(f_*,h_*)$ replaced by $(f_0,h_0)$).
The same arguments as in the proof of Lemma \ref{L:LOP1} show that, given $\e\in(0,1)$, there exists $K_1=K_1(\e)>0$ such that  
\begin{align}
&\|\p_{f}\Phi_1(0,(f_0,h_0))[f]- \p_{f}\Phi_1^\pi[f]\|_{1+\alpha}\leq \e\|f\|_{4+\alpha}+K_1\|f\|_{1+\alpha}\qquad\text{for all $f\in h^{4+\alpha}(\s)$}.\label{Claim4K}
\end{align}

\begin{thm}\label{T1K} Let   $\mu>0$ and  $\alpha'\in(0,\alpha)$ be given.
Then there exist an integer $p\geq 3$, a $p$-partition of unity $\{\Pi_j^p\}_{1\leq j\leq 2^{p+1}}$, and a constant $K_2=K_2(p)$, and  for each $1\leq j\leq 2^{p+1}$ there 
are bounded operators $\bA_{j}\in\kL\big(h^{4+\alpha}(\s), h^{1+\alpha}(\s)\big)$
 such that 
 \begin{equation*} 
  \|\Pi_j^p\p_f\Phi^\pi_{1}[f]-\bA_{j}[\Pi^p_j f]\|_{1+\alpha}\leq \mu \|\Pi_j^pf\|_{4+\alpha}+K_2\|f\|_{4+\alpha'}
 \end{equation*}
 for all $f\in h^{4+\alpha}(\s)$.
 The operators $\bA_{j}$ are defined  by the formula
  \begin{equation*} 
  \bA_{j}[f]:=  - \frac{k}{\mu_-}\Big(\frac{1+f_0'^2}{ f_0-d}\Big|_{x_j^p}\tr_0 \p_yw_-^{\pi,j}[f]-f'_0(x_j^p)\tr_0 \p_xw_-^{\pi,j}[f]\Big),
 \end{equation*}
 where 
 $( w_+^{\pi,j}[f], w_-^{\pi,j}[f])$ denotes the solution to the problem
 \begin{equation*} 
\left\{
\begin{array}{rllllll}
\A^\pi_{0,j}(f_0,h_0) w_+^{\pi,j}[f]\!\!&=&\!\!0&\text{in $ \Omega_+$},\\
\A^\pi_{0,j}(f_0) w_-^{\pi,j}[f]\!\!&=&\!\!0&\text{in $ \Omega_-$},\\
 \B_j(f_0,h_0)w_+^{\pi,j}[f]-\B_j(f_0)w_-^{\pi,j}[f]\!\!&=&\!\!0&\text{on $ \Gamma_{0}$},\\
w_+^{\pi,j}[f]-w_-^{\pi,j}[f]\!\!&=&\!\! V_f f''&\text{on $\Gamma_0$},\\
w_+^{\pi,j}[f]\!\!&=&\!\!0&\text{on $ \Gamma_1$},\\
 w_-^{\pi,j}[f]\!\!&=&\!\!0&\text{on $\G_{-1}$}.
\end{array}
\right.
\end{equation*}
The operators $\A^\pi_{0,j}(f_0,h_0),$ $ \A^\pi_{0,j}(f_0),$ $ \B_j(f_0,h_0),$ and $\B_j(f_0)$ are the same as in Theorem \ref{T1} (with $(f_*,h_*)$ replaced by $(f_0,h_0)$) and
$V_f:=\gamma_f (1+f_0'^2(x_j^p))^{-3/2}.$
\end{thm}
\begin{proof}
 The proof is similar to that of Theorem \ref{T1} and is therefore omitted.
\end{proof}

From the computations in Section \ref{Sec:5} one infers  that, given a $p$-partition of unity $\{\Pi_j^p\}_{1\leq j\leq 2^{p+1}}$ and $1\leq j\leq 2^{p+1}$, the operator  $\bA_{j}$ is  a Fourier multiplier of the form 
\begin{align*} 
 \bA_{j}\Big[\sum_{m\in\Z} f_me^{imx}\Big]=&\sum_{m\in\Z} \wt\lambda_m f_me^{imx},
 \end{align*}
the (real) symbol $(\wt\lambda_m)_m$  being given by
\begin{align*} 
\wt\lambda_m
 =& -  V_f\Big(\frac{\tanh(D_+m)}{\beta_2^+D_+m}+\frac{\tanh(D_-m)}{\beta_2^-D_-m}\Big)^{-1} m^2 ,
 \end{align*}
cf. \eqref{FMFR} and \eqref{FMFI}. 
Here, $\beta_2^\pm$ and $D_\pm$ are the positive constants defined in Section \ref{Sec:5} (with $(f_*,h_*)$ replaced by $(f_0,h_0)$).

\begin{lemma}\label{L:2K}
Given  $(f_0,h_0)\in\V\cap (h^{4+\alpha}(\s))^2$  and $b_0\in h^{2+\alpha}(\s)$, there exist  constants  $\kappa_1\geq 1$ and $\omega_1>0$ depending only on $(f_0,h_0,b_0)$  
such that for any $p$-partition of unity $\{\Pi_j^p\}_{1\leq j\leq 2^{p+1}}$ with $p\geq 3,$ 
the operators  $\bA_{j}$,     $1\leq j\leq 2^{p+1}$,    satisfy
 \begin{align*} 
&\lambda-\bA_{ j}\in{\rm Isom}(h^{4+\alpha}(\s),h^{1+\alpha}(\s)),\\[1ex]
& \kappa_1\|(\lambda-\bA_{ j})[f]\|_{1+\alpha}\geq  |\lambda|\cdot\|f\|_{1+\alpha}+\|f\|_{4+\alpha}
\end{align*}
for all $f\in h^{4+\alpha}(\s)$ and $\lambda\in\C$ with $\re \lambda\geq \omega_1$.
\end{lemma}
\begin{proof}
This follows analogously as in the  proof of Lemma \ref{L:2}.
\end{proof}

We are now in the position to prove the generator property for $\p_{f}\Phi_1(0,(f_0,h_0)).$
\begin{thm}\label{T:c1K}
 If  $(f_0,h_0)\in\V\cap (h^{4+\alpha}(\s))^2$  and $b_0\in h^{2+\alpha}(\s)$, then
 \begin{align*}
 -\p_{f}\Phi_1(0,(f_0,h_0))\in\kH\big(h^{4+\alpha}(\s), h^{1+\alpha}(\s)\big).
\end{align*}
\end{thm}
\begin{proof}
Because of  \eqref{Claim4K} and   \cite[Theorem I.1.3.1 (ii)]{Am95}, we only need to prove that  $-\p_f\Phi^\pi_1$  is an analytic generator.
 The arguments in the proof of Theorem \ref{T:c1} together with  Remark \ref{R:1}, Theorem \ref{T1K}, and Lemma \ref{L:2K}  show there exist constants $\wt\kappa_1>0$  and  $\wt \omega_1>1$ such that 
\begin{align}\label{FE1K}
 \wt\kappa_1\|(\lambda-\p_f\Phi^\pi_1)[f]\|_{1+\alpha}\geq&   |\lambda|\cdot\|f\|_{1+\alpha}+\|f\|_{4+\alpha} 
\end{align}
for all $\lambda\in\C$ with $\re \lambda \geq\wt \omega_1$ and  $f\in h^{4+\alpha}(\s)$.
As $(\lambda-\p_f\Phi^\pi_{1})$ is one-to-one  for $\re\lambda\geq\wt \omega_1$ by \eqref{FE1K}, we are left to show that  
   $(\lambda-\p_f\Phi^\pi_{1})$ is a Fredholm operator of index zero for $\lambda\geq\wt \omega_1$, see \cite[Remark I.1.2.1]{Am95}.

   We can decompose $\p_f\Phi^\pi_{1}$ as 
$\p_f\Phi^\pi_{1}=\Psi_1\circ\Psi_2$, 
where $\Psi_2:h^{4+\alpha}(\s)\to h^{2+\alpha}(\s)$  is defined by
\begin{equation*}
\Psi_2f:= -\cfrac{\gamma_f }{(1+f_0'^2)^{3/2}}f'',
\end{equation*}
and $\Psi_1:h^{2+\alpha}(\s)\to h^{1+\alpha}(\s)$ is the operator obtained from $\p_f\Phi^\pi_{1,0}[f]$ when choosing $\Delta_\rho=1$, cf. \eqref{HO1} and \eqref{HO1a}.
A simple consequence of Proposition  \ref{L:DN1} is that $\Psi_1$ is a Fredholm operator of index zero, the same Fredholm property being valid also for $\Psi_2.$
By a classical result \cite[Theorem 13.1]{TL80} we then know that also $\p_f\Phi^\pi_{1}=\Psi_1\circ\Psi_2$  is also a Fredholm operator of index zero.
This property allows us to conclude that $(\lambda-\p_f\Phi^\pi_{1})$ is bijective for all $\lambda\geq\wt \omega_1$, and to finish the proof.
 \end{proof}

\begin{rem}\label{R:1K}
  Let $(f_*,h_*) \in\V$  be given and $a\in h^{2+\alpha}(\s)$ be a negative function. Moreover, let $\bA$ be  the operator defined by   \eqref{HO1}.
 Then,  
 $${-\bA\circ \partial_x^2}\in\kH\big(h^{4+\alpha}(\s), h^{1+\alpha}(\s)\big).$$
\end{rem}\bigskip

\noindent{\bf The second diagonal entry.} 
We define the operator $\p_{h}\Phi_2^\pi\in\kL\big(h^{4+\alpha}(\s), h^{1+\alpha}(\s)\big)$  by the formula
 \begin{equation*} 
 \p_{h}\Phi_2^\pi[h]:= -\B_1(f_0, h_0) W_+^{\pi}[h],\qquad h\in h^{4+\alpha}(\s),
 \end{equation*}
 where  $ W_+^{\pi}[h]$  denotes the solution to the problem
 \begin{equation*} 
\left\{
\begin{array}{rllllll}
\A^\pi_1(f_0,h_0) W_+^{\pi}[h]\!\!&=&\!\!0&\text{in $\Omega_+$},\\
W_+^{\pi}[h]\!\!&=&\!\!0&\text{on $ \Gamma_0$},\\
W_+^{\pi}[h]\!\!&=&\!\!-\gamma_h (1+h_0'^2)^{-3/2}h'' &\text{on $ \Gamma_1$},
\end{array}
\right.
\end{equation*}
 and with $ \A^\pi_1(f_0,h_0)$ defined by \eqref{eqs2} (with $(f_*,h_*)$ replaced by $(f_0,h_0)$).
The operator $\p_{h}\Phi_2^\pi$ is the  ``leading order'' part of the Fr\'echet derivative $\p_{h}\Phi_2(0,(f_0,h_0))$ in the sense that, given $\e\in(0,1)$, there exists $K_3=K_3(\e)>0$ such that  
\begin{align}
&\|\p_{h}\Phi_2(0,(f_0,h_0))[h]-\p_{h}\Phi_2^\pi[h]\|_{1+\alpha}\leq \e\|h\|_{4+\alpha}+K_3\|h\|_{1+\alpha}\qquad\text{for all $h\in h^{4+\alpha}(\s)$}.\label{Claim5K}
\end{align}
 The estimate \eqref{Claim5K} is obtained by arguing as in Lemma \ref{L:LOP2}.

\begin{thm}\label{T2K} Let   $\mu>0$ and  $\alpha'\in(0,\alpha)$ be given.
Then there exist an integer $p\geq 3$, a $p$-partition of unity $\{\Pi_j^p\}_{1\leq j\leq 2^{p+1}}$, and a constant $K_4=K_4(p)$, and for each $1\leq j\leq 2^{p+1}$ 
there are bounded operators $\bB_{j}\in\kL\big(h^{4+\alpha}(\s), h^{1+\alpha}(\s)\big)$
 such that 
 \begin{equation*} 
  \|\Pi_j^p\p_h\Phi^\pi_{2}[h]-\bB_{j}[\Pi^p_j h]\|_{1+\alpha}\leq \mu \|\Pi_j^ph\|_{4+\alpha}+K_4\|h\|_{4+\alpha'} 
 \end{equation*}
 for all $h\in h^{4+\alpha}(\s)$. The operators $\bB_{ j}$ are defined  by the formula
  \begin{equation*} 
 \bB_{ j}[h]:= - \frac{k}{\mu_+}\Big(\frac{1+h_0'^2}{h_0-f_0}\Big|_{x_j^p}\tr_1 \p_yW_{+}^{\pi,j}[h]-h'_0(x_j^p)\tr_1 \p_xW_{+}^{\pi,j}[h]\Big),
 \end{equation*}
 where 
 $ W_{+}^{\pi,j}[h]$ denotes the solution to the problem
  \begin{equation*} 
\left\{
\begin{array}{rllllll}
\A^\pi_{1,j}(f_0,h_0) W_{+}^{\pi,j}[h]\!\!&=&\!\!0&\text{in $\Omega_+$},\\
W_{+}^{\pi,j}[h]\!\!&=&\!\!0&\text{on $ \Gamma_0$},\\
W_{+}^{\pi,j}[h]\!\!&=&\!\!-V_hh''&\text{on $ \Gamma_1$},
\end{array}
\right.
\end{equation*}
with $\A^\pi_{1,j}(f_0,h_0)$  as in Theorem \ref{T2} (with $(f_*,h_*)$ replaced by $(f_0,h_0)$) and
$V_h:=\gamma_h (1+h_0'^2(x_j^p))^{-3/2}.$
\end{thm}
\begin{proof}
  The proof is similar to that of Theorem \ref{T2} and is therefore omitted.
\end{proof}

The computations in Section \ref{Sec:6} show that, given a $p$-partition of unity $\{\Pi_j^p\}_{1\leq j\leq 2^{p+1}}$ and $1\leq j\leq 2^{p+1}$, the operator   $\bB_{j}$ is  a Fourier multiplier of the form
\begin{align*} 
 \bB_{j}\Big[\sum_{m\in\Z} h_me^{imx}\Big]=&\sum_{m\in\Z} \wt\varphi_mh_me^{imx},
 \end{align*}
the (real) symbol $(\wt\varphi_m)_m$ being given by
\begin{align*} 
\wt\varphi_m:=& -  \frac{kV_h}{\mu_+}\frac{m^3}{\tanh(Dm)},
 \end{align*}
cf. \eqref{FMHR} and \eqref{FMHI}, where $D$ is the positive constant  defined in Section \ref{Sec:6}  (with $(f_*,h_*)$ replaced by $(f_0,h_0)$).

\begin{lemma}\label{L:3K}
Given  $(f_0,h_0)\in\V\cap (h^{4+\alpha}(\s))^2$  and $b_0\in h^{2+\alpha}(\s)$, there exist  constants  $\kappa_2\geq 1$ and $\omega_2>0$ depending only on $(f_0,h_0,b_0)$  
such that for any $p$-partition of unity $\{\Pi_j^p\}_{1\leq j\leq 2^{p+1}}$ with $p\geq 3,$ 
the operators  $\bB_{j}$,     $1\leq j\leq 2^{p+1}$,      satisfy
 \begin{align*} 
&\lambda-\bB_{ j}\in{\rm Isom}(h^{4+\alpha}(\s),h^{1+\alpha}(\s)),\\[1ex]
& \kappa_2\|(\lambda-\bB_{ j})[h]\|_{1+\alpha}\geq  |\lambda|\cdot\|h\|_{1+\alpha}+\|h\|_{4+\alpha}
\end{align*}
for all $h\in h^{4+\alpha}(\s)$ and $\lambda\in\C$ with $\re \lambda\geq \omega_2$.
\end{lemma}
\begin{proof}
This follows analogously as in the  proof of Lemma \ref{L:2h}.
\end{proof}

We  now  establish the generator property for $\p_{h}\Phi_2(0,(f_0,h_0)).$
\begin{thm}\label{T:c2K}
If  $(f_0,h_0)\in\V\cap (h^{4+\alpha}(\s))^2$  and $b_0\in h^{2+\alpha}(\s)$, then
 \begin{align*}
 -\p_{h}\Phi_2(0,(f_0,h_0))\in\kH\big(h^{4+\alpha}(\s), h^{1+\alpha}(\s)\big).
\end{align*}
\end{thm}
\begin{proof}
The proof follows by combining the arguments used to establish Theorems \ref{T:c2} and \ref{T:c1K}.
 \end{proof}

We conclude this section with the proof of Theorem \ref{MT2}.
\begin{proof}[{\bf Proof of Theorem \ref{MT2}}]
 In view of Theorems \ref{T:c1K}, \ref{T:c2K}, relation \eqref{Claim3K} and \cite[Theorem I.1.6.1 and Remark I.1.6.2]{Am95}),
 we get that 
 $$-\p_{(f,h)}\Phi(0,(f_0,h_0))\in\kH\big((h^{4+\alpha}(\s))^2, (h^{1+\alpha}(\s))^2\big)$$
 for all $(f_0,h_0)\in\V\cap (h^{4+\alpha}(\s))^2$  and $b_0=b(0)\in h^{2+\alpha}(\s).$
 The remaining part of the proof is identical to that of Theorem \ref{MT1}.
\end{proof}

\bigskip

{\bf Acknowledgment.} 
We thank the anonymous referee for carefully reading the manuscript.

\appendix
\section{}\label{App}

We collect here some explicit formulae for operators used in the previous sections.
Given $(f,h)\in\V$, it follows readily from the definitions in Section~\ref{Sec:3} of the operators $\A(f)$, $\A(f,h)$, $\B(f)$, and $\B(f,h)$ that
\begin{align*}
\A(f)=&\p_{xx}-2\frac{(1+y)f'}{f-d}\p_{xy}+\frac{(1+y)^2f'^2+1}{(f-d)^2}\p_{yy}-(1+y)\frac{(f-d)f''- 2f'^2}{(f-d)^2}\p_y,\\
\A(f,h)=&\p_{xx}-2\frac{yh'+(1-y)f'}{ h-f}\p_{xy}+\frac{\big(yh'+(1-y)f'\big)^2+1}{(h-f)^2}\p_{yy}-\Big[\frac{yh''+(1-y)f''}{h-f}\\
&-2\frac{(h'-f')\big(yh'+(1-y)f'\big)}{(h-f)^2}\Big]\p_y,\\
\B(f)=&\frac{k}{\mu_-}\Big(\frac{1+f'^2}{ f-d}\tr_0 \p_y-f'\tr_0 \p_x\Big),\qquad \B(f,h)=\frac{k}{\mu_+}\Big(\frac{1+f'^2}{h- f}\tr_0 \p_y-f'\tr_0 \p_x\Big),\\
\B_1(f,h)=&\frac{k}{\mu_+}\Big(\frac{1+h'^2}{h- f}\tr_1 \p_y-h'\tr_1 \p_x\Big),
\end{align*}
and therefore their (partial) derivatives at a fixed $(f_*,h_*)\in\V$  are given by
\begin{align*}
\p_f\A(f_*)[f]=&2\Big[\frac{(1+y)f_*'f}{(f_*-d)^2}-\frac{(1+y)f'}{f_*-d}\Big]\p_{xy}+2\Big[\frac{(1+y)^2f_*'f'}{(f_*-d)^2}-\frac{(1+y)^2f_*'^2+1}{(f_*-d)^3}f\Big]\p_{yy}\\
&-(1+y)\Big[\frac{(f_*-d)f''+f_*''f- 4f_*'f'}{(f_*-d)^2}-2\frac{(f_*-d)f_*''- 2f_*'^2}{(f_*-d)^3}f\Big]\p_y,\\
\p_f\A(f_*,h_*)[f]=& -2\Big[\frac{ (1-y)f'}{ h_*-f_*}+\frac{yh_*'+(1-y)f_*'}{ (h_*-f_*)^2}f\Big]\p_{xy}+\Big[2\frac{\big(yh_*'+(1-y)f_*'\big)^2+1}{(h_*-f_*)^3}f\\
&+\frac{(1-y)\big(yh_*'+(1-y)f_*'\big)f'}{(h_*-f_*)^2}\Big]\p_{yy}-\Big[\frac{(1-y)f''}{h_*-f_*}+\frac{yh_*''+(1-y)f_*''}{(h_*-f_*)^2}f\\
&-2\frac{ (1-2y)(h_*'-f_*')-f_*'}{(h_*-f_*)^2}f'-4\frac{(h_*'-f_*')\big(yh_*'+(1-y)f_*'\big)}{(h_*-f_*)^3}f\Big]\p_y,
\end{align*}
\begin{align*}
\p_h\A(f_*,h_*)[h]=&2\Big[\frac{yh_*'+(1-y)f_*'}{ (h_*-f_*)^2}h-\frac{yh'}{ h_*-f_*}\Big]\p_{xy}+2\Big[\frac{y\big(yh_*'+(1-y)f_*'\big)}{(h_*-f_*)^2}h'\\
&-\frac{\big(yh_*'+(1-y)f_*'\big)^2+1}{(h_*-f_*)^3}h\Big]\p_{yy}-\Big[\frac{yh''}{h_*-f_*}-\frac{yh_*''+(1-y)f_*''}{(h_*-f_*)^2}h \\
&-2\frac{2yh_*'+(1-2y)f_*'}{(h_*-f_*)^2}h'-4\frac{(h_*'-f_*')\big(yh_*'+(1-y)f_*'\big)}{(h_*-f_*)^3}h\Big]\p_y,
\end{align*}
respectively by
\begin{align*}
\p_f\B(f_*)[f]=&\frac{k}{\mu_-}\Big[\Big(\frac{2f_*'f'}{ f_*-d}-\frac{1+f_*'^2}{ (f_*-d)^2}f\Big)\tr_0 \p_y-f' \tr_0 \p_x\Big],\\
\p_f\B(f_*,h_*)[f]=&\frac{k}{\mu_+}\Big[\Big(\frac{2f_*'f'}{h_*- f_*} +\frac{1+f_*'^2}{(h_*- f_*)^2}f\Big)\tr_0\p_y-f'\tr_0 \p_x\Big],
\end{align*}
\begin{align*}
\p_h\B(f_*,h_*)[h]=&-\frac{k}{\mu_+} \frac{1+f_*'^2}{(h_*- f_*)^2}h\tr_0 \p_y ,\\
\p_h\B_1(f_*,h_*)[h]=&\frac{k}{\mu_+}\Big[\Big(\frac{2h_*'}{h_*- f_*}h' -\frac{1+h_*'^2}{(h_*- f_*)^2}h\Big)\tr_1 \p_y-h'\tr_1\p_x\Big]
\end{align*}
for $(f,h)\in \big(h^{2+\alpha}(\s)\big)^2.$

\end{document}